\documentclass[12pt, a4paper, reqno]{amsart}

\allowdisplaybreaks

\usepackage{amsmath, amssymb, amsthm, mathtools, kotex}
\usepackage[T1]{fontenc}
\usepackage[utf8]{inputenc}
\usepackage[all]{xy}
\usepackage[
  margin=1.5cm,
  includefoot,
  footskip=30pt,
  ]{geometry}
  \usepackage{xcolor}
\usepackage[colorlinks=true,citecolor=cyan,linkcolor=magenta, urlcolor=violet, filecolor=green, backref=page]{hyperref}
\usepackage{hyperref}
\usepackage{indentfirst}

\newcommand{\cD}{\mathcal{D}}
\newcommand{\cE}{\mathcal{E}}
\newcommand{\cF}{\mathcal{F}}

\newcommand{\cM}{\mathcal{M}}
\newcommand{\cO}{\mathcal{O}}
\newcommand{\cP}{\mathcal{P}}

\newcommand{\cX}{\mathcal{X}}
\newcommand{\cY}{\mathcal{Y}}

\newcommand{\bA}{\mathbb{A}}
\newcommand{\bC}{\mathbb{C}}
\newcommand{\bE}{\mathbb{E}}
\newcommand{\bF}{\mathbb{F}}
\newcommand{\bG}{\mathbb{G}}
\newcommand{\bP}{\mathbb{P}}
\newcommand{\bQ}{\mathbb{Q}}
\newcommand{\bR}{\mathbb{R}}

\newcommand{\bZ}{\mathbb{Z}}

\newcommand{\fa}{\mathfrak{a}}
\newcommand{\fb}{\mathfrak{b}}
\newcommand{\fc}{\mathfrak{c}}
\newcommand{\fd}{\mathfrak{d}}

\newcommand{\fm}{\mathfrak{m}}

\newcommand{\fp}{\mathfrak{p}}
\newcommand{\fq}{\mathfrak{q}}
\newcommand{\ft}{\mathfrak{t}}

\newcommand{\fI}{\mathfrak{I}}
\newcommand{\fD}{\mathfrak{D}}

\newcommand{\fS}{\mathfrak{S}}

\renewcommand{\Re}{\operatorname{Re}}

\DeclareMathOperator{\Aut}{Aut}

\DeclareMathOperator{\Gal}{Gal}
\DeclareMathOperator{\GL}{GL}

\DeclareMathOperator{\SL}{SL}
\DeclareMathOperator{\Spec}{Spec}
\DeclareMathOperator{\Tr}{Tr}

\DeclareMathOperator{\aff}{aff}

\DeclareMathOperator{\dual}{dual}
\DeclareMathOperator{\id}{id}
\DeclareMathOperator{\im}{im}

\DeclareMathOperator{\hyp}{hyp}
\DeclareMathOperator{\mult}{mult}
\DeclareMathOperator{\ord}{ord}
\DeclareMathOperator{\rk}{rank}

\DeclareMathOperator{\trace}{trace}
\DeclareMathOperator{\wt}{wt}

\theoremstyle{plain}
\newtheorem{theorem}{Theorem}[section]
\newtheorem{lemma}[theorem]{Lemma}
\newtheorem{proposition}[theorem]{Proposition}
\newtheorem{corollary}[theorem]{Corollary}

\theoremstyle{definition} 
\newtheorem{definition}{Definition}
\newtheorem{conjecture}{Conjecture}

\newtheorem{remark}{Remark}
\newtheorem*{notation}{Notation}

\theoremstyle{remark} 
\newtheorem*{example}{Example}

\DeclareFontFamily{U}{wncy}{}
\DeclareFontShape{U}{wncy}{m}{n}{<->wncyr10}{}
\DeclareSymbolFont{mcy}{U}{wncy}{m}{n}
\DeclareMathSymbol{\Sha}{\mathord}{mcy}{"58}

\newcommand{\lp}{\left(}
\newcommand{\rp}{\right)}
\newcommand{\la}{\langle}
\newcommand{\ra}{\rangle}
\newcommand{\lara}[1]{\la #1 \ra}
\newcommand{\lbrb}[1]{\lp #1 \rp}
\newcommand{\lcrc}[1]{\left\{ #1 \right\}}

\title{The average analytic rank of elliptic curves with prescribed level structure}

\author{Peter J. Cho}
\address{Department of Mathematical Sciences, Ulsan National Institute of Science and Technology, UNIST-gil 50, Ulsan 44919, Korea}
\email{petercho@unist.ac.kr}

\author{Keunyoung Jeong}
\address{Department of Mathematics Education, Chonnam National University, 77, Yongbong-ro, Buk-gu, Gwangju 61186, Korea}
\email{keunyoung@jnu.ac.kr}

\author{Junyeong Park}
\address{Department of Mathematics Education, Chonnam National University, 77, Yongbong-ro, Buk-gu, Gwangju 61186, Korea}
\email{junyeongp@gmail.com}

\numberwithin{equation}{section}
\subjclass{11G05, 11M26 (primary), 11F72, 14D23 (secondary)}
\begin{document}

\begin{abstract}
Assuming the Hasse--Weil conjecture and the generalized Riemann hypothesis for the $L$-functions of elliptic curves, we establish an upper bound of the average analytic rank of elliptic curves over a number field with a level structure such that the corresponding compactified moduli stack is representable by the projective line.
\end{abstract}

\maketitle

\section{Introduction}

Research on the distribution of analytic ranks of elliptic curves has been actively conducted under the generalized Riemann hypothesis for the $L$-functions of elliptic curves.
Under this assumption, Brumer \cite{Bru} first showed that the average analytic rank of the elliptic curves over the rationals is less than $2.3$. 
Heath-Brown \cite{Hea} refined it by $2$. Later, Young \cite{You} improved it further by $\frac{25}{14}$.
In \cite{CJ1}, Cho and Jeong studied the distribution of analytic ranks of elliptic curves over the rationals by estimating the $n$-level density with multiplicity, and in \cite{CJ2} they gave an explicit upper bound of the average analytic rank of elliptic curves with a prescribed torsion subgroup under certain moment conditions.\footnote{Some errors in \cite{CJ2} were addressed in its corrigendum \cite{CJ3}, where the moment conditions were removed.} Recently, Philips \cite{Phi2} showed that the average analytic rank of elliptic curves over a number field $K$ of degree $d$ is at most $\frac{9}{2}d + \frac{1}{2}$.
The main result of this paper is an explicit upper bound for the average analytic rank of elliptic curves over a number field with a prescribed level structure.

To describe the setting more precisely, we introduce several definitions and employ the following notation.
\begin{notation} Throughout this article, the following notation will be used without further comment.
\begin{itemize}
    \item Given a stack $\mathcal{M}$ in groupoids over the category of schemes, for each ring $R$ we denote by $\underline{\mathcal{M}}(R)$ the \emph{groupoid of $R$-points}. We denote by $\mathcal{M}(R)$ the \emph{set of isomorphism classes} of $\underline{\mathcal{M}}(R)$.
    \item A stack is \emph{representable} if it is representable by a scheme. Then the representability of ($1$-)morphisms of stacks is correspondingly defined.
    \item Given a number field $K$ and a nonzero prime $\mathfrak{p}\subseteq\mathcal{O}_K$, we denote
    \begin{itemize}
        \item $K_\mathfrak{p}$ the $\mathfrak{p}$-adic completion of $K$,
        \item $\mathcal{O}_{K,\mathfrak{p}}\subseteq K_\mathfrak{p}$ the valuation ring, and
        \item $\kappa(\mathfrak{p})$ the residue field at $\mathfrak{p}$.
    \end{itemize}
    \item For a positive integer $N$, denote $\zeta_N$ any primitive $N$-th root of unity.
\end{itemize} 
\end{notation}
Let $\Gamma$ be a congruence subgroup, and let $\cX_{\Gamma}$ be the moduli stack of elliptic curves with level structure $\Gamma$.
There is an isomorphism between $\cX_{\SL_2(\bZ)}$ and $\cP(4, 6)$, a weighted projective stack where the height is canonically given,  over $\mathbb{Z}[1/6]$. We define the naive height of the elliptic curves over a number field as the pullback of the height on $\cP(4, 6)$.
Then $\cE_{K, \Gamma}(X)$, the set of $K$-isomorphism classes of elliptic curves over $K$ with naive height less than $X$ and level structure $\Gamma$, is well-defined.
Since we want to count the number of elliptic curves that admit a prescribed level structure,
not the number of pairs of an elliptic curve and a level structure,
$\cE_{K, \Gamma}(X)$ should be identified with a subset of the image of the forgetful functor $\cX_{\Gamma}(K) \to \cX_{\SL_2(\bZ)}(K)$ with a condition on the naive height.

From now on, we assume that the $L$-functions $L(E/K,s)$ of $E$ over $K$ admit analytic continuation. We define the analytic rank of $E$ over $K$ to be the vanishing order of $L(E/K,s)$ at the central point, and denote it by $r_E$. The average analytic rank of elliptic curves over $K$ is given by
\begin{align*}
    \lim_{X \to \infty} \frac{1}{|\cE_{K, \Gamma}(X)|} \sum_{E \in \cE_{K, \Gamma}(X)} r_E,
\end{align*}
even though the existence of the limit is not known yet.
However, Young’s result \cite{You} can be reformulated as
\begin{align*}
    \limsup_{X \to \infty} \frac{1}{|\cE_{\bQ, \SL_2(\bZ)}(X)|} \sum_{E \in \cE_{\bQ, \SL_2(\bZ)}(X)} r_E < \frac{25}{14}.
\end{align*}
Similarly, Theorem 1 of \cite{CJ2} can be rewritten as, for example, 
\begin{align} \label{eqn:resultCJ2}
    \limsup_{X \to \infty} \frac{1}{|\cE_{\bQ, \Gamma_1(N)}(X)|} \sum_{E \in \cE_{\bQ, \Gamma_1(N)}(X)} r_E < 30.5
\end{align}
for $N = 5, 6$.

In addition, we assume that the $L$-functions $L(E/K,s)$ satisfy a standard functional equation and the generalized Riemann hypothesis (GRH).
The GRH is required to relate the average analytic rank to the average value of the test function evaluated at low-lying zeros over these $L$-functions, while holomorphicity and the functional equation are necessary for applying Weil's explicit formula. 
Although assuming such deep conjectures may seem unsatisfactory, we note the following points.
Previous results \cite{Bru, Hea, You, Phi2} also rely on these same conjectures. Moreover, over $\bQ$, the Hasse–Weil conjecture is already established through the modularity theorem for elliptic curves over $\bQ$.
Furthermore, there exist many number fields $K$ for which every elliptic curve is known to be modular (and hence its
$L$-function has analytic continuation and satisfies a functional equation).
On the other hand, we can find many number fields $K$ where an arbitrary elliptic curve is modular (hence, its $L$-function has the analytic continuation and the functional equation).
For instance, elliptic curves over real quadratic fields \cite{FLS}, totally real cubic fields \cite{DNS}, totally real quartic fields not containing $\sqrt{5}$ \cite{Box}, totally real field of degree $5$ with finitely many exceptions \cite{IIY}, and infinitely many imaginary quadratic fields \cite{CN} are known to be modular.

\smallskip

For positive integers $M, N$, we define $\Gamma_1(M, N) \vcentcolon =\Gamma(M) \cap \Gamma_1(MN)$.
The first main theorem of this paper is stated below.
\begin{theorem} \label{mainthm:rankbound}
    Assume the Hasse--Weil conjecture and the GRH for the $L$-functions of elliptic curves $E$ over a number field $K$.
    Let $\Gamma = \Gamma_1(M, N)$ be a congruence subgroup for which $\cX_{\Gamma} \cong \bP^1$ over $\Spec \bZ[1/6MN]$.
    Then there is an explicit constant $c(K, \Gamma)$ satisfying
    \begin{align*}
        \limsup_{X \to \infty} \frac{1}{|\cE_{K, \Gamma}(X)|} \sum_{E \in \cE_{K, \Gamma}(X)} r_E < c(K, \Gamma) + \frac{1}{2}.
    \end{align*}
\end{theorem}
For a number field $K$ of degree $d$, the constant $c(K, \Gamma)$ is given as below. 
\begin{align*}
    \begin{array}{l|ll}
    c(K, \Gamma) & \Gamma \\ \hline 
       18d   & \Gamma_1(5), \Gamma_1(6), \Gamma(2) \cap \Gamma_1(4), \Gamma(3) \\
       36d  & \Gamma_1(7), \Gamma_1(8), \Gamma(2) \cap \Gamma_1(6), \Gamma(4) \\
       54d & \Gamma_1(9), \Gamma_1(10), \Gamma(3) \cap \Gamma_1(6) \\
       72d & \Gamma_1(12), \Gamma(2) \cap \Gamma_1(8) \\
       90d & \Gamma(5)
    \end{array}
\end{align*}

Theorem \ref{mainthm:rankbound} gives the bound $18.5$ when $K = \bQ$ and $\Gamma = \Gamma_1(5), \Gamma_1(6)$, that is better than our previous results (\ref{eqn:resultCJ2}).
See Remark \ref{rmk:estpowcompare} for the reason. 

By studying the average size of Selmer groups of elliptic curves, 
Bhargava and Shankar~\cite{BS, BS15} obtained an \emph{unconditional} upper bound of~$1$ for the average \emph{algebraic} rank of elliptic curves over~$\bQ$. 
In his thesis~\cite{Sha}, Shankar also established an upper bound of~$1.5$ for the average algebraic rank of elliptic curves over a number field with class number~$1$. 
Bhargava and Ho~\cite{BH} obtained an upper bound of~$\tfrac{7}{6}$ (resp.\ $\tfrac{3}{2}$) for the algebraic rank of elliptic curves over~$\bQ$ with $2$-torsion (resp.\ $3$-torsion) using the Bhargava--Shankar method. 
However, no results are currently known for more complicated level structures beyond $\Gamma_1(2)$ and $\Gamma_1(3)$ (see also Laga's thesis~\cite[I.1]{Lag}) or for number fields with class number~$>1$.

We follow the strategy of \cite{CJ2} to prove Theorem \ref{mainthm:rankbound}.
Let $\cE$ be a certain set of isomorphism classes of elliptic curves.
Counting the number of elements in $\cE$ satisfying a local condition naturally leads to a kind of weighted Hurwitz class number, where the weights reflect arithmetic data from $\cE$.
Suppose we further show that the moments of traces of Frobenius, weighted by these Hurwitz class numbers, admit good asymptotic bounds. In that case, we obtain an upper bound for the average analytic rank of $\cE$.
Hence, we reduce the average analytic rank problems to the problem of counting elements in $\cE$ with a ``local condition''.
Here, a ``local condition'' on an elliptic curve $E$ at a prime $\fp$ of $K$ means a condition on the mod $\fp$ reduction $\overline{E}$ of $E$.
For instance, the condition of good (resp. multiplicative, additive) reduction means that the smooth locus $\overline{E}^\mathrm{sm}\subseteq\overline{E}$ is an elliptic curve (resp. the multiplicative group $\bG_m$, the additive group $\bG_a$). 

One of the main contributions of this paper is the count of elliptic curves which admit a prescribed level structure and a local condition at a prime $\fp$ with a power-saving error term which is independent of $\fp$.
There has been extensive work on counting elliptic curves with prescribed level structures, including \cite{HS, PPV, PS, CKV, BN, BS24, MV}, as well as several studies incorporating local conditions, either on elliptic curves themselves \cite{CJ1} or on Weierstrass equations \cite{CS}.
Phillips \cite{Phi1} studied the case where $\cX_{\Gamma}$ is a weighted projective line with a certain technical condition, and suggested an asymptotic formula for the number of elements of $\cE_{\Gamma}(X)$ satisfying a local condition at $\fp$, and infinitely many local conditions.
However, the error term in his result depends on $\fp$.
In contrast, when $\cX_{\Gamma}$ is representable, we establish a count with an error term independent of $\fp$, a feature that is crucial for deriving an upper bound on the average analytic rank.
In this setting, the problem reduces to counting points in weighted projective space, a subject investigated in earlier works such as \cite{Den, Dar, BN}, and more recently in \cite{BM}.
For further details and remarks on previous approaches, see the beginning of section~\ref{sec: phi refine}.


In Cho and Jeong's previous works \cite{CJ1, CJ2}, a local condition was defined simply as a condition on the Weierstrass equation modulo $\fp$.
By adopting the modular stack viewpoint already introduced by Phillips \cite{Phi1, Phi2}, such a condition is now described as a subset of $\cX(\kappa(\fp))$, where $\kappa(\fp)$ is the residue field modulo $\fp$.
This perspective has several advantages.
First, it allows us to overcome the difficulty discussed in \cite[Remark 2]{CJ2}, namely that when one imposes a local condition on a Weierstrass equation, the number of fibers of $J_1 \in \bF_p^2$ and $J_2 \in \bF_p^2$ may differ even if the corresponding elliptic curves are isomorphic (see Remark \ref{rmk:Wvsphi}).
Second, it provides a more intuitive framework and enables generalizations of certain results in \cite{CJ2} (see Remarks \ref{rem:cuspcompareCJ}, \ref{rmk:cusprmk}, and \ref{rmk: HGam and HG}).
As an example, we obtain the following theorem.

\begin{theorem}({Theorem \ref{thm:cuspthm}}) \label{mainthm:cusp}
Let $K$ be a number field, $\Gamma$ a congruence subgroup of genus zero of level $N$ for which $\mathcal{X}_\Gamma \cong \bP^1$ over $\bZ[1/6N]$, $\fp$ a prime of $K$ that does not divide $6N$, and $\kappa(\fp)$ the residue field of $\fp$.
Let $\cX_{\Gamma}^{\mathrm{cusp}}(\kappa(\fp))$ be the set of $\kappa(\fp)$-rational cusps of $\cX_{\Gamma}$.
    Then, the probability that an elliptic curve has multiplicative reduction at $\fp$ is 
    $|\cX_{\Gamma}^{\mathrm{cusp}}(\kappa(\fp))|/|\cX_{\Gamma}(\kappa(\fp))|$.
\end{theorem}
Here, we define the probability by the proportion of elliptic curves counted by naive height.
Another result, which can be easily obtained after understanding the moduli stack heuristic, is the following.
In \cite[Corollary 3.13]{CJ2}, Cho and Jeong gave examples of primes at which the probabilities of having split and non-split multiplicative reduction in the set of elliptic curves with a prescribed torsion subgroup are not equal.
After taking a finite extension of number fields, this phenomenon will disappear at almost all primes if $\cX_{\Gamma}$ is representable.
The precise statement can be found in Corollary \ref{cor:spnonsp}.

Another key step in the proof of Theorem~\ref{mainthm:rankbound} is to establish a bound for the moments of the trace of Frobenius, weighted by a certain variant of the Hurwitz class numbers.
As in \cite{CJ2}, we carry out this estimation using the Eichler--Selberg trace formula of Kaplan and Petrow~\cite{KP2}.
However, we note that \cite{CJ2} contains an error in the estimation of 
the first moment of the trace of Frobenius weighted by a certain class number. This issue has been corrected in corrigendum \cite{CJ3}. In this paper, we follow the approach of \cite{CJ3}, which employs a variant of the prime number theorem for Hecke eigenforms, rather than relying solely on the Deligne bound.


\smallskip

Theorem \ref{mainthm:rankbound} is proved in Section \ref{sec:average}. 
In the proof, we have 
\begin{align*}
    \frac{1}{|\cE_{K, \Gamma}(X)|} \sum_{E \in \cE_{K, \Gamma}(X)} r_E
    &\leq  \frac{12}{\sigma} + \frac{1}{2} + 
    o(1).
\end{align*}
Here, the constant $\sigma$ is a positive constant such that the support of the Fourier transform of a test function is contained in $[-\sigma,\sigma]$.
Roughly, Katz and Sarnak's philosophy says that the same result holds for any test function with no restriction on supports, and the average comes from the terms not related to $\sigma$ (For other examples, see \cite[Conjecture 3.3]{You}).
Therefore, we suggest the following conjecture.
\begin{conjecture} \label{conj:main}
    Let $\Gamma$ be a congruence group such that $\cX_{\Gamma} \cong \bP^1$ over an open subset of $\Spec \bZ[1/6]$.
    The average analytic rank of elliptic curves over a number field which admit a prescribed level structure $\Gamma$ is $\frac{1}{2}$.
\end{conjecture}
It is worth mentioning that there is a heuristic on Selmer ranks that makes the $\frac{1}{2}$-conjecture for the algebraic rank of all elliptic curves. 
For a global field $k$ and a prime integer $p$, the average rank of $p$-Selmer groups of elliptic curves is believed to be $p+1$. 
If we further assume that the parity of ranks of elliptic curves is equidistributed, then we can conclude the average algebraic rank of $E/k$ is $\frac{1}{2}$. 
For the details, we refer \cite[Theorem 1.1, Conjecture 1.2]{PR12}.

Finally, we remark that the representability condition on $\cX_{\Gamma}$ is used to define the mod $\fp$ reduction on $\cX_{\Gamma}(K_{\fp})$, to show that the forgetful functor has finite defect, and to control the number of preimages of the forgetful functor.
Since these problems may not require the full strength of representability, we hope that they can be resolved and that our results can be extended accordingly in the near future.

In section \ref{sec:Prestack}, we define a mod $\fp$ reduction map on the rational points of the compactified moduli stack.
In section \ref{sec:counting}, we count the number of elliptic curves over a number field with a level structure and a local condition, and prove Theorem \ref{mainthm:cusp}, assuming Theorem \ref{prop:phi411}.
In section \ref{sec:classnum}, we define the weighted Hurwitz class numbers and give an asymptotic of the moments of traces of Frobenius automorphism weighted by the weighted Hurwitz class numbers.
In section \ref{sec:average}, we give a proof of Theorem \ref{mainthm:rankbound} under the assumption of Theorem \ref{prop:phi411}.
Finally, we consider the problem of counting points in a weighted projective space and prove Theorem \ref{prop:phi411} in section \ref{sec: phi refine}.

\smallskip

\textbf{Acknowledgement} The authors thank Dohyeong Kim for suggesting Theorem \ref{mainthm:cusp}.
We also thank Yeong-Wook Kwon, Daeyeol Jeon, and Chul-hee Lee for the useful discussion, Tristan Phillips for his kind explanations of our countless questions, and the anonymous referee for useful suggestions and numerous corrections, not only in the present paper but also in our earlier work (notably Corollary \ref{cor: Cor of CJ2}).

\smallskip

\textbf{Funding}
P. J. Cho was supported by the National Research Foundation of Korea (NRF) grant funded by the Korea government (MSIT) (No. RS-2022-NR069491 and No. RS-2025-02262988). 
K. Jeong was supported by the National Research Foundation of Korea (NRF) grant funded by the Korea government (MSIT) (RS-2024-00341372 and RS-2024-00415601). 
J. Park was supported by the Samsung Science and Technology Foundation under Project Number SSTF-BA2001-02, Basic Science Research Program through the
National Research Foundation of Korea (NRF) funded by the Ministry of Education (RS-2024-00449679), and the National Research Foundation of Korea(NRF) grant funded by the Korea government (MSIT) (RS-2024-00415601).

\section{Preliminaries on moduli stacks} \label{sec:Prestack}

\subsection{Cusps} \label{subsec:cusps}
Let $\Gamma\subseteq\mathrm{SL}_2(\mathbb{Z})$ be a congruence subgroup of level $N$, 
and let $\mathcal{Y}_\Gamma$ be the corresponding moduli stack with $\mathcal{X}_\Gamma$ its compactification. For the modular descriptions of $\mathcal{X}_\Gamma$, we will use Drinfeld level structures on (generalized) elliptic curves as in \cite[Chapter 1]{KM85} and \cite[Section 2.3]{Con}. 
For a ring $R$, the members of $\mathcal{X}_\Gamma(R)$ are written as a pair $(E,\alpha)$ of a generalized elliptic curve $E$ and a level structure $\alpha$ on $E$ defined by $\Gamma$, as in \cite{DR73}, \cite{KM85}, and \cite{Con}.
For simplicity, we denote 
\begin{align*}
    \cX \vcentcolon= \cX_{\SL_2(\bZ)}, \qquad \cY \vcentcolon= \cY_{\SL_2(\bZ)}.
\end{align*}
We also denote by $\phi_\Gamma:\mathcal{X}_\Gamma\rightarrow\mathcal{X}$ the morphism forgetting the corresponding level structure.
If $R$ is a ring where $N$ is invertible, then $\mathcal{Y}_{\Gamma,R}$ is smooth over $\Spec R$ by \cite[Th\'eor\`eme IV.3.4]{DR73}. Hence $\mathcal{Y}_{\Gamma,R}$ is normal (cf. \cite[\href{https://stacks.math.columbia.edu/tag/033C}{Tag 033C}]{Stacks} and \cite[\href{https://stacks.math.columbia.edu/tag/04YE}{Tag 04YE}]{Stacks}). Since $\mathcal{X}_{\Gamma,R}$ is the normalization of $\mathcal{X}_R$ in $\mathcal{Y}_{\Gamma,R}$ by \cite[D\'efinition IV.3.2]{DR73}, we have Cartesian squares:
\begin{align*}
    \xymatrix{
    \mathcal{Y}_{\Gamma,R} \ar[d]_-{\phi_{\Gamma, R}} \ar@{^(->}[r] & \mathcal{X}_{\Gamma,R} \ar[d]^-{\phi_{\Gamma, R}} \\
    \mathcal{Y}_R \ar[d]_-j \ar@{^(->}[r] & \mathcal{X}_R \ar[d]^-j \\
    \mathbb{A}^1_R \ar@{^(->}[r] & \mathbb{P}^1_R
    }
\end{align*}
where $j$ here is the ``$j$-invariant'' or, more precisely, the universal map from the moduli stack to the associated coarse moduli scheme (cf. \cite[VI.1]{DR73}). By \cite[Th\'eor\`eme IV.3.4]{DR73},
\begin{align} \label{eqn: def cX cusp}
    \mathcal{X}_{\Gamma,R}^\mathrm{cusp}\vcentcolon=\mathcal{X}_{\Gamma,R}\setminus\mathcal{Y}_{\Gamma,R}=(j\circ\phi_{\Gamma, R})^{-1}(\infty)
\end{align}
is a substack of $\mathcal{X}_{\Gamma,R}$ finite \'etale over $\Spec R$, which we will call the substack of cusps. In what follows, \emph{cusps of $\mathcal{X}_{\Gamma,R}$} means the elements in $\mathcal{X}_{\Gamma,R}^\mathrm{cusp}(T)$ for some $R$-scheme $T$.

\begin{lemma} \label{lem:cuspGam} 
Let $\Gamma$ be a congruence subgroup of level $N$ and let $q$ be a prime power satisfying $(q, N)=1$.
Then, $\phi_{\Gamma,\mathbb{F}_q}^{-1}(\mathcal{X}_{\mathbb{F}_q}^\mathrm{cusp}(\mathbb{F}_q))=\mathcal{X}_{\Gamma,\mathbb{F}_q}^\mathrm{cusp}(\mathbb{F}_q)$.
\end{lemma}
\begin{proof}
   It follows immediately from the above discussion.
\end{proof}

Let $Y_{\Gamma,R}$ and $X_{\Gamma,R}$ be the coarse moduli scheme of $\mathcal{Y}_{\Gamma,R}$ and $\mathcal{X}_{\Gamma,R}$ over $\Spec R$ respectively. By definition, $j\circ\phi_{\Gamma, R}:\mathcal{Y}_{\Gamma,R}\rightarrow\mathbb{A}^1_R$ factors uniquely through the canonical map $\mathcal{Y}_{\Gamma,R}\rightarrow Y_{\Gamma,R}$, and by \cite[Proposition IV.3.10]{DR73}, $X_{\Gamma,R}$ is isomorphic to the normalization of $\mathbb{P}^1_R$ in $Y_{\Gamma,R}$. Consequently, we get a commutative cube:
\begin{align*}
    \xymatrix@!0@R=3.5pc@C=3.5pc{
    & \mathcal{Y}_{\Gamma,R} \ar[dl]_-{\phi_{\Gamma, R}} \ar[dd]|\hole \ar@{^(->}[rr] & & \mathcal{X}_{\Gamma,R} \ar[dl]_-{\phi_{\Gamma, R}} \ar[dd] \\
    \mathcal{Y}_R \ar[dd]_(.3)j \ar@{^(->}[rr] & & \mathcal{X}_R \ar[dd]_(.3)j & \\
    & Y_{\Gamma,R} \ar[dl] \ar@{^(->}[rr]|\hole & & X_{\Gamma,R} \ar[dl] \\
    \mathbb{A}^1_R \ar@{^(->}[rr] & & \mathbb{P}^1_R &
    }
\end{align*}
where the bottom square and the back square are Cartesian as well. Now $(X_\Gamma\setminus Y_\Gamma)_\mathrm{red}$\footnote{The reduced induced closed subscheme (see \cite[\href{https://stacks.math.columbia.edu/tag/01J4}{Tag 01J4}]{Stacks}).} is the cusp of the coarse moduli in the sense of \cite[8.6.3]{KM85}. Note that (cf. \cite[Definition I.8.1]{DR73}) if $\overline{s}$ is a geometric point of $\Spec R$, then the canonical map $\mathcal{X}_{\Gamma,R}\rightarrow X_{\Gamma,R}$ induces a bijection\footnote{Recall that given a stack $\mathcal{M}$ we denote by $\mathcal{M}(\overline{s})$ the \emph{set of isomorphism classes} in the groupoid $\underline{\mathcal{M}}(\overline{s})$.}
\begin{align*}
    \xymatrix{\mathcal{X}_{\Gamma,R}^\mathrm{cusp}(\overline{s})=\mathcal{X}_{\Gamma,R}(\overline{s})\setminus\mathcal{Y}_{\Gamma,R}(\overline{s}) \ar[r]^-\sim & X_{\Gamma,R}(\overline{s})\setminus Y_{\Gamma,R}(\overline{s})}.
\end{align*}

\begin{lemma} \label{lem:cuspbound}
Let $\Gamma$ be a congruence subgroup of level $N$ and let $q$ be a prime power satisfying $(q, N)=1$. Then $|\mathcal{X}_{\Gamma,\mathbb{F}_q}^\mathrm{cusp}(\mathbb{F}_q)|$ is bounded by a constant depending only on $\Gamma$.
\end{lemma}
\begin{proof} Applying \cite[IV.3]{DR73} to
\begin{align*}
    H=(\Gamma\bmod N)\rtimes(\mathbb{Z}/N\mathbb{Z})^\times\subseteq\mathrm{SL}_2(\mathbb{Z}/N\mathbb{Z})\rtimes(\mathbb{Z}/N\mathbb{Z})^\times\cong\mathrm{GL}_2(\mathbb{Z}/N\mathbb{Z}),
\end{align*}
we conclude that $\phi_\Gamma:\mathcal{X}_\Gamma\rightarrow\mathcal{X}$ is finite and representable of degree $[\mathrm{SL}_2(\mathbb{Z}):\Gamma]$ (One may use the Riemann surface description to get the degree). Hence we have
\begin{align*}
    |\mathcal{X}_{\Gamma,\mathbb{F}_q}^\mathrm{cusp}(\mathbb{F}_q)|\leq[\mathrm{SL}_2(\mathbb{Z}):\Gamma]\cdot|\mathcal{X}_{\mathbb{F}_q}^\mathrm{cusp}(\mathbb{F}_q)|\rlap{\ .}
\end{align*}
By \cite[VI.1.4]{DR73}, we have $|\mathcal{X}_{\mathbb{F}_q}^\mathrm{cusp}(\mathbb{F}_q)|=|j^{-1}[1,0]|\leq2$ so we get
\begin{align*}
    |\mathcal{X}_{\Gamma,\mathbb{F}_q}^\mathrm{cusp}(\mathbb{F}_q)|\leq2[\mathrm{SL}_2(\mathbb{Z}):\Gamma]\rlap{\ .}
\end{align*}
This concludes the proof.
\end{proof}

\subsection{Reduction}\label{subsection:reduction}

Let $K$ be a number field with ring of integers $\mathcal{O}_K$. For each nonzero prime ideal $\mathfrak{p}$ of $\mathcal{O}_K$, a possible obstruction to get the ``reduction modulo $\mathfrak{p}$'' on $\mathcal{X}_\Gamma(K)$ is finding an appropriate section of the usual base change map:
\begin{align}\label{obstruction}
    \xymatrix{-\otimes_{\mathcal{O}_{K,\mathfrak{p}}}K_\mathfrak{p}:\mathcal{X}_\Gamma(\mathcal{O}_{K,\mathfrak{p}}) \ar[r] & \mathcal{X}_\Gamma(K_\mathfrak{p}) & C \ar@{|->}[r] & C\otimes_{\mathcal{O}_{K,\mathfrak{p}}}K_\mathfrak{p}.}
\end{align}
\begin{remark}\label{rmk:obstruction}
The map \eqref{obstruction} is injective by the valuative criterion for properness of algebraic stacks \cite[\href{https://stacks.math.columbia.edu/tag/0CLZ}{Tag 0CLZ}]{Stacks} together with the uniqueness part of the valuative criterion of algebraic stacks \cite[\href{https://stacks.math.columbia.edu/tag/0CLG}{Tag 0CLG}]{Stacks}. Consequently, \eqref{obstruction} has at most one section. This uniqueness also implies that a section of \eqref{obstruction} is independent of parametrizations, i.e., of replacing $\mathcal{X}_\Gamma$ by an isomorphic stack.
\end{remark}

To make the story clear, we first consider the case where $\mathcal{X}_\Gamma$ is representable over $\mathbb{Z}[1/N]$ and the characteristic of the residue field $\kappa(\mathfrak{p})$ is relatively prime to $N$ so that $\mathfrak{p}\in\Spec\mathcal{O}_K[1/N]$. In this case, the canonical map $\mathcal{X}_\Gamma\rightarrow X_\Gamma$ to the associated coarse moduli scheme induces a bijection
\begin{align*}
    \xymatrix{\mathcal{X}_\Gamma(R) \ar[r]^-\sim & X_\Gamma(R)}
\end{align*}
for every ring $R$ where $N$ is invertible. By \cite[Proposition IV 3.10]{DR73} and \cite[\href{https://stacks.math.columbia.edu/tag/0CL6}{Tag 0CL6}]{Stacks}, $X_\Gamma$ is proper over $\mathbb{Z}[1/N]$. Hence, by the valuative criterion for properness \cite[\href{https://stacks.math.columbia.edu/tag/0BX5}{Tag 0BX5}]{Stacks}, the commutative solid diagram
\begin{align*}
    \xymatrix{
    \Spec K_\mathfrak{p} \ar[d] \ar[r] & X_{\Gamma,\mathbb{Z}[1/N]} \ar[d] \\
    \Spec\mathcal{O}_{K,\mathfrak{p}} \ar[r] \ar@{-->}[ur] & \Spec\mathbb{Z}[1/N]
    }
\end{align*}
admits a unique dashed arrow making the diagram commutative, i.e., $X_\Gamma(\mathcal{O}_{K,\mathfrak{p}})\rightarrow X_\Gamma(K_\mathfrak{p})$ is bijective. Since we have assumed that $\mathcal{X}_\Gamma$ is representable, \eqref{obstruction} becomes bijective for each nonzero prime ideal $\mathfrak{p}$ in $\mathcal{O}_K[1/N]$. Therefore, we have only one possible choice, the inverse of this base change map. Using this, we may define the modulo $\mathfrak{p}$ map $\psi_\mathfrak{p}$ on $X_\Gamma(K)$ to be the composite:
\begin{align}\label{reduction-representable}
    \begin{aligned}
    \xymatrixcolsep{4pc}\xymatrix{
    & X_\Gamma(\mathcal{O}_{K,\mathfrak{p}}) \ar[d]^-\wr_-{-\otimes_{\mathcal{O}_{K,\mathfrak{p}}}K_\mathfrak{p}} \ar[r]^-{-\otimes_{\mathcal{O}_{K,\mathfrak{p}}}\kappa(\mathfrak{p})} & X_\Gamma(\kappa(\mathfrak{p}))\rlap{\ .} \\
    X_\Gamma(K) \ar[r]_-{-\otimes_KK_\mathfrak{p}} & X_\Gamma(K_\mathfrak{p}) & 
    }
    \end{aligned}
\end{align}
\begin{remark} \label{reduction-P1}
Later, we will work with representable $\cX_\Gamma$ of genus $0$. We have an isomorphism $X_\Gamma\cong\mathbb{P}^1$ in this case. Then the corresponding map $\psi_\mathfrak{p}:\mathbb{P}^1(K_\mathfrak{p})\rightarrow\mathbb{P}^1(\kappa(\mathfrak{p}))$ is described as follows: Given $x\in\mathbb{P}^1(K_\mathfrak{p})$, we choose a representative $x=[x_0,x_1]$ such that $x_i\in\mathcal{O}_{K,\mathfrak{p}}$ with $\min\{\mathrm{val}_\mathfrak{p}(x_0),\mathrm{val}_\mathfrak{p}(x_1)\}$ is minimal, and then we have $\psi_\mathfrak{p}(x)=[x_0\bmod\mathfrak{p},x_1\bmod\mathfrak{p}]$.
\end{remark}

\begin{proposition} \label{representability}
    The stack $\mathcal{X}_\Gamma$ is representable over $\Spec\mathbb{Z}[1/N]$ for each of the following cases.
    \begin{enumerate}
        \item $\Gamma=\Gamma(N)$ with $N\geq3$.
        \item $\Gamma=\Gamma_1(N)$ with $N\geq5$.
    \end{enumerate}
\end{proposition}
\begin{proof}
    (1) is \cite[Corollaire IV.2.9]{DR73}. (2) is \cite[Proposition 2.1]{Gro90}.
\end{proof}

\begin{proposition} \label{representability-(M,N)}
Let $\Gamma\subseteq\Sigma\subseteq\mathrm{SL}_2(\mathbb{Z})$ be congruence subgroups. If $\mathcal{X}_\Sigma$ is representable over some base scheme, then so is $\mathcal{X}_\Gamma$ over the same base scheme. Consequently, the following statements hold.
\begin{enumerate}
    \item Denote $\Gamma_1(M,N)\vcentcolon=\Gamma(M)\cap\Gamma_1(MN)$.
    \begin{itemize}
        \item $\mathcal{X}_{\Gamma_1(M,N)}$ for $MN\geq5$ is representable over $\Spec\mathbb{Z}[1/MN]$.
        \item $\mathcal{X}_{\Gamma_1(M,N)}$ for $M\geq3$ is representable over $\Spec\mathbb{Z}[1/M]$.
    \end{itemize}
    \item $\mathcal{X}_{\Gamma_0(16)\cap\Gamma_1(8)}$ is representable over $\Spec\mathbb{Z}[1/2]$.
    \item $\mathcal{X}_{\Gamma_0(25)\cap\Gamma_1(5)}$ is representable over $\Spec\mathbb{Z}[1/5]$.
\end{enumerate}
\end{proposition}
\begin{proof} By assumption, we have the following tower of moduli stacks (over some fixed base scheme):
\begin{align*}
    \xymatrix{
    \mathcal{X}_\Gamma \ar[d] \\
    \mathcal{X}_\Sigma \ar[d] \\
    \mathcal{X}
    }
\end{align*}
By construction, $\mathcal{X}_\Gamma\rightarrow\mathcal{X}$ and $\mathcal{X}_\Sigma\rightarrow\mathcal{X}$ are representable (cf. \cite[IV.3]{DR73}). Given a map from a scheme $t:T\rightarrow\mathcal{X}_\Sigma$, we may take $2$-fibered products:
\begin{align*}
    \xymatrixcolsep{4pc}\xymatrix{
    t^\ast\mathcal{X}_\Gamma \ar[d] \ar[r] & (f\circ t)^\ast\mathcal{X}_\Gamma \ar[d] \ar[r] & \mathcal{X}_\Gamma \ar[d] \\
    T \ar@/_1pc/[dr]_-{\mathrm{Id}_T} \ar[r]^-{(t,\mathrm{Id}_T)} & (f\circ t)^\ast\mathcal{X}_\Sigma \ar[d]_-{\mathrm{pr}_T} \ar[r]^-{\mathrm{pr}_{\mathcal{X}_\Sigma}} & \mathcal{X}_\Sigma \ar[d]^-f \\
    & T \ar[r]_-{f\circ t} & \mathcal{X}
    }
\end{align*}
where we implicitly use the pasting property of $2$-Cartesian squares (cf. \cite[\href{https://stacks.math.columbia.edu/tag/02XD}{Tag 02XD}]{Stacks}). Since $\mathcal{X}_\Sigma$ is representable and $\mathcal{X}$ is a Deligne--Mumford stack, $(f\circ t)^\ast\mathcal{X}_\Sigma$ is representable. Since $\mathcal{X}_\Gamma\rightarrow\mathcal{X}$ is representable, $(f\circ t)^\ast\mathcal{X}_\Gamma$ is representable. Being the $2$-fibered product of representable stacks, $t^\ast\mathcal{X}_\Gamma$ is representable as well. Hence we conclude that $\mathcal{X}_\Gamma\rightarrow\mathcal{X}_\Sigma$ is representable. Therefore, $\mathcal{X}_\Gamma$ is representable.\\
\indent Note that $\mathcal{X}_{\Gamma_1(MN)}$ over $\Spec\mathbb{Z}[1/MN]$ for $MN\geq5$ and $\mathcal{X}_{\Gamma(M)}$ over $\Spec\mathbb{Z}[1/M]$ for $M\geq3$ are representable by Proposition \ref{representability} so the desired representability follows from the above observation.
\end{proof}

\begin{remark} \label{rmk: gen zero cong}
We now list candidate congruence subgroups for which $\cX_{\Gamma}$ is isomorphic to $\bP^1$ over some open subset of $\Spec \bZ[1/6]$. 
By \cite{CP03}, the following congruence subgroups are of genus-zero:
\begin{align*}
    &\Gamma_1(N) \textrm{ for } N = 1, \cdots, 10, 12, \\
    &\Gamma_0(N) \textrm{ for } N = 1, \cdots, 10, 12, 16, 18, 25, \\
    &\Gamma(N) \textrm{ for } N = 1, \cdots, 5, \\
    &\Gamma_1(4) \cap \Gamma(2), \quad \Gamma_0(4) \cap \Gamma(2), \quad \Gamma_1(3) \cap \Gamma(2), \quad \Gamma_0(3) \cap \Gamma(2), \\
    &\Gamma_1(2) \cap \Gamma(3),   \quad 
    \Gamma_1(8) \cap \Gamma(2), \quad \Gamma_0(8) \cap \Gamma(2), \quad \Gamma_0(16) \cap \Gamma_1(8), \quad \Gamma_0(25) \cap \Gamma_1(5).
\end{align*}
When $\Gamma$ is of genus-zero and $\cX_{\Gamma}$ is representable, then $\cX_{\Gamma} \cong \bP^1$.
We summarize the results for representability (over some open subset of $\Spec\mathbb{Z})$ as follows:
\begin{itemize}
\item $\mathcal{X}_{\Gamma_1(N)}$ for $N=5,\cdots,10,12$ are representable over $\Spec\mathbb{Z}[1/N]$ 
(cf. \cite[Table 1]{BN}).
\item $\mathcal{X}_{\Gamma_0(N)}$ is not representable.
\item By Proposition \ref{representability}, $\mathcal{X}_{\Gamma(N)}$ for $N \geq 3$ is representable over $\Spec\mathbb{Z}[1/N]$.
\item By Proposition \ref{representability-(M,N)}, we know that
\begin{align*}
    &\mathcal{X}_{\Gamma_1(2) \cap \Gamma(3)},  \quad \mathcal{X}_{\Gamma_1(8) \cap \Gamma(2)}, \quad \mathcal{X}_{\Gamma_1(3) \cap \Gamma(2)}, \quad \mathcal{X}_{\Gamma_0(16) \cap \Gamma_1(8)},  \quad \mathcal{X}_{\Gamma_0(25) \cap \Gamma_1(5)}
\end{align*}
are representable over some open subset of $\Spec\mathbb{Z}$. 
Also, $\cX_{\Gamma_1(4) \cap \Gamma(2)}$ is also isomorphic to $\bP^1$ (cf. \cite[Table 1]{BN}).
\item 
We do not know whether remaining $3$ cases $\Gamma_0(4) \cap \Gamma(2), \Gamma_0(3) \cap \Gamma(2), \Gamma_0(8) \cap \Gamma(2)$ satisfy $\cX_{\Gamma} \cong \bP^1$ or not.
\end{itemize}
We do not have the average rank theorem for $\Gamma_0(16) \cap \Gamma_1(8)$ and $ \Gamma_0(25) \cap \Gamma_1(5)$, because the lack of Theorem \ref{thm:Hmoment12} in this case.
\end{remark}

Suppose that $\mathcal{X}_\Gamma$ is not representable. 
In this case, we cannot say that \eqref{obstruction} is bijective because the existence part of the valuative criterion for algebraic stacks needs an extension of valuation rings \cite[\href{https://stacks.math.columbia.edu/tag/0CLK}{Tag 0CLK}]{Stacks}.
Suppose that $\Gamma$ has level $N$ and $\mathcal{X}_\Gamma$ is not necessarily representable. Note that the cusps are $\mathcal{O}_K[\zeta_N]$-rational by \cite[section 9.4]{KM85} and \cite[Theorem 10.9.1]{KM85} (cf. \cite[VII.2]{DR73}). 

Let $\mathfrak{p}$ be a nonzero prime ideal of $\mathcal{O}_K[1/N]$.
For a given $E\in\mathcal{Y}(K_\mathfrak{p})$, we let $\mathfrak{W}(E)$ denote a minimal Weierstrass model of $E$ over $\mathcal{O}_{K,\mathfrak{p}}$ in the sense of \cite[VII.1]{Sil}. This is unique up to isomorphism over $\mathcal{O}_{K,\mathfrak{p}}$ by \cite[Proposition VII.1.3]{Sil}. For each ring $R$, we also denote $\underline{\mathcal{X}_\Gamma}(R)^\ast\vcentcolon=\underline{\mathcal{X}_\Gamma}(R)\coprod\ast$ where $\ast$ is the category with a single object and a single morphism. By assigning the elliptic curves with additive reduction at $\fp$ to the additional point $\ast$, the usual reduction process via the minimal Weierstrass model gives the following map\footnote{This map is well-defined because we are working over isomorphism classes (see the notation on the first page).}
\begin{align}\label{reduction-Weierstrass}
    \xymatrix{\mathcal{Y}(K_\mathfrak{p}) \ar[r] & \mathcal{X}(\kappa(\mathfrak{p}))^\ast & E \ar@{|->}[r] & \mathfrak{W}(E)\otimes_{\mathcal{O}_{K,\mathfrak{p}}}\kappa(\mathfrak{p})}.
\end{align}
Note that the smooth locus $\mathfrak{W}(E)^\mathrm{sm}\subseteq\mathfrak{W}(E)$ becomes the identity component of the N\'eron model of $E$ over $\mathcal{O}_{K,\mathfrak{p}}$. By the N\'eron mapping property (see \cite[Definition 1.2.1]{BLR90}), the canonical map
\begin{align*}  
    \xymatrix{\mathfrak{W}(E)^\mathrm{sm}(\mathcal{O}_{K,\mathfrak{p}}) \ar[r] & E(K_\mathfrak{p})}
\end{align*}
is a group isomorphism. Following \cite[VII.3 Proposition 3.1]{Sil} and its proof, for a positive integer $\ell$, one can show that the composition
\begin{align*}
    \xymatrix{E(K_\mathfrak{p})[\ell^\infty] \ar[r]^-\sim & \mathfrak{W}(E)^\mathrm{sm}(\mathcal{O}_{K,\mathfrak{p}})[\ell^\infty] \ar[r] & \mathfrak{W}(E)^\mathrm{sm}(\kappa(\mathfrak{p}))}
\end{align*}
is injective when the characteristic of $\kappa(\fp)$ does not divide $\ell$. Consequently, if a $\Gamma$-structure on an elliptic curve $E$ over $K$ is determined by a set of $K$-rational torsion points in $E$ of order relatively prime to the characteristic of $\kappa(\mathfrak{p})$, then \eqref{reduction-Weierstrass} uniquely lifts to
\begin{align*}
    \xymatrix{\mathcal{Y}_\Gamma(K_\mathfrak{p}) \ar[r] & \mathcal{X}_\Gamma(\kappa(\mathfrak{p}))^\ast\rlap{\ .}}
\end{align*}

In some cases, we can extend this to $\mathcal{X}_\Gamma(K_\mathfrak{p})$. For example, if $\mathcal{O}_{K,\mathfrak{p}}$ admits a primitive $N$-th root of unity, then \eqref{obstruction} induces a bijection on $\mathcal{X}_\Gamma^\mathrm{cusp}(\mathcal{O}_{K,\mathfrak{p}})\cong\mathcal{X}_\Gamma^\mathrm{cusp}(K_\mathfrak{p})$. In this case, we have an obvious reduction modulo $\mathfrak{p}$ on cusps analogous to \eqref{reduction-representable}:
\begin{align*}
    \xymatrixcolsep{4pc}\xymatrix{
    & \mathcal{X}_\Gamma^\mathrm{cusp}(\mathcal{O}_{K,\mathfrak{p}}) \ar[d]^-\wr \ar[r]^-{-\otimes_{\mathcal{O}_{K,\mathfrak{p}}}\kappa(\mathfrak{p})} & \mathcal{X}_\Gamma^\mathrm{cusp}(\kappa(\mathfrak{p}))\rlap{\ .} \\
    \mathcal{X}_\Gamma^\mathrm{cusp}(K) \ar[r] & \mathcal{X}_\Gamma^\mathrm{cusp}(K_\mathfrak{p}) & 
    }
\end{align*}
Consequently, we get the reduction maps for such $K$ in this case:
\begin{align}\label{reduction-zetaN}
    \xymatrix{\psi_\mathfrak{p}:\mathcal{X}_\Gamma(K) \ar[r] & \mathcal{X}_\Gamma(K_\mathfrak{p}) \ar[r] & \mathcal{X}_\Gamma(\kappa(\mathfrak{p}))^\ast}.
\end{align}
In particular, the cusps of $\mathcal{X}$ are all defined over $\mathbb{Z}$, and hence for any number field $K$ and any nonzero prime ideal $\mathfrak{p}$ of $\mathcal{O}_K$, we have reduction modulo $\mathfrak{p}$ on $\mathcal{X}(K)$.

Finally, by construction, \eqref{reduction-representable} and \eqref{reduction-zetaN} are compatible with forgetting the level structures. In other words, they fit into the following commutative square:
\begin{align} \label{eqn:comdiag}
\begin{aligned}
    \xymatrix{
    \mathcal{X}_\Gamma(K) \ar[d]_-{\psi_\mathfrak{p}} \ar[r]^-{\phi_\Gamma} & \mathcal{X}(K) \ar[d]^-{\psi_\mathfrak{p}} \\
    \mathcal{X}_\Gamma(\kappa(\mathfrak{p}))^\ast \ar[r]_-{\phi_\Gamma} & \mathcal{X}(\kappa(\mathfrak{p}))^\ast\rlap{\ .}
    }
\end{aligned}
\end{align}

\begin{remark}\label{Weierstarass-not-honest}
    The reduction process using the Weierstrass minimal model does not give an honest section of \eqref{obstruction}. Namely, we have minimal models over $\mathcal{O}_{K,\mathfrak{p}}$ which have additive reduction. This is impossible for objects in $\mathcal{X}_{\Gamma}(\mathcal{O}_{K,\mathfrak{p}})$.
\end{remark}

\begin{remark}\label{wtdprojred}
    Note that $\Gamma$ in Propositions \ref{representability} and \ref{representability-(M,N)} has the following properties.
    \begin{itemize}
        \item $\mathcal{X}_\Gamma$ is representable.
        \item A $\Gamma$-structure on an elliptic curve $E$ is determined by a subset of rational torsion points on $E$ (cf. \cite[Chapter 1]{KM85}).
    \end{itemize}
    By the uniqueness of $\psi_\mathfrak{p}$ in the representable case, the reduction process using the minimal model agrees with the one defined via the valuative criterion. 
\end{remark}

Let $\mathcal{P}(w)\vcentcolon=\mathcal{P}(w_0,\cdots,w_n)$ be a weighted projective space regarded as a quotient stack.

\begin{remark} \label{rmk:cXbPdiag}
    Suppose that we have an isomorphism $\mathcal{X}_\Gamma\cong\mathcal{P}(u)$ onto a weighted projective line as a stack. On $\mathcal{P}(u)=\mathcal{P}(u_0,u_1)$, we can imagine a reduction process analogous to Remark \ref{reduction-P1}: For a given $x\in\mathcal{P}(u)(K_\mathfrak{p})$, we choose a representative $x=[x_0,x_1]$ such that $x_i\in\mathcal{O}_{K,\mathfrak{p}}$ with $\min\{\mathrm{val}_\mathfrak{p}(x_0),\mathrm{val}_\mathfrak{p}(x_1)\}$ minimal, and then take $[x_0\bmod\mathfrak{p},x_1\bmod\mathfrak{p}]$.
    
    However, if one of $u_i$ is larger than $1$, then representatives like $(x_0,x_1)$ with $\mathrm{val}_\mathfrak{p}(x_i)=1$ for $i=0,1$ map to $(0,0)\in\mathbb{A}^2(\kappa(\mathfrak{p}))$ which does not define a point in $\mathcal{P}(u)(\kappa(\mathfrak{p}))$. Correspondingly, we add a dummy point $\mathcal{P}(u)(\kappa(\mathfrak{p}))^\ast\vcentcolon=\mathcal{P}(u)(\kappa(\mathfrak{p}))\cup\ast$ and send the ill-defined points to $\ast$.

    Moreover, by Remark \ref{Weierstarass-not-honest}, we cannot guarantee the uniqueness of ``sections'' and hence the compatibility of the reduction process on $\mathcal{X}_\Gamma$ and the reduction process on $\mathcal{P}(u)$ described above. Fortunately, for $\mathfrak{p}\in\Spec\mathcal{O}_K[1/6N]$, the two reduction processes are compatible under the usual parametrization
    \begin{align*}
        \xymatrix{\mathcal{P}(4,6) \ar[r] & \mathcal{X} & [A,B] \ar@{|->}[r] & y^2=x^3+Ax+B.}
    \end{align*}
    In summary, we have a commutative diagram
    \begin{align} \label{diag:modp}
    \begin{aligned}
    \xymatrix{
    \cX_{\Gamma}(K_{\fp}) \ar[r] \ar[d]_-{\psi_{\fp}} & \cP(u)(K_{\fp}) \ar[d]^-{\psi_{\fp}} \\
    \cX_{\Gamma}(\kappa(\fp))^* \ar[r] & \cP(u)(\kappa(\fp))^*
    }    
    \end{aligned}
    \end{align}
    when $\Gamma = \SL_2$ or $\cX_{\Gamma}$ is representable.
\end{remark}

\begin{remark} \label{rmk:Phivsphi}
    It is a natural question to compare the parametrization $\Phi_{\Gamma}$ in \cite{CJ2} and the morphism $\phi_{\Gamma} : \cX_{\Gamma} \to \cX$ forgetting the level structure. In \cite{CJ2}, we implicitly chose an isomorphism
    \begin{align*}
        \xymatrix{\mathcal{P}(1,1) \ar[r]^-\sim & \mathcal{X}_\Gamma & t \ar@{|->}[r] & (E(u_t,v_t),\alpha_t)}
    \end{align*}
    and rewrote $E(u_t,v_t)$ as $y^2=x^3+f_\Gamma(t)x+g_\Gamma(t)$. This gives a commutative diagram:
    \begin{align*}
        \xymatrix{
        \mathcal{P}(1,1)(\mathbb{Q}) \ar[d]_-{\Phi_\Gamma} \ar[r]^-\sim & \mathcal{X}_\Gamma(\mathbb{Q}) \ar[d]^-{\phi_\Gamma} & t \ar@{|->}[d] \ar@{|->}[r] & (E(u_t,v_t),\alpha_t) \ar@{|->}[d] \\
        \mathcal{P}(4,6)(\mathbb{Q}) \ar[r]^-\sim & \mathcal{X}(\mathbb{Q}) & (f_\Gamma(t),g_\Gamma(t)) \ar@{|->}[r] & E(u_t,v_t).
        }
    \end{align*}
\end{remark}

\section{Counting elliptic curves revisited} \label{sec:counting}

\subsection{Weights for local conditions} \label{subse:Intui}

In this section, we define weights for local conditions and compare them with those in \cite{CJ2}.
To compare the results with the cases of $\cX_{\Gamma} \cong \bP^1$, we recall the results of \cite{CJ2} for $|G| \geq 5$, where $G$ is a finite abelian group that can arise as a torsion subgroup of elliptic curves over $\bQ$. 
For such $G$, there exist polynomials $f_{G}(t), g_{G}(t)$ such that 
\begin{align*}
	y^2 = x^3 + f_{G}(t)x + g_{G}(t), \qquad t \in \bQ,
\end{align*}
parametrizes the elliptic curves with a prescribed torsion subgroup $G$.
After clearing denominators, we may regard $f_{G}, g_{G}$ as polynomials with integral coefficients, defined on pairs of relatively prime integers.
Let $\cE(X)$ be the set of isomorphism classes of elliptic curves over $\bQ$ whose naive height is $\leq X$, and
\begin{align*}
    \cE_G(X) \vcentcolon = \lcrc{E \in \cE(X) : E(\bQ) \geq G}, \qquad 
    \cE_{G, J}(X) \vcentcolon = \lcrc{E \in \cE_G(X) : E \equiv C_J \pmod{p} },
\end{align*}
where $C_J$ denotes the (possibly singular) curve given by the equation $y^2 = x^3 + \alpha x + \beta$ for $J = (\alpha, \beta) \in \bF_p^2$.
Heuristically,
\begin{align} \label{eqn:WGJinCJ}
	W_{G, J} = \lcrc{(a, b) \in \bF_p^2 : (f_{G}(a, b) , g_{G}(a, b)) \equiv J \pmod{p}}
\end{align}
is a ``weight'' which measures the number of fibers of $\cE_{G}(X) \to \bF_p^2$.
By \cite[Theorem 3.9]{CJ2}, as expected, there are explicit constants $c(G), d(G)$ satisfying
\begin{align*}
	|\cE_{G, J}(X)| = \frac{|W_{G,J}|}{p^2-1}c(G)X^{\frac{1}{d(G)}} + O(X^{\frac{1}{2d(G)}} + p^{-1}X^{\frac{1}{2d(G)}}\log X).
\end{align*}
Moreover, $|W_{G, J}|$ is the number of embeddings of $G$ into $E_J(\bF_p)$ when $G = \bZ/5\bZ$ or $\bZ/6\bZ$.
However, this interpretation of $|W_{G,J}|$ as embeddings does not hold for other $G$.
In fact, the cardinality of $W_{G, J}$ is not the same even when the curves $E_{J}$ are isomorphic over $\bF_p$.
For example, let $G = \bZ/7\bZ$ and let
\begin{align*}
    E_{J_1} : y^2 = x^3 + 2x + 1,  \qquad
    E_{J_2} : y^2 = x^3 + 2x + 4.
\end{align*}
Then $E_{J_1}$ and $E_{J_2}$ are isomorphic over $\bF_5$ but $|W_{\bZ/7\bZ, J_1}| = 0$ and $|W_{\bZ/7\bZ, J_2}|=12$.
Note that the number of embeddings of $\bZ/7\bZ$ in $E_{J_i}(\bF_5)$ is $6$ (cf. \cite[Remark 2]{CJ2}).
Therefore, we need to redefine the weight $W_{G, J}$ to ensure that $|W_{G, J_1}| = |W_{G, J_2}|$ when $E_{J_1} \cong E_{J_2}$.

As we have seen in section \ref{sec:Prestack}, it can be achieved by considering a local condition of an elliptic curve as a finite subset in $\cX(\kappa(\fp))^\ast$.
More precisely, we say that $(E, \alpha) \in \cX_{\Gamma}(K_{\fp})$ satisfies the local condition $S \subset \cX(\kappa(\fp))^*$ if $(\psi_{\fp} \circ \phi_{\Gamma, K_{\fp}})(E, \alpha)$ is in $S$.
According to the commutative diagram (\ref{eqn:comdiag}), it is natural to replace $W_{G, J}$ by $\phi_{\Gamma, \kappa(\fp)}^{-1}(E_z)$ where $E_J$ is isomorphic to the (generalized) elliptic curve corresponds to $E_z \in \cX(\kappa(\fp))$ over $\kappa(\fp)$.

From now on, we usually write $\phi_{\Gamma, R}$ as $\phi_{\Gamma}$.

\begin{proposition} \label{prop:keymod}
Let $M, N$ be positive integers, $\Gamma = \Gamma_1(M, N)$ and $\phi_{\Gamma} : \cY_{\Gamma} \to \cY$ the map forgetting the level structure. 
Suppose that the characteristic of $\kappa(\fp)$ is relatively prime to the level of $\Gamma$.
Then, for the induced function $\phi_\Gamma:\mathcal{Y}_\Gamma(\kappa(\fp))\rightarrow\mathcal{Y}(\kappa(\fp))$ and for $E\in\mathcal{Y}(\kappa(\fp))$,
\begin{align*}
|\phi_\Gamma^{-1}(E)|=\frac{\left| \lcrc{i : \bZ/M\bZ \times \bZ/MN\bZ \hookrightarrow E(\kappa(\fp))} \right| }{|\Aut_{\kappa(\fp)}(E)|}
 \end{align*}
\end{proposition}
\begin{proof} 
By \cite[Theorem 7.1.3]{KM85}, the fibers of $\phi_\Gamma:\mathcal{Y}_\Gamma\rightarrow\mathcal{Y}$ are representable for $\Gamma$. Hence, each point in the fiber has no additional automorphisms. Then
\begin{align*}
    |\phi_\Gamma^{-1}(E)|=\frac{|\underline{\phi_\Gamma}^{-1}(E)|}{|\Aut_{\kappa(\fp)}(E)|}
\end{align*}
because each automorphism of $E$ defines a different but isomorphic level structure. Hence it suffices to determine $|\underline{\phi_\Gamma}^{-1}(E)|$ in each case, where $\underline{\phi_\Gamma}^{-1}(E)$ is the set of corresponding level structures on $E$.
\end{proof}

Now, we compare $\phi_{\Gamma}^{-1}(E_z)$ and $W_{G, J}$ when $\cX_{\Gamma} \cong \bP^1$.

\begin{lemma} \label{lem:WGJcomapare}
Let $\Gamma$ be a congruence subgroup, either $\Gamma_1(N)$ for $N = 5, \ldots, 10, 12$ or $\Gamma(2) \cap \Gamma_1(2N)$ for $N = 3, 4$, and let $G$ be the corresponding abelian group. Let $p \geq 5$ be a prime, $E_z \in \cX(\bF_p)$, and $z \in \cP(4,6)(\bF_p)$ the point associated to it.
Then, 
\begin{align*}
|\phi_{\Gamma}^{-1}(E_z)| = \frac{1}{p-1}\sum_{\substack{J \in \bF_p^2  \\ E_J \cong E_z}} |W_{G, J}|.
\end{align*}
\end{lemma}
\begin{proof}
Since we are in the situation of Remark \ref{rmk:Phivsphi}, $\phi_\Gamma$ is identified with $(f_G,g_G)$ via the implicitly chosen isomorphism $\mathcal{X}_\Gamma \cong \bP^1$. 
This is also compatible with the mod $p$ reduction introduced in (\ref{diag:modp}), provided $p \geq 5$. 
Therefore, for $E_z \in \cY(\bF_p)$,
\begin{align*}
	\bigsqcup_{\substack{J \in \bF_p^2  \\ E_J \cong E_z}} W_{G, J}
        &= \lcrc{ (a, b) \in \bF_p^2 : (f_{G}(a, b), g_{G}(a, b)) \equiv J \pmod{ p} \textrm{ for } E_J \cong E_z  }\\
	&= \lcrc{(a, b) \in \bF_p^2 : E_{(f_{G}(a, b), g_{G}(a, b))} \cong E_z},
\end{align*}
by the definition of $W_{G, J}$ in (\ref{eqn:WGJinCJ}).
By \cite[Appendix A]{CJ2}, if $(a, b) \in W_{G, J}$ then $(ua, ub) \in W_{G, J}$ for $u \in \bF_p^\times$. 
On the other hand, for $t \in \bP^1(\bF_p)$, there are $(p-1)$ pairs $(a, b) \in \bF_p^2$ such that $(f_G(a, b), g_G(a, b)) = \phi_{\Gamma}(t)$. 
So
\begin{align*}
\frac{1}{p-1}\sum_{E_J \cong E_z}|W_{G, J}| &=
\left| \lcrc{ t \in \bP^1(\bF_p) : E_{(f_{G}(t), g_{G}(t))}  \cong E_z} \right| = \left| \lcrc{t \in \bP^1(\bF_p) : E_{\phi_{\Gamma}(t)}  \cong E_z} \right| \\
&= \left| \lcrc{t \in \bP^1(\bF_p) : \phi_{\Gamma}(t) = z} \right|.
\end{align*}
Hence, we have the result.
\end{proof}

\begin{remark} \label{rmk:Wvsphi}
Proposition \ref{prop:keymod} can be regarded as a modification of \cite[Lemma 2.6]{CJ2}, which proves Proposition \ref{prop:keymod} for $\bZ/N\bZ$ with $N = 2, \ldots, 6$ and $\bZ/2\bZ \times \bZ/2\bZ, \bZ/2\bZ \times \bZ/4\bZ$ by direct computation with coordinates. 
For $E_z \in \cY(\bF_p)$, the number of pairs $J \in \bF_p^2$ such that $E_J \cong E_z$ is $\frac{p-1}{|\Aut_{\bF_p}(E_z)|}$. 
Hence, the right-hand side of Lemma \ref{lem:WGJcomapare} can be interpreted as follows:
\begin{align*}
    \frac{1}{p-1} \sum_{E_J \cong E_z}|W_{G, J}|=\frac{1}{|\Aut_{\bF_p}(E_z)|}\times (\text{``average of $|W_{G,J}|$''}).
\end{align*}
Furthermore, Lemma \ref{lem:WGJcomapare} resolves the problem introduced in \cite[Remark 2]{CJ2}. 
Since there are two isomorphic elliptic curves $E_{J_1}, E_{J_2}$ with $|W_{G, J_1}| = 0$ and $|W_{G, J_2}| = 12$, 
we have 
\begin{align*}
|\underline{\phi_\Gamma}^{-1}(E_{J_i})| = |\Aut_{\bF_5}(E_{J_i})||\phi_{\Gamma}^{-1}(E_{J_i})| = \frac{1}{2}(0 + 12) = 6   
\end{align*}
which is exactly the number of embeddings of $\bZ/7\bZ$ into $E_{J_i}(\bF_5)$.
\end{remark}

\subsection{Counting elliptic curves} \label{subsec:countingec}

We recall that $\cP(w)$ denotes the weighted projective space, and that $\cP(w)(\kappa(\fp))^\ast$ is defined as $\cP(w)(\kappa(\fp)) \cup \ast$, as in Remark \ref{rmk:cXbPdiag}.
\begin{definition} \label{def:wtz}
For $z \in \cP(w_0, w_1)(\kappa(\fp))^*$, we define its weight with respect to the map $\bA^2(\kappa(\fp)) \to \cP(w_0, w_1)(\kappa(\fp))^*$ by
\begin{align*}
    \mathrm{wt}(z) 
    = \sum_{\substack{(z_0, z_1) \in \bA^2(\kappa(\fp)) \\ (z_0, z_1) \equiv z \in \cP(w_0, w_1)(\kappa(\fp))^{\ast} }} 1.
\end{align*}  
\end{definition}
In other words, $\mathrm{wt}(z)$ is simply the cardinality of the fiber of 
$\bA^2(\kappa(\fp)) \to \cP(w_0, w_1)(\kappa(\fp))^\ast$ over $z$.
From now on, we write $q \vcentcolon= |\kappa(\fp)|$ and use also $\bF_q$ for $\kappa(\fp)$.
We write $\mu_k(\bF_q)$ for the group of $k$-th roots of unity in $\bF_q$.

\begin{lemma} \label{lem:wtXast}
Let $\cP(w_0, w_1)$ be a weighted projective line and let $g = \gcd(w_0, w_1)$.
For a point $z \in \cP(w_0, w_1)(\bF_q)^*$ and $a, b \in \bF_q^\times$, we have $\wt(\ast) = 1$, and 
\begin{align*}
    \wt([a,b]) =  \frac{q-1}{|\mu_g(\bF_q)|}, \quad 
    \wt([a,0]) = \frac{q-1}{|\mu_{w_0}(\bF_q)|}, \quad
    \wt([0,b]) = \frac{q-1}{|\mu_{w_1}(\bF_q)|}.
\end{align*}
\end{lemma}
\begin{proof}
This follows immediately from the definition of the weight $\wt$.
\end{proof}

For a tuple of natural numbers $w$, we define $|w| = \sum_{i=0}^{n} w_i$.
Let $M_K$ (resp. $M_{K, 0}$, $M_{K, \infty}$) be the set of places of $K$ (resp. finite places, infinite places).
We normalize the absolute values as follows:
\begin{align*}
    |\pi_v|_v = \frac{1}{N_{K/\bQ}(\fp_v)}, \, \textrm{for finite }v, \qquad
    |a|_v = |i_v(a)|^{[K_v : \bR]}, \, \textrm{for infinite }v,
\end{align*}
where $\pi_v$ is a uniformizer at $v$ and $i_v : K \hookrightarrow K_v$ is the natural embedding.
For $(x_0, \cdots, x_n) \in K^{n+1} \setminus \lcrc{0}$ we define
\begin{align*}
    |(x_0, \cdots, x_n)|_{w, v} \vcentcolon=
    \left\{ \begin{array}{lll}
        \max\limits_{0 \leq i \leq n} \lcrc{|\pi_v|_v^{ \lfloor \frac{\ord_v(x_i)}{w_i} \rfloor }} & \textrm{for } v \in M_{K, 0},  \\
        \max\limits_{0 \leq i \leq n} \lcrc{|x_i|_v^{\frac{1}{w_i}}} &  \textrm{for } v \in M_{K,\infty}.
    \end{array}
    \right.
\end{align*}
Then, for $[x_0, \cdots, x_n] \in \cP(w)(K)$, we define
\begin{align*}
    H_{w, K}([x_0, \cdots, x_n]) \vcentcolon=
    \prod_{v \in M_K} |(x_0, \cdots, x_n)|_{w, v}.
\end{align*}
This definition does not depend on the choice of homogeneous coordinates, so $H_{w,K}$ is well defined on $\cP(w)(K)$.
For an isomorphism class of elliptic curves over $K$, we define its height $H(E)$ as the height of the corresponding point in $\cP(4, 6)(K)$.
In other words,
\begin{align*}
    H(E) \vcentcolon= H_{(4, 6), K}([A, B])
\end{align*}
where $y^2 = x^3 + Ax + B$ is any Weierstrass model of $E$.
If a given elliptic curve has a global minimal Weierstrass model, in other words, if there exist $A, B \in \cO_{K}$ such that no place $v \in M_{K,0}$ satisfies $v^4 \mid A$ and $v^6 \mid B$, then we have
\begin{align} \label{eqn:Heightobs}
    H_{(4, 6), K}([A, B]) 
    = \prod_{v \in M_{K, \infty}} \max \lcrc{|A|_v^{\frac{1}{4}}, |B|_v^{\frac{1}{6}} }
    = \prod_{v \in M_{K, \infty}} \max \lcrc{|A^3|_v, |B^2|_v }^{\frac{1}{12}}.
\end{align}

When the class number of $K$ is greater than $1$, an elliptic curve may not admit a global minimal model.
In this case, we make the following modification, which will be used later.
\begin{lemma} \label{lem: Height inf}
Let $E$ be an elliptic curve over a number field $K$.
Then there exists a Weierstrass model $y^2 = x^3 + Ax + B$ and a constant $C_K$, depending only on $K$, such that
\begin{align*}
H(E) = H_{(4,6), K}([A,B]) \geq C_K \prod_{v \in M_{K, \infty}} |(A, B)|_{(4, 6), v}. 
\end{align*}
\end{lemma}
\begin{proof}
For a number field $K$ and a set $S$ of finite places of $K$, we denote by $\mathcal{O}_{K,S}$ the ring of $S$-integers.
Since the class number of $K$ is finite, there exists a finite set of places $S$ such that $\cO_{K,S}$ is a principal ideal domain.
Then every elliptic curve $E$ over $K$ admits an $S$-minimal Weierstrass equation of the form $y^2 = x^3 + Ax + B$ (see \cite[Proposition VIII.8.7]{Sil}).
For any prime $v \notin S$, we have $\ord_v(A) \leq 4$ or $\ord_v(B) \leq 6$, which implies $|(A,B)|_{(4,6),v} = 1$.

For a prime $v \in S$, there exists a positive integer $n_v$ such that $\fp_v^{n_v}$ is principal, i.e., $\fp_v^{n_v} = a_v\cO_K$ for some $a_v \in \cO_K$, by the finiteness of the class number.
Then there exists an integer $k_v \geq 0$ such that, after replacing the Weierstrass equation by
\begin{align*}
    y^2 = x^3 + a_v^{4k_v}Ax + a_v^{6k_v}B,
\end{align*}
we may assume that 
\begin{align*}
    \ord_v(a_v^{4k_v}A) < 4n_v \quad \textrm{or} \quad \ord_v(a_v^{6k_v}B) < 6n_v.
\end{align*}
We set $A \vcentcolon = a_v^{4k_v}A$ and $B \vcentcolon = a_v^{6k_v}B$.
Let $q_v \vcentcolon= |\kappa(\fp_v)|$. 
Then for $v \in S$,
\begin{align*}
    |(A, B)|_{(4, 6), v} = \max \lcrc{ q_v^{-\left\lfloor\frac{\ord_v(A)}{4}\right\rfloor}, q_v^{-\left\lfloor\frac{\ord_v(B)}{6}\right\rfloor}}
    \geq q_v^{-n_v}.
\end{align*}
Considering both cases $v \in S$ and $v \not\in S$, we have
\begin{align*} 
    H_{(4, 6), K}([A, B]) 
    = \prod_{v \in M_{K, 0}} |(A, B)|_{(4, 6), v} \prod_{v \in M_{K, \infty}} |(A, B)|_{(4, 6), v}
    \geq \prod_{v \in S} q_v^{-n_v} \prod_{v \in M_{K, \infty}} |(A, B)|_{(4, 6), v}.
\end{align*}
Since the set $S$ and the integers $n_v$ depend only on the number field $K$, we obtain the result.
\end{proof}

When $\Gamma$ satisfies $\cX_{\Gamma} \cong \bP^1$ over $\Spec \bZ[1/6N]$,
the morphism $\phi_{\Gamma} : \cX_{\Gamma} \to \cX$ forgetting the level structure induces a morphism $\cP(1, 1) \to \cP(4, 6)$, which we also denote by $\phi_{\Gamma}$.
We define
\begin{align} \label{eqn:def eGam}
    e(\Gamma) \vcentcolon= e(\phi_{\Gamma}) = \frac{u_0u_1}{24}[\SL_2(\bZ) : \Gamma]
\end{align}
when $\cX_{\Gamma} \cong \cP(u)$ following \cite[Lemma 4.1, \S 8]{BN}.
We recall that $q = |\kappa(\fp)|$ and $d = [K : \bQ]$.

\begin{theorem} \label{prop:phi411}
Let $\Gamma$ be a congruence subgroup such that $\cX_{\Gamma} \cong \bP^1$ over some open subset of $\Spec \bZ[1/6]$, and let $\phi_{\Gamma}: \cP(1, 1) \to \cP(4, 6)$ be the morphism corresponding to the morphism forgetting the level structure. 
For a prime $\fp \not\in S_{\phi_{\Gamma}}$, a point $z \in \cP(4, 6)(\bF_q)$, and
\begin{align*} 
    \Omega_{\fp, z } = \lcrc{ y_{\fp} \in \cP(4, 6)(K_{\fp}) : \psi_{\fp}(y_{\fp}) = z  },
\end{align*}
there exists a constant $\kappa$, depending on $K$ and $\Gamma$, such that 
\begin{align*}
    &\left|
    \lcrc{y \in \phi_{\Gamma}(\cP(1,1)(K)) : H_{(4,6), K}(y) \leq X, \psi_{\fp}(i_{\fp}(y)) \in \Omega_{\fp, z}}
    \right| \\
    &= \kappa \cdot \frac{\# \phi_{\Gamma, \bF_q}^{-1}(z)}{q+1} \cdot X^{\frac{2}{e(\Gamma)}} +
    O\lbrb{ \lbrb{1 + \frac{1}{q}} X^{\frac{2d- 1}{de(\Gamma)}} \log X  
    }.
\end{align*}
\end{theorem}
\begin{proof} 
The definition of $S_{\phi_\Gamma}$ and the proof are given in the last section.
We emphasize that the implied constant does not depend on $\fp$.
\end{proof}

Let $\cE_K(X)$ be the set of isomorphism classes of elliptic curves over $K$ of naive height at most $X$, and let $\cE_{K, \Gamma}(X) \subset \cE_K(X)$ be the subset consisting of elliptic curves that admit a $\Gamma$-level structure.
Then
\begin{align*}
    \cE_{K, \Gamma}(X) = \lcrc{E \in \phi_{\Gamma}(\cX_{\Gamma}(K)) \subset \cX(K) : H(E) \leq X}.
\end{align*}
For $E_z \in \cX(\bF_q)$, we define
\begin{align*}
    \cE_{K, \Gamma, \fp}^z(X) \vcentcolon=
    \lcrc{E \in \cE_{K, \Gamma}(X) : E \cong E_z \pmod{\fp}},
\end{align*}
and for $a$ an integer satisfying the Weil bound $[-2\sqrt{q}, 2\sqrt{q}]$ at $\fp$, we define
\begin{align*}
    \cE_{K, \Gamma, \fp}^a(X) \vcentcolon = 
    \lcrc{E \in \cE_{K, \Gamma}(X) : E \text{ has good reduction at } \fp \text{ and } a_{\fp}(E) = a}.
\end{align*}

We paraphrase Theorem \ref{prop:phi411} as follows.
\begin{proposition} \label{prop:EKGamz}
Let $\Gamma$ be a congruence subgroup of level $N$ such that $\cX_{\Gamma} \cong \bP^1$ over $\bZ[1/6N]$, and let $\fp$ be a prime not dividing $6N$.
Let $\phi_{\Gamma} : \cP(1, 1) \to \cP(4, 6)$ be the morphism forgetting the level structure and $E_z \in \cX(\bF_q)$.
Then there exists a constant $\kappa$ depending on $K$ and $\Gamma$ such that 
\begin{align} \label{eqn:EKGamz}
  |\cE_{K, \Gamma, \fp}^z(X)| =
  \kappa \cdot \frac{|\phi_{\Gamma, \bF_q}^{-1}(z)|}{q+1}   \cdot  X^{\frac{2}{e(\Gamma)}}
    +    O\lbrb{ \lbrb{1 + \frac{1}{q}} X^{\frac{2d- 1}{de(\Gamma)}} \log X }.
\end{align}
Moreover, $|\cE^*_{K, \Gamma, \fp}(X)| = 0$.
\end{proposition}
\begin{proof}
Using $\cX_{\Gamma} \cong \bP^1$, $\cX \cong \cP(4, 6)$, and the definition of $\phi_{\Gamma}$ and height of elliptic curves,
\begin{align*}
    \cE_{K, \Gamma}(X) = \lcrc{ y \in \phi_{\Gamma}(\cP(1,1)(K)) \subset \cP(4, 6)(K) : H_{(4, 6), K}(y) \leq X}.
\end{align*}
Since the diagram (\ref{diag:modp}) commutes, 
\begin{align*}
    \cE_{K, \Gamma, \fp}^z(X) = \lcrc{ y \in \phi_{\Gamma}(\cP(1,1)(K)) \subset \cP(4, 6)(K) : H_{(4, 6), K}(y) \leq X, \psi_{\fp}(y) = z },
\end{align*}
and (\ref{eqn:EKGamz}) follows from Theorem \ref{prop:phi411}.
Note that $|\phi_{\Gamma, \bF_q}^{-1}(z)|$ is bounded by the degree of $\phi_{\Gamma}$.
Also, there are no elliptic curves with a representable level structure that have additive reduction at $\fp$, by (\ref{reduction-representable}).
\end{proof}

By definition, we have
\begin{align*}
    \sum_{E_z \in \cX(\bF_q)^*}|\cE_{K, \Gamma, \fp}^z(X)| = |\cE_{K, \Gamma}(X)|.
\end{align*}
Using Proposition \ref{prop:EKGamz}, we see that the sum of the leading terms on the left-hand side agrees with that of the right-hand side, which provides a simple sanity check.
By Proposition \ref{prop:EKGamz} (see also \cite[Corollary 3.10]{CJ2} and Remark \ref{rmk:estpowcompare}),
\begin{align}  \label{eqn:EKgama}
    |\cE_{K, \Gamma, \fp}^a(X)| &= 
    \frac{\kappa}{q+1}
    \lbrb{\sum_{a_{\fp}(E_z) = a} |\phi_{\Gamma}^{-1}(z)|} X^{\frac{2}{e(\Gamma)}} + O_{\Gamma, K}\lbrb{ \lbrb{1 + \frac{1}{q}}
    \lbrb{\sum_{a_{\fp}(E_z) = a} 1} X^{\frac{2d-1}{de(\Gamma)}}\log X}.
\end{align}
Estimating (\ref{eqn:EKgama}) will be crucial for obtaining an upper bound for the average rank.

\begin{remark}
At first glance, Proposition \ref{prop:EKGamz} claims that the exponent of the asymptotic of $|\cE_{K, \Gamma}(X)|$ does not depend on $K$. 
This is because we used the non-absolute height $H_{w,K}$ in this paper.
Hence, some modification is needed to compare Proposition \ref{prop:EKGamz} with the previous results \cite{CJ1, CJ2}.
\end{remark}

\subsection{Multiplicative reduction}

First, we give a concrete computation of the number of cusps of modular curves over finite fields.

\begin{lemma} \label{lem:numcuspff}
Let $k$ be a field and $\overline{k}$ its algebraic closure. Choose a primitive root of unity $\zeta_N\in\overline{k}$ and identify $\mu_N(\overline{k})\cong\mathbb{Z}/N\mathbb{Z}$ as a $\mathrm{Gal}(\overline{k}/k)$-module via the mod $N$ cyclotomic character $\chi$ (i.e., $\sigma(\zeta_N)=\zeta_N^{\chi(\sigma)}$).
Identifying $(\mathbb{Z}/N\mathbb{Z})^2$ with column vectors, we introduce
\begin{align*}
    \xymatrix{\mathrm{Hom}((\mathbb{Z}/N\mathbb{Z})^2,\mathbb{Z}/N\mathbb{Z})\times\mathrm{Gal}(\overline{k}/k) \ar[r] & \mathrm{Hom}((\mathbb{Z}/N\mathbb{Z})^2,\mathbb{Z}/N\mathbb{Z}) & (\rho,\sigma) \ar@{|->}[r] & \rho\circ{\begin{pmatrix} 1 & 0 \\ 0 & \chi(\sigma) \end{pmatrix}}.}
\end{align*}
and we let $\mathrm{SL}_2(\mathbb{Z}/N\mathbb{Z})$ act on $\mathrm{Hom}((\mathbb{Z}/N\mathbb{Z})^2,\mathbb{Z}/N\mathbb{Z})$ by the right multiplication. In this setting, the number of $k$-rational cusps of $X_\Gamma$ for each congruence subgroup $\Gamma$ of $\mathrm{SL}_2(\mathbb{Z}/N\mathbb{Z})$ is given as follows:
    \begin{align*}
    \#
    \left(
     {\left\{\begin{array}{c} \textrm{surjective homomorphisms} \\(\mathbb{Z}/N\mathbb{Z})^2\rightarrow\mathbb{Z}/N\mathbb{Z}\end{array}\middle\}\right/(\pm \Gamma)}
     \right)^{\Gal(\overline{k}/k)}.
    \end{align*}
\end{lemma}
\begin{proof} This is \cite[Theorem 10.9.1]{KM85} considering the Galois action. To recover the Galois action, we use the identification $\mu_N(\overline{k})\cong\mathbb{Z}/N\mathbb{Z}$ and \cite[10.6]{KM85} to identify the displayed description and the one \cite[VI.5]{DR73} which uses a certain quotient of $\mathrm{Isom}(\mathbb{Z}/N\mathbb{Z}\times\mu_N,(\mathbb{Z}/N\mathbb{Z})^2)$.
\end{proof}

\begin{corollary} \label{cor:Numcusps}
Let $N \geq 1$ be an integer and let $q$ be a prime power.
Suppose that $(q, N) = 1$. Then, \\
(i) If $N = \ell$ or $2\ell$ for an odd prime $\ell$, we have
\begin{align*}
    |X_1(N)^{\mathrm{cusp}}(\bF_q)| =  \left\{
        \begin{array}{lll}
        2(\ell-1)     & \textrm{if }  q \equiv \pm 1 \pmod{N}, \\
        \ell-1     &  \textrm{if }  q \not\equiv \pm 1 \pmod{N}.
        \end{array}
        \right.
\end{align*}
(ii) For $N = 4, 8, 9, 12$, we have $|X_1(4)^{\mathrm{cusp}}(\bF_q)|=3$,
\begin{align*}
        &|X_1(8)^{\mathrm{cusp}}(\bF_q)| = \left\{
        \begin{array}{lll}
        6     & \textrm{if } q \equiv \pm 1 \pmod{8},  \\
        4     & \textrm{if } q \not\equiv \pm 1 \pmod{8}.  
        \end{array}
        \right.\\
       &|X_1(12)^{\mathrm{cusp}}(\bF_q)| = \left\{
        \begin{array}{lll}
        10     & \textrm{if } q \equiv 1 \pmod{12},  \\
        6     & \textrm{if }q \not\equiv 1 \pmod{12}.
        \end{array}
        \right.\\
        &|X_1(9)^{\mathrm{cusp}}(\bF_q)| = \left\{
        \begin{array}{lll}
        8     & \textrm{if }   q \equiv 1 \pmod{9},  \\
        6     & \textrm{if }  q \equiv 8 \pmod{9}, \\
        5   & \textrm{if }  q \equiv 4, 7 \pmod{9}, \\
        3   & \textrm{if }  q \equiv 2, 5 \pmod{9}.
        \end{array}
        \right.
\end{align*}
\end{corollary}
\begin{proof}
    This is an application of Lemma \ref{lem:numcuspff}. 
    Denote by $e_1,e_2$ the standard basis for $(\mathbb{Z}/N\mathbb{Z})^2$. 
    Then a surjective homomorphism $\alpha:(\mathbb{Z}/N\mathbb{Z})^2\rightarrow\mathbb{Z}/N\mathbb{Z}$ is determined by $\alpha_i\vcentcolon=\alpha(e_i)$ for $i=1,2$ which are relatively prime to each other. 
    Let $\overline{\Gamma_1(N)}\leq\mathrm{SL}_2(\bZ/N\bZ)$ denote the image of $\Gamma_1(N) \leq \mathrm{SL}_2(\bZ)$ under the reduction modulo $N$ homomorphism $\mathrm{SL}_2(\bZ)\rightarrow\mathrm{SL}_2(\bZ/N\bZ)$. 
    Then, an element $T$ in $\pm \overline{\Gamma_1(N)}$ can be written as
    \begin{align*}
        \begin{pmatrix}
            \pm 1 & t \\ 0 & \pm 1
        \end{pmatrix}
    \end{align*}
    for some $t \in \bZ/N\bZ$.
    In the setting of Lemma \ref{lem:numcuspff}, $T$ naturally acts on $\alpha$ from right:
    \begin{align*}
        \alpha\cdot T=\begin{pmatrix}
            \alpha_1 & \alpha_2
        \end{pmatrix}
        \begin{pmatrix}
            \pm1 & t \\ 0 & \pm1 
        \end{pmatrix}
        =\begin{pmatrix}
            \pm\alpha_1 & t\alpha_1\pm\alpha_2
        \end{pmatrix}.
    \end{align*}
    Hence a surjective homomorphism $\beta:(\mathbb{Z}/N\mathbb{Z})^2\rightarrow\mathbb{Z}/N\mathbb{Z}$ with $\beta_i\vcentcolon=\beta(e_i)$ lies in the same $\pm\overline{\Gamma_1(N)}$-orbit of $\alpha$ if and only if there is $t\in\mathbb{Z}/N\mathbb{Z}$ such that
    \begin{align}\label{orbit-alpha}
\beta_1=\pm\alpha_1,\qquad\beta_2=t\alpha_1\pm\alpha_2
    \end{align}
    where the signs must be in the same order. 
    
    
    We will give a $\pm\overline{\Gamma_1(N)}$-orbit of each $\alpha$ by concretely computing the action.
    Since $\alpha$ is assumed to be surjective, there are only three possible cases:
    \begin{enumerate}
        \item $\alpha_1=0$ and $\alpha_2\in (\bZ/N\bZ)^\times$.
        \item $\alpha_1 \in (\bZ/N\bZ)^\times$ and there is no restriction on $\alpha_2$.
        \item $\alpha_1 = vd$ for some $v \in (\bZ/N\bZ)^\times$ and $d \mid N$ with $d \neq 1, N$ and $\alpha_2$ is relatively prime to $d$.
    \end{enumerate}

    Case (1). Let $\varphi$ be the Euler totient function. From \eqref{orbit-alpha}, we deduce that $\beta$ is in the $\pm\overline{\Gamma_1(N)}$-orbit of $\alpha$ if and only if $(\beta_1,\beta_2)=(0,\pm\alpha_2)$. Hence the $\pm\overline{\Gamma_1(N)}$-orbit of $\alpha$ is of size $2$ and represented by $\{(0,\pm\alpha_2)\}$. This also shows that there are $\varphi(N)/2$ many $\pm\overline{\Gamma_1(N)}$-orbits of this type.
    
    Case (2). From \eqref{orbit-alpha}, we deduce that $\beta$ is in the $\pm\overline{\Gamma_1(N)}$-orbit of $\alpha$ if and only if $\beta_1=\pm\alpha_1$. Hence the $\pm\overline{\Gamma_1(N)}$-orbit of $\alpha$ is of size $2N$ and represented by $\{(\pm\alpha_1,\gamma_2) :  \gamma_2\in\mathbb{Z}/N\mathbb{Z}\}$. 
    This also shows that there are $(N\varphi(N))/(2N)=\varphi(N)/2$ many $\pm\overline{\Gamma_1(N)}$-orbits of this type.

    Case (3). We first count the possible choices of $\alpha$:
    \begin{align*}
        \#\{\textrm{possible $\alpha$}\}=\#\{\textrm{possible $\alpha_1$}\}\times\#\{\textrm{possible $\alpha_2$}\}=\varphi\left(\frac{N}{d}\right)\times\varphi(d)\frac{N}{d}=\frac{N}{d}\varphi\left(\frac{N}{d}\right)\varphi(d).
    \end{align*}
    
    We will compute concrete representatives of the $\pm\overline{\Gamma_1(N)}$-orbits of the given $\alpha=(\alpha_1,\alpha_2)$. Since $\alpha_1$ is not a unit, we have
    \begin{align*}
        \#\{\pm\alpha_1\}=\left\{\begin{array}{ll}
        1 & \textrm{if $N$ is even and $\alpha_1=\frac{N}{2}$}, \\
        2 & \textrm{otherwise}.
        \end{array}\right.
    \end{align*}
    
    We first consider the exceptional case, $N$ is even and $\alpha_1=N/2$. For $t\in\mathbb{Z}/N\mathbb{Z}$ we have
    \begin{align*}
        t\alpha_1=\left\{\begin{array}{ll}
        \alpha_1 & \textrm{if $t\equiv1\bmod2$}, \\
        0 & \textrm{if $t\equiv0\bmod2$}.
        \end{array}\right.
    \end{align*}
    So $\beta_2\in\{\pm\alpha_2,\alpha_1\pm\alpha_2\}$. To compute $\# \lcrc{ \pm \alpha_2, \alpha_1 \pm \alpha_2}$, we remove duplicate expressions. Since $(\alpha_2,2)=1$, we have $\alpha_2\neq-\alpha_2$. Moreover, $\alpha_2=N/2-\alpha_2$ if and only if $2\alpha_2=N/2$. By assumption, $\alpha_2$ is relatively prime to $\alpha_1=N/2$. Hence $\alpha_2\in(\mathbb{Z}/\alpha_1\bZ)^\times$, which says that $2\alpha_2=N/2$ if and only if $N/2$ divides $2$. Since we have assumed that $N/2\neq 1$, we must have $N=4$. Consequently, we get
    \begin{align*}
        \#\{\pm\alpha_2,\alpha_1\pm\alpha_2\}=\left\{\begin{array}{ll}
        2 & \textrm{if $N=4$}, \\
        4 & \textrm{otherwise.}
        \end{array}\right.
    \end{align*}
    Hence the $\pm\overline{\Gamma_1(N)}$-orbit of $\alpha$ in this case can be represented as follows:
    \begin{align*}
        \{(\alpha_1,t\alpha_1\pm\alpha_2)\ |\ t\in\mathbb{Z}/N\mathbb{Z}\}=\{(\alpha_1,\pm\alpha_2),(\alpha_1,\alpha_1\pm\alpha_2)\}.
    \end{align*}
    This shows that the $\pm\overline{\Gamma_1(N)}$-orbit of $\alpha$ in this case is of size $4$.
    Since we have $d=N/2$ in this case, the number of $\pm\overline{\Gamma_1(N)}$-orbits becomes
    \begin{align*}
        \frac{1}{4}\times2\varphi(2)\varphi\left(\frac{N}{2}\right)=\frac{1}{2}\varphi\left(\frac{N}{d}\right)\varphi(d).
    \end{align*}

    For the other case, we have $\#\{\pm\alpha_1\}=2$. Since $\{t\alpha_1 : t\in\mathbb{Z}/N\mathbb{Z}\}=d\mathbb{Z}/N\mathbb{Z}$, we have
    \begin{align*}
        \{t\alpha_1\pm\alpha_2 : 
        t\in\mathbb{Z}/N\mathbb{Z}\}=\{\gamma_2\in\mathbb{Z}/N\mathbb{Z} :  \gamma_2\equiv\pm\alpha_2\bmod d\}.
    \end{align*}
    Hence the $\pm\overline{\Gamma_1(N)}$-orbit of $\alpha$ in this case can be represented as follows:
    \begin{align*}
        \{(\pm\alpha_1,\gamma_2) : \gamma_2\equiv\pm\alpha_2\bmod d\},
    \end{align*} 
    where the signs must be in the same order.
    Since the size of the set is $2N/d$, the number of $\pm \overline{\Gamma_1(N)}$-orbit of $\alpha$ is 
    \begin{align*}
        \frac{d}{2N}\times\frac{N}{d}\varphi\left(\frac{N}{d}\right)\varphi(d)=\frac{1}{2}\varphi\lbrb{\frac{N}{d}}\varphi(d)
    \end{align*}
    again. Considering cases (1)-(3), there are totally
    \begin{align*}
        \frac{\varphi(N)}{2} + \frac{\varphi(N)}{2} + \frac{1}{2}\sum_{\substack{d \mid N \\ d \neq 1, N}}\varphi\lbrb{\frac{N}{d}}\varphi(d)
        = \frac{1}{2}\sum_{\substack{d \mid N }}\varphi\lbrb{\frac{N}{d}}\varphi(d)
    \end{align*}
    many $\pm\overline{\Gamma_1(N)}$-orbits. This computation reproves the classical result on the number of cusps of the Riemann surface $X_1(N)(\mathbb{C})$, which is
    \begin{align*}
        \left| X_1(N)^{\mathrm{cusp}}(\bC) \right| = \frac{1}{2}\sum_{\substack{d \mid N }}\varphi\lbrb{\frac{N}{d}}\varphi(d)
    \end{align*}
    for $N > 4$ (cf. \cite[p. 107]{DS}).
    
From the computation so far, we get representatives of the set
    \begin{align} \label{eqn: surjective orbit}
     {\left\{\begin{array}{c} \textrm{surjective homomorphisms} \\(\mathbb{Z}/N\mathbb{Z})^2\rightarrow\mathbb{Z}/N\mathbb{Z}\end{array}\middle\}\right/(\pm \Gamma)}.
    \end{align} 
    It remains to compute the subset of elements in the set \eqref{eqn: surjective orbit} fixed by the Galois action. The action in Lemma \ref{lem:numcuspff} is determined by
    \begin{align*}
        \alpha^\sigma(1,0) = \alpha \lbrb{ \begin{pmatrix}
            1 & 0
        \end{pmatrix} \begin{pmatrix}
            1 & 0 \\ 0 & \chi(\sigma)
        \end{pmatrix} } = \alpha(1, 0), \qquad
        \alpha^\sigma(0,1) = \alpha \lbrb{ \begin{pmatrix}
            0 & 1
        \end{pmatrix} \begin{pmatrix}
            1 & 0 \\ 0 & \chi(\sigma)
        \end{pmatrix} } = \alpha(0, \chi(\sigma)).
    \end{align*}
    In other words, $\sigma$ acts on $\alpha_1$ trivially, and on $\alpha_2$ via the cyclotomic character $\sigma(\zeta_N) = \zeta_N^{\chi(\sigma)}$.

    When $N=5$, there are $4$ classes in the set (\ref{eqn: surjective orbit}) whose representatives are
    \begin{align*}
    \lcrc{(0, \pm 1)}, \quad \lcrc{(0, \pm 2)}, \quad \lcrc{(\pm 1, \beta_2) : \beta_2 \in \bZ/5\bZ}, \quad
    \lcrc{(\pm 2, \beta_2) : \beta_2 \in \bZ/5\bZ}.
     \end{align*}
     We note that the first two are case (1), and the last two are case (2).
     If $q \not\equiv \pm 1 \pmod{5}$, then there is $\sigma \in \Gal(\overline{\bF_q}/\bF_q)$ such that $\sigma(\zeta_5) = \zeta_5^e$ for a generator $e \in (\bZ/5\bZ)^\times$, since $[\bF_q(\zeta_N) : \bF_q] \leq 2$ if and only if $q \equiv \pm 1 \pmod{N}$.
     Since $\sigma$ permutes $\{(0,\pm1)\}$ and $\{(0,\pm2)\}$, we have
     \begin{align*}
         |X_1(5)^{\mathrm{cusp}}(\bF_q)| = \left\{ \begin{array}{lll}
         4     &  \textrm{if } q \equiv \pm 1 \pmod{5}, \\
         2     &  \textrm{if } q \not\equiv \pm 1 \pmod{5}.
         \end{array} \right.
    \end{align*}
    This proof also works for all prime $N$:
    since $N$ is prime, the cusps come from cases (1) or (2). The cusps from (2), whose representatives are $\lcrc{(\pm \beta_1, \beta_2) : \beta_2 \in \bZ/N\bZ}$ are always rational, and the ones from (1), whose representatives are $\lcrc{(0, \pm \beta_2) : \beta_2 \in (\bZ/N\bZ)^\times}$, are rational only when $\Gal(\bF_q(\zeta_N)/\bF_q) \cong \lcrc{\pm 1}$ which is equivalent to $q \equiv \pm 1 \pmod{N}$.
    

    Let $N = 2\ell$ for a prime $\ell > 2$. 
    Then, the representatives of \eqref{eqn: surjective orbit} are
    \begin{align*}
        &\lcrc{(0, \pm \gamma_2)} \quad \textrm{ for } \gamma_2 \in (\bZ/N\bZ)^\times, \\
        & \lcrc{(\pm \gamma_1, \gamma_2) : \gamma_2 \in \bZ/N\bZ} \quad \textrm{ for } \gamma_2 \in (\bZ/N\bZ)^\times, \\
        & \lcrc{(\pm 2 \gamma_1, \gamma_2) : \gamma_2 \equiv 1 \pmod{2}} \quad \textrm{ for } \gamma_1 \in \lcrc{1, 2, \cdots, \frac{\ell-1}{2}}, \\
        & \lcrc{(\ell, \pm \gamma_2), (\ell, \ell \pm \gamma_2) } \quad \textrm{ for } \gamma_1 \in \lcrc{1, 2, \cdots, \frac{\ell-1}{2}}.
    \end{align*}
    We note that the cusps in each line come from the case (1), (2), (3) with $d=2$, and (3) with $d = \ell$, respectively. 
    The second and third ones are always rational, and the first and fourth ones are rational when $\Gal(\bF_q(\zeta_N)/\bF_q) \cong \lcrc{\pm 1}$. Therefore,
    \begin{align*}
        |X_1(2\ell)^{\mathrm{cusp}}(\bF_q)| = \left\{
        \begin{array}{lll}
        2\varphi(2\ell)     & \textrm{if }  q \equiv \pm 1 \pmod{2\ell}, \\
        \varphi(2\ell)     &  \textrm{if }  q \not\equiv \pm 1 \pmod{2\ell}.
        \end{array}
        \right.
    \end{align*}
The remaining cases can each be shown by similar calculations.
Here we describe the representatives of cusps of the case (3) and the Galois group for the double-check.
For $N = 4$, it is easy to check that every cusp is rational.
When $N = 8$, the cusps of the case (3) are
\begin{align*}
    \lcrc{(\pm 2, \gamma_2) : \gamma_2 \equiv 1 \pmod{2}}, \qquad
    \lcrc{(4, \pm 1), (4, 2 \pm 1)}
\end{align*}
which arise from the cases $d=2$ and $d=4$, respectively.
Hence they are all rational, so we have
\begin{align*}
        |X_1(8)^{\mathrm{cusp}}(\bF_q)| = \left\{
        \begin{array}{lll}
        6     & \textrm{if } q \equiv \pm 1 \pmod{8},  \\
        4     & \textrm{if } q \not\equiv \pm 1 \pmod{8}.  
        \end{array}
        \right.   
\end{align*}
    When $N = 9$, the cusps of the case (3) are
    \begin{align*}
        & \lcrc{(3, \gamma_2) : \gamma_2 \equiv 1 \pmod{3}} \cup \lcrc{(6, \gamma_2) : \gamma_2 \equiv 2 \pmod{3}}, \\ 
        & \lcrc{(3, \gamma_2) : \gamma_2 \equiv 2 \pmod{3}} \cup \lcrc{(6, \gamma_2) : \gamma_2 \equiv 1 \pmod{3}},
    \end{align*}
    and
    \begin{align*}
        \Gal(\bF_q(\zeta_9)/\bF_q) = \left\{
        \begin{array}{lll}
        e     & \textrm{if }  q \equiv 1 & \pmod{9},  \\
        \bZ/2\bZ     &  \textrm{if }  q\equiv 8 &\pmod{9}, \\
        \bZ/3\bZ     &  \textrm{if }  q\equiv 4, 7 &\pmod{9}, \\
        \bZ/6\bZ     &  \textrm{if }  q\equiv 2, 5 &\pmod{9}. \\
        \end{array}
        \right.
    \end{align*}
    Hence, we have
    \begin{align*}
        |X_1(9)^{\mathrm{cusp}}(\bF_q)| = \left\{
        \begin{array}{lll}
        8     & \textrm{if }   q \equiv 1 &\pmod{9},  \\
        6     & \textrm{if }  q \equiv 8 &\pmod{9}, \\
        5   & \textrm{if }  q \equiv 4, 7 &\pmod{9}, \\
        3   & \textrm{if }  q \equiv 2, 5 &\pmod{9}.
        \end{array}
        \right.
    \end{align*}

When $N = 12$, there are $\varphi(12)/2$ cusps of the case (1), which are rational only when $\Gal(\bF_q(\zeta_N)/\bF_q) \cong \lcrc{\pm 1}$, and the same number of cusps of the case (2) which are rational.
    The case (3) is subdivided into four cases, $d = 2, 3, 4, 6$.
    For each subcase, the representatives are
    \begin{align*}
        &\lcrc{(\pm 2, \gamma_2) : \gamma_2 \equiv 1 \pmod{2}}, \\
        &\lcrc{(3, \gamma_2) : \gamma_2 \equiv 1 \pmod{3}} \cup \lcrc{(9, \gamma_2) : \gamma_2 \equiv 2 \pmod{3}}, \\
        &\lcrc{(3, \gamma_2) : \gamma_2 \equiv 2 \pmod{3}} \cup \lcrc{(9, \gamma_2) : \gamma_2 \equiv 1 \pmod{3}}, \\
        &\lcrc{(4, \gamma_2) : \gamma_2 \equiv 1 \pmod{4}} \cup \lcrc{(8, \gamma_2) : \gamma_2 \equiv 2 \pmod{4}}, \\
        &\lcrc{(4, \gamma_2) : \gamma_2 \equiv 2 \pmod{4}} \cup \lcrc{(8, \gamma_2) : \gamma_2 \equiv 1 \pmod{4}}, \\
        & \lcrc{(6, \pm 1), (6, \pm 5)}.
    \end{align*}
    Since
    \begin{align*}
        \Gal(\bF_q(\zeta_{12})/\bF_q) = \left\{
        \begin{array}{lll}
        e     & q \equiv 1 & \pmod{12},  \\
        \bZ/2\bZ     &  q\equiv 5, 7, 11 & \pmod{12},
        \end{array}
        \right.
    \end{align*}
    we have
    \begin{align*}
        |X_1(12)^{\mathrm{cusp}}(\bF_q)| = \left\{
        \begin{array}{lll}
        10     & q \equiv 1 & \pmod{12},  \\
        6     & q \not\equiv 1  & \pmod{12}.
        \end{array}
        \right.
    \end{align*}
    This completes the verification for all cases.
\end{proof}

We believe that Corollary \ref{cor:Numcusps} is known to experts, but we included the detailed proof for comparison with the previous result \cite[Proposition 2.2]{CJ2} in Remark \ref{rem:cuspcompareCJ}. 
We also attempted to generalize Corollary \ref{cor:Numcusps} to arbitrary $N$, but the resulting statement does not appear to admit a simple description. 
For example, when $N = \ell^2$ for an odd prime $\ell$, the case (3) for $d = \ell$ gives representatives
\begin{align*}
    \lcrc{(\ell, \gamma_2) : \gamma_2 \equiv i \pmod{\ell}} \cup 
    \lcrc{(\ell, \gamma_2) : \gamma_2 \equiv \ell - i \pmod{\ell}}
\end{align*}
for $i = 1, \cdots, \ell-1$. 
Since $\Gal(\bF_q(\zeta_{\ell^2})/\bF_q)$ is a quotient of $(\bZ/\ell^2\bZ)^\times \cong \bZ/\ell(\ell-1)\bZ$, the number of subcases increases as the number of prime divisors of $(\ell-1)$ grows. 
Instead, we give a general statement for the easier cases $N = \ell, 2\ell$, and include the special cases $N = 4, 8, 9, 12$ to cover the cases in Mazur's torsion theorem and \cite[Proposition 2.2]{CJ2}.


\begin{remark} \label{rem:cuspcompareCJ}
Corollary \ref{cor:Numcusps} generalizes \cite[Proposition 2.2]{CJ2}, which computes
\begin{align*}
    \sum_{\substack{J = (A, B)\in \bF_p^2 \\ 4A^3 + 27B^2 = 0}} |W_{G, J}|.
\end{align*}
By the definition of $W_{G, J}$, the sum counts the number of pairs $(a, b)$ in $\bF_p^2$ satisfying $(f_G(a, b), g_G(a, b)) = (A, B)$ for $4A^3 + 27B^2 \equiv 0 \pmod{p}$.
On the other hand, if $q$ is a $p$-power coprime to $6$, then there are exactly two cusps in $\cX_{\bF_q} \cong \cP(4, 6)_{\bF_q}$ which are $[A:B]$ satisfying $4A^3 + 27B^2 \equiv 0 \pmod{p}$ by definition of $\cX_{\Gamma, \bF_q}^{\mathrm{cusp}}$ in (\ref{eqn: def cX cusp}).
Thus, the above sum equals
\begin{align*}
    \sum_{E_z \in \cX_{\bF_q}^{\mathrm{cusp}}(\bF_q)} |\phi_{\Gamma}^{-1}(z)| = |\cX_{\Gamma, \bF_q}^{\mathrm{cusp}}(\bF_q)| = |X_{\Gamma, \bF_q}^{\mathrm{cusp}}(\bF_q)|
\end{align*}
when $\cX_{\Gamma}$ is representable.
Here we identify $\phi_{\Gamma}$ with $(f_G,g_G)$ as in Remark \ref{rmk:Phivsphi}.
In fact, the computation of $\bF_p$-cusps of $X_1(N)$ in Corollary \ref{cor:Numcusps} agrees with \cite[Proposition 2.2]{CJ2}.
Finally, we note that the prime condition in \cite[Proposition 2.2]{CJ2} for $N=7$, namely $\gamma_7 \in (\bF_p[\sqrt{-3}]^\times)^3$, is equivalent to the condition $p \equiv \pm 1 \pmod{7}$.
Therefore, our generalization recovers their result as a special case.
\end{remark}

Using Proposition \ref{prop:EKGamz}, one can also compute the probability of the local condition as in \cite[Theorem 1.4]{CJ1} and \cite[Theorem 1.1.2]{Phi2}.
For the multiplicative reduction condition, we define
\begin{align*}
    \cE_{K, \Gamma, \fp}^{\mathrm{mult}}(X) \vcentcolon = \lcrc{E \in \cE_{K, \Gamma}(X) : E \text{ has multiplicative reduction at } \fp }.
\end{align*}
\begin{theorem} \label{thm:cuspthm}
Let $\Gamma$ be a congruence subgroup of level $N$ such that $\mathcal{X}_\Gamma \cong \bP^1$ over $\Spec\bZ[1/6N]$. For every prime $\fp$ not dividing $6N$,
\begin{align*}
    \lim_{X \to \infty} \frac{|\cE_{K, \Gamma, \fp}^{\mult}(X)|}{|\cE_{K, \Gamma}(X)|} = 
    \frac{|\cX_{\Gamma}^{\mathrm{cusp}}(\bF_q)|}{|\cX_{\Gamma}(\bF_q)|}.
\end{align*}
\end{theorem}
\begin{proof}
Elliptic curves with multiplicative reduction at $\fp$ are mapped to the cusps in $\cX(\bF_q)^*$ under $\psi_\fp$. 
Therefore, by Proposition \ref{prop:EKGamz},
\begin{align} \label{mult count}
    |\cE_{K, \Gamma, \fp}^{\mult}(X)| 
    &= \frac{1}{q+1} \lbrb{ \sum_{z \in \cX^{\mathrm{cusp}}(\bF_q)} |\phi_{\Gamma}^{-1}(z)| } \kappa \cdot X^{\frac{2}{e(\Gamma)}} + O\lbrb{ {\left(1+\frac{1}{q} \right)|\cX^{\mathrm{cusp}}(\bF_q)|}  X^{\frac{2d - 1}{de(\Gamma)}}\log X}.
\end{align}
Since the inverse image of the cusps of $\cX$ under $\phi_\Gamma$ is precisely the set of cusps of $\cX_{\Gamma}$ by Lemma \ref{lem:cuspGam}, we have
\begin{align*}
    \lim_{X \to \infty} \frac{|\cE_{K, \Gamma, \fp}^{\mult}(X)| }{|\cE_{K, \Gamma}(X)| }
    = \frac{1}{q+1} \sum_{z \in \cX^{\mathrm{cusp}}(\bF_q)} |\phi_{\Gamma}^{-1}(z)| 
    = \frac{|\cX_{\Gamma}^{\mathrm{cusp}}(\bF_q)|}{q+1}.
\end{align*}
In our case we have $|\cX_{\Gamma}(\bF_q)| = |\bP^1(\bF_q)| = q+1$.
\end{proof}

\begin{remark} \label{rmk:cusprmk}
This is a generalization of \cite[Corollary 6]{CJ2}, which states that the probability of multiplicative reduction at $p$ is proportional to the number of cusps of $X_1(N)/\bC$ for a set of primes $p$ of positive density. 
Now, by Lemma \ref{lem:cuspGam} and Remark \ref{rem:cuspcompareCJ}, we can consider all primes $\fp$ not dividing $N$.
\end{remark}

From now on, we separate the multiplicative reduction into the split and non-split cases. 
Along the way, we also correct the errors in \cite[Corollary 3.13]{CJ2} and the split multiplicative reduction part of \cite[Theorem 3.7]{CJ2}, as stated in Corollary \ref{cor: Cor of CJ2}.

\begin{lemma} \label{lem: sp nsp cusp}
Let $q$ be a prime power coprime to $6$, and let $\alpha_0 \in \bF_q^\times$ be a nonsquare. 
The cusps of $\cX(\bF_q)$ correspond under the identification $\cX \cong \cP(4,6)$ to the two points $[-3,2]$ and $[-3\alpha_0^2,2\alpha_0^3]$ in $\cP(4,6)(\bF_q)$.  \\
(i) Suppose that $E \in \cY(K_{\fp})$ satisfies $\psi_{\fp}(E) = [-3,2]$.  
Then $E/K_\fp$ has split multiplicative reduction if and only if $3$ is a square in $\bF_q^\times$.  \\
(ii) Suppose that $E \in \cY(K_{\fp})$ satisfies $\psi_{\fp}(E) = [-3\alpha_0^2,2\alpha_0^3]$.  
Then $E/K_\fp$ has split multiplicative reduction if and only if $3$ is a non-square in $\bF_q^\times$.
\end{lemma}

\begin{proof}
By definition (\ref{eqn: def cX cusp}), a point $[A,B] \in \cP(4,6)(\bF_q)$ lies in the image of $\cX^{\mathrm{cusp}}(\bF_q)$ if and only if $4A^3 + 27B^2 = 0$ in $\bF_q$ provided $(q, 6) = 1$ (cf. Remark \ref{rem:cuspcompareCJ}).
The nonzero solutions of this equation are parametrized by $(-3\alpha^2, \pm 2\alpha^3)$ with $\alpha \in \bF_q^\times$.
Thus, when $(q,6)=1$, we may choose four representatives of cusps in $\cP(4,6)(\bF_q)$:
\begin{align*}
    [-3, \pm 2], \qquad [-3 \alpha_0^2, \pm 2 \alpha_0^3],
\end{align*}
where $\alpha_0$ is a nonsquare in $\bF_q^\times$.
Among these four points, two coincide: if $-1$ is a square in $\bF_q$ then $[-3,2]=[-3,-2]$, while if $-1$ is a nonsquare then $[-3,2]=[-3\alpha_0^2,-2\alpha_0^3]$.
In either case, we obtain exactly two distinct cusps, namely
\begin{align*}
    \cX^{\mathrm{cusp}}(\bF_q) \cong \lcrc{[-3, 2], [-3 \alpha_0^2, 2 \alpha_0^3]}.
\end{align*}

Recall that a generalized elliptic curve $E_z \in \cX^{\mathrm{cusp}}(\bF_q)$ corresponds to split (resp. non-split) multiplicative reduction if and only if the slope at its singular point lies in (resp. does not lie in) $\bF_q$.
Since 
\begin{align*}
    y^2 - x^3 + 3\alpha^2x - 2\alpha^3
    = (y - \sqrt{3\alpha}(x - \alpha) )(y + \sqrt{3\alpha}(x - \alpha) ) - (x - \alpha)^3,
\end{align*}
the slope of a nodal curve $y^2 = x^3 - 3\alpha^2 x + 2\alpha^3$
at the singular point $(\alpha, 0)$ is $\sqrt{3\alpha}$.
Therefore, if $E_z \in \cX^{\mathrm{cusp}}(K_{\fp})$ reduces to $[-3,2] \in \cP(4,6)(\bF_q)$ (resp. $[-3\alpha_0^2, 2\alpha_0^3]$), then it has split multiplicative reduction if and only if $3$ is a square (resp. a non-square) in $\bF_q^\times$.
\end{proof}

The definitions of $\cE_{K,\Gamma,\fp}^{\mathrm{split}}(X)$ and $\cE_{K,\Gamma,\fp}^{\mathrm{nonsplit}}(X)$ are analogous to that of $\cE_{K,\Gamma,\fp}^{\mathrm{mult}}(X)$.

\begin{corollary}[{Correction of \cite[Corollary 3.13]{CJ2} }]
\label{cor: Cor of CJ2}
\begin{align*}
    \lim_{X \to \infty} \frac{|\cE_{\bQ, \Gamma_1(3), p}^{\mathrm{split}}(X)|}{|\cE_{\bQ, \Gamma_1(3),p}^{\mathrm{mult}}(X)|}
    = \left\{
\begin{array}{ll}
    1 & \textrm{if } p \equiv 1, 7 \pmod{12}, \\
    \frac{1}{2} & \textrm{if } p \equiv 5, 11 \pmod{12}.
\end{array}
    \right.
\end{align*}
\end{corollary}
\begin{proof}
By \cite[Theorem 3.7]{CJ2}, there exists an explicit constant $c_3$ such that
\begin{align*}
    |\cE_{\bQ, \Gamma_1(3)}(X)| = c_3X^{\frac{1}{3}}  + O(X^{\frac{1}{4}}).
\end{align*}
Let
\begin{align*}
    w_p \vcentcolon = \left\{
        \begin{array}{lll}
        2(p-1)    & \textrm{if } p \equiv 1 &\pmod{12}, \\
        p-1         & \textrm{if } p \equiv 5, 11 &\pmod{12}, \\
        0       & \textrm{if } p \equiv 7 & \pmod{12}.
        \end{array}
    \right.
\end{align*}
Then by \cite[(5), Proposition 3.6, Theorem 3.7]{CJ2}, 
\begin{align*}
    |\cE_{\bQ, \Gamma_1(3), p}^{\mathrm{mult}}(X)| &= c_3 \frac{2(p-1)}{p^2} \frac{p^{4}}{p^{4}-1} X^{\frac{1}{3}} + O\lbrb{ X^{\frac{1}{4}} + pX^{\frac{1}{12}} },
\end{align*}
and
\begin{align} \label{eqn: CJ2 error}
\left| \lcrc{E \in \cE_{\bQ, \Gamma_1(3)}(X) : \,\,
\psi_p(E) = [-3, 2] 
} \right| &= c_3 \frac{w_p}{p^2} \frac{p^{4}}{p^{4}-1} X^{\frac{1}{3}} + O\lbrb{\frac{w_p}{p} X^{\frac{1}{4}} + X^{\frac{1}{12}} }.
\end{align}
Hence, Lemma \ref{lem: sp nsp cusp} together with the direct computation of the quadratic residue\footnote{In \cite{CJ2}, the authors mistakenly identified the right-hand side of (\ref{eqn: CJ2 error}) as $|\cE_{\bQ, \Gamma_1(3), p}^{\mathrm{mult}}(X)|$. The other parts of \cite[Theorem 3.7]{CJ2} are correct, and we use the multiplicative reduction part in our proof. }
\begin{align*}
    \lbrb{\frac{3}{p}} = \left\{ \begin{array}{lll}
       1  & \textrm{if } p \equiv 1, 11 &\pmod{12}, \\
       -1  & \textrm{if } p \equiv 5, 7 &\pmod{12},
    \end{array} \right.
\end{align*}
yields the result.
\end{proof}


As of August 2025, the LMFDB \cite{LMFDB} provides a complete list of elliptic curves over $\bQ$ with conductor less than or equal to $500{,}000$. 
Let $\mathrm{E}_{\bQ, \bZ/3\bZ, p}^{\mathrm{split}}$ and $\mathrm{E}_{\bQ, \bZ/3\bZ, p}^{\mathrm{nonsplit}}$ denote the sets of elliptic curves with conductor $\leq 500{,}000$, torsion subgroup $\bZ/3\bZ$, and split or non-split multiplicative reduction at $p$, respectively. 
Then we obtain the following table, which provides a numerical test of Corollary \ref{cor: Cor of CJ2}. 
Note that $|\mathrm{E}_{\bQ, \bZ/3\bZ}| = 51{,}405$.
\begin{align*}
    \begin{array}{|c|c|c|c|} \hline
        & |\mathrm{E}_{\bQ, \bZ/3\bZ, p}^{\mathrm{split}}| & |\mathrm{E}_{\bQ, \bZ/3\bZ, p}^{\mathrm{nonsplit}}| \\ \hline 
        p=5  &  9725   & 10500 \\ \hline 
        p=7  & 15476   & 0     \\ \hline 
        p=11 & 4603    & 5041  \\ \hline 
        p=13 & 8141    & 0     \\ \hline 
    \end{array}
\end{align*}
These counts are consistent with Corollary \ref{cor: Cor of CJ2}, confirming the predicted distribution between split and non-split multiplicative reduction.
Strictly speaking, the above table is not an exact numerical test of Corollary \ref{cor: Cor of CJ2}, for the following two reasons:
\begin{itemize}
    \item We used the naive height order in Corollary \ref{cor: Cor of CJ2}, but the LMFDB list is given in the conductor order.
    \item We used $\cE_{\bQ, \Gamma_1(3)}(X)$ as a set of elliptic curves with $\Gamma_1(3)$-level structure, whereas here we only consider elliptic curves whose torsion subgroup is isomorphic to $\bZ/3\bZ$.
\end{itemize}
However, there are in total $6759$ (resp. $20$, $17$) elliptic curves with conductor $\leq 500{,}000$ and torsion subgroup $\bZ/6\bZ$ (resp. $\bZ/9\bZ$, $\bZ/12\bZ$), which is too small to affect the proportion.

\medskip

Corollary \ref{cor: Cor of CJ2} provides an example (for $\Gamma=\Gamma_1(3)$) in which the probabilities of split and non-split multiplicative reduction at a prime $p$ differ for elliptic curves with prescribed level structure $\Gamma$. 
We expect that this phenomenon disappears after base change to a finite extension, provided $\fp \nmid 6N$. 
When $\cX_{\Gamma}$ is representable, we prove this in what follows.

\begin{corollary} \label{cor:spnonsp}
Let $\Gamma$ be a congruence subgroup of level $N$ such that $\mathcal{X}_\Gamma \cong \bP^1$ over $\Spec \bZ[1/6N]$.
For any algebraic extension $K/\mathbb{Q}(\zeta_N)$ and any prime $\mathfrak{p}$ of $K$ not dividing $6N$, we have
\begin{align*}
    \lim_{X \to \infty} \frac{|\cE_{K, \Gamma, \fp}^{\mathrm{split}}(X)|}{|\cE_{K, \Gamma}(X)|}
=\lim_{X \to \infty} \frac{|\cE_{K, \Gamma, \fp}^{\mathrm{nonsplit}}(X)|}{|\cE_{K, \Gamma}(X)|}.
\end{align*}
\end{corollary}

\begin{proof}
Here, we use $\kappa(\fp)$ instead of $\bF_q$.
By Lemma \ref{lem: sp nsp cusp}, there are two cusps $z_1 = z_1(q)$ and $z_2 = z_2(q)$ of $\cX(\kappa(\fp))$ which correspond to split and non-split multiplicative reduction respectively.
In other words, we have $|\cE_{K, \Gamma, \fp}^{\mathrm{split}}(X)| = |\cE_{K, \Gamma, \fp}^{z_1}(X)|$ and $|\cE_{K, \Gamma, \fp}^{\mathrm{nonsplit}}(X)| = |\cE_{K, \Gamma, \fp}^{z_2}(X)|.$
By Proposition \ref{prop:EKGamz}, we have
\begin{align*}
\lim_{X \to \infty} \frac{|\cE_{K, \Gamma, \fp}^{\mathrm{split}}(X)|}{|\cE_{K, \Gamma}(X)|}
= \frac{|\phi^{-1}_{\Gamma}(z_1)|}{q+1} ,
\qquad    \lim_{X \to \infty} \frac{|\cE_{K, \Gamma, \fp}^{\mathrm{nonsplit}}(X)|}{|\cE_{K, \Gamma}(X)|}
= \frac{|\phi^{-1}_{\Gamma}(z_2)|}{q+1}.
\end{align*}
From the discussion in section \ref{subsec:cusps}, we get the following Cartesian squares of algebraic stacks:
\begin{align*}
    \xymatrix{
    \mathcal{X}_{\Gamma,\kappa(\mathfrak{p}),z_i} \ar[d] \ar[r] & \mathcal{X}_{\Gamma,\kappa(\mathfrak{p})}^\mathrm{cusp} \ar[d] \ar@{^(->}[r] & \mathcal{X}_{\Gamma,\kappa(\mathfrak{p})} \ar[d]^-{\phi_{\Gamma,\kappa(\mathfrak{p})}} \\
    \Spec\kappa(\mathfrak{p}) \ar[r]_-{z_i} & \mathcal{X}_{\kappa(\mathfrak{p})}^\mathrm{cusp} \ar@{^(->}[r] & \mathcal{X}_{\kappa(\mathfrak{p})}\rlap{\ .}
    }
\end{align*}
Since $\phi_{\Gamma,\kappa(\mathfrak{p})}$ is representable, the fiber $\mathcal{X}_{\Gamma,\kappa(\mathfrak{p}),z_i}$ is a scheme. Note that the right column of the above diagram is surjective as a map of algebraic stacks because it is a finite dominant map into an irreducible stack $\mathcal{X}_{\kappa(\mathfrak{p})}\cong\mathcal{P}(4,6)_{\kappa(\mathfrak{p})}$. Since surjective maps are stable under base change (cf. \cite[\href{https://stacks.math.columbia.edu/tag/04ZU}{Tag 04ZU}]{Stacks}), each column in the above diagram is surjective. Since all the points of $\mathcal{X}_{\Gamma,\kappa(\mathfrak{p})}^\mathrm{cusp}$ are $\kappa(\mathfrak{p})$-rational (cf. \cite[VII.2]{DR73}, and \cite[10.9.1]{KM85}), the set $\mathcal{X}_{\Gamma,\kappa(\mathfrak{p}),z_i}(\kappa(\mathfrak{p}))$ is nonempty. Consequently, the function $\phi_{\Gamma,\kappa(\mathfrak{p})}:\mathcal{X}_{\Gamma,\kappa(\mathfrak{p})}^\mathrm{cusp}(\kappa(\mathfrak{p}))\rightarrow\mathcal{X}^\mathrm{cusp}_{\kappa(\mathfrak{p})}(\kappa(\mathfrak{p}))=\{z_1,z_2\}$ is surjective. Since $\phi_{\Gamma,\kappa(\mathfrak{p})}$ is finite \'etale on the substacks of cusps (cf. \cite[Th\'eor\`eme IV.3.4]{DR73}), we deduce from the $\kappa(\mathfrak{p})$-rationality that $|\phi_\Gamma^{-1}(z_1)|=\deg\phi_\Gamma=|\phi_\Gamma^{-1}(z_2)|$ for any $\mathfrak{p}$ not dividing $6N$.
\end{proof}

\section{Moments of traces of the Frobenius} \label{sec:classnum}

\subsection{Class number: generalization}
Let $H(D)$ be the Hurwitz class number of discriminant $D$.
For a given integer $a$ satisfying the Weil bound $[-2\sqrt{p},2\sqrt{p}]$, the number of short Weierstrass equations of elliptic curves over $\bF_p$ with trace of Frobenius $a$ is
\begin{align*}
    \frac{p-1}{2}H(a^2-4p),
\end{align*}
(cf. \cite[Theorem 14.18]{Cox13}).
Building on the computations of the previous section, we propose the following generalization of the Hurwitz class number. This invariant measures the proportion of isomorphism classes of elliptic curves over $\bF_q$ with trace of Frobenius $a$ whose group structure is related to the $\Gamma$-level structure.

Recall that $\phi_{\Gamma} : \cX_{\Gamma} \to \cX$ is the forgetful functor, that $\wt$ is the weight function defined in Definition \ref{def:wtz}, and that $E_z \in \cY(\bF_q)$ is the elliptic curve corresponds to a point $z \in \cP(4, 6)(\bF_q)$.
\begin{definition} \label{def:H Gamma}
For a congruence subgroup $\Gamma$ and an integer $a$ satisfying the Weil bound $[-2\sqrt{q}, 2\sqrt{q}]$, we define
\begin{align*}
    H_{\Gamma}(a, q) \vcentcolon= \frac{1}{q^2}  \sum_{\substack{E_z \in \cY(\bF_q) \\ a_q(E_z) = a}} \sum_{\widetilde{z} \in \phi^{-1}_{\Gamma}(z)} \wt(\widetilde{z}).
\end{align*}
\end{definition}
When $\cX_{\Gamma, \bF_q} \cong \bP_{\bF_q}^1$, we have $\mathrm{wt}(\widetilde{z}) = 1$ for any $\widetilde{z} \in \bP^1(\bF_q)$ by Lemma \ref{lem:wtXast}.
In this case,
\begin{align} \label{eqn:HGamma def rep}
    H_{\Gamma}(a, q) = \frac{q-1}{q^2}\sum_{\substack{E_z \in \cY(\bF_q) \\ a_q(E_z) = a}}
    |\phi_{\Gamma}^{-1}(z)|.
\end{align}

\begin{remark} \label{rmk: HGam and HG}
When $\Gamma$ is a congruence subgroup corresponding to a torsion subgroup of $E/\bQ$,
it is natural to view $H_{\Gamma}(a,q)$ as a refinement of the function
\begin{align*}
    H_{G}(a, p) = \sum_{\substack{J = (A, B) \in \bF_p^2 \\ a_p(E_J) = a \\ 4A^3 + 27B^2 \not\equiv 0 \pmod{p}}} |W_{G, J}|
\end{align*}
defined in \cite[(7)]{CJ2}.
We recall that the definition of $W_{G, J}$ is given in (\ref{eqn:WGJinCJ}).
We note that the sum over $E_z \in \cY(\bF_q)$ in the definition of $H_{\Gamma}(a, q)$ is a sum taken over an $\bF_q$-isomorphism classes of elliptic curves, but the sum over $J \in \bF_p^2$ in $H_G(a, p)$ is a sum over a short Weierstrass equation.
By Lemma \ref{lem:WGJcomapare},
\begin{align*}
    H_{G}(a, p)
    = \sum_{\substack{E_z \in \cY(\bF_p) \\ a_p(E_z) =a }}
    \sum_{\substack{J \in \bF_p^2 \\ E_J \cong E_z}} |W_{G, J}|
    = (p-1) \sum_{\substack{E_z \in \cY(\bF_p) \\ a_p(E_z) =a }} |\phi_{\Gamma}^{-1}(z)|.
\end{align*}
Hence, in this case,
\begin{align*}
    H_{\Gamma}(a, p) = \frac{H_G(a, p)}{p^2}.
\end{align*}
\end{remark}

By definition, (\ref{eqn:EKgama}) gives 
\begin{align} \label{eqn:EKgamaref}
    |\cE_{K, \Gamma, \fp}^a(X)| = \frac{q^2}{q^2-1} \cdot \kappa \cdot  H_{\Gamma}(a, q) \cdot X^{\frac{2}{e(\Gamma)}} + O_{\Gamma, K}\lbrb{  \lbrb{\sum_{a_{\fp}(E_z) = a} 1} X^{\frac{2d - 1}{de(\Gamma)}}\log X}.
\end{align}

For $R = 0, 1,2$, the goal of this section is to give an asymptotic of
\begin{align*}
    \sum_{|a| \leq 2\sqrt{q}} a^R H_{\Gamma}(a, q) 
\end{align*}
which are analogues of \cite[(8), (9), (10)]{CJ2}.
The first one follows from section \ref{subsec:cusps}.
In fact, it also holds in the more general case $\cX_{\Gamma} \cong \cP(u)$, not only when $\cX_{\Gamma} \cong \bP^1$.

\begin{lemma} \label{lem:Hmoment0}
Let $\Gamma$ be a congruence subgroup of level $N$ satisfying $\cX_{\Gamma} \cong \cP(u)$ over $\Spec \bZ[1/N]$, and let $q$ be a prime power satisfying $(q, N) = 1$.
Then
\begin{align*}
\sum_{|a|\leq 2\sqrt{q}}H_{\Gamma}(a, q) = 1 +  O_{\Gamma}(q^{-1}).
\end{align*}
\end{lemma}
\begin{proof}
By definition of $H_{\Gamma}(a, q)$ and Lemma \ref{lem:cuspGam}, we have
\begin{align*}
    \sum_{|a| \leq 2\sqrt{q} } H_{\Gamma}(a, q) = \frac{1}{q^2} \sum_{|a| \leq 2\sqrt{q} }    \sum_{\substack{E_z \in \cY(\bF_q) \\ a_q(E_z) = a}} \sum_{\widetilde{z} \in \phi_{\Gamma}^{-1}(z)} \wt(\widetilde{z})
    =\frac{1}{q^2}\sum_{\substack{E_z \in \cY(\bF_q) }} \sum_{\widetilde{z} \in \phi_{\Gamma}^{-1}(z)} \wt(\widetilde{z})
    = \frac{1}{q^2} \sum_{E_z \in \cY_{\Gamma}(\bF_q)} \wt(\widetilde{z}).
\end{align*}
By the definition of the weight, we have
\begin{align*}
    \sum_{E_z \in \cX_{\Gamma}(\bF_q)} \wt(\widetilde{z}) = q^2 - 1.
\end{align*}
By Lemma \ref{lem:wtXast}, we have $\wt(\widetilde{z}) \leq q-1$ for any $\widetilde{z} \in \phi_{\Gamma}^{-1}(z)$ when $\cX_{\Gamma} \cong \cP(u)$.
By Lemma \ref{lem:cuspbound},
\begin{align*}
    \sum_{E_z \in \cX_{\Gamma}^{\mathrm{cusp}}(\bF_q)} \wt(\widetilde{z}) 
    \leq (q-1)|\cX_{\Gamma}^{\mathrm{cusp}}(\bF_q)| \ll O_{\Gamma}(q).
\end{align*}
So we obtain the result.
\end{proof}

\subsection{Moments of traces of Frobenius}

We first recall the notations and the result of Kaplan--Petrow \cite[Theorem 3]{KP2}.
In this section, we write $q=p^v$ where $p$ is prime and $v$ is a non-negative integer.
Let $M, N, m, n, c, n_1, \lambda, \tau$ be positive integers satisfying $ \tau \mid \lambda n_1$ and $\lambda \mid (d^2q-1, n_1)$, and let $S_k(\Gamma)$ be the space of cusp forms of weight $k$.
Let 
\begin{enumerate}
\item $\Gamma(N, M) \vcentcolon= \Gamma_1(N) \cap \Gamma_0( NM)$ for $M \mid N$ following \cite[(1.4)]{KP2} and \cite[(1-4)]{Petrow18},
\item $\delta(m,n)$ is the indicator function of $m=n$ and $\delta_c(m,n)$ is that of $m \equiv n \pmod{c}$ (if $q$ is not a square, then we define $\delta_{c}(q^{1/2}, \bullet )=0$ for a natural number $\bullet$),
\item $y_i$ is the unique element of $(\mathbb{Z}/(n_1\lambda /g) \mathbb{Z})^\times$ such that $y_i \equiv p^i$ (mod $\tau$) and $y_i\equiv p^{v-i}$ (mod $n_1\lambda/\tau$),
\item $\varphi(n)=n\prod_{ p \mid n}(1-1/p)$, $\psi(n)=n\prod_{ p \mid n}(1+1/p)$, and $\phi(n)=n \prod_{p\mid n}(-\varphi(p))$,
\item $T_q$ (resp. $\lara{d}$) are the Hecke (resp. diamond) operators on $S_k(\Gamma(n_1, \lambda))$,
\item $\sigma$ is the sum of the divisor function.
\end{enumerate}
For an integer $n_1$ and $\lambda \mid (d^2q -1, n_1)$, let
\begin{align*}
    T_{n_1,\lambda}(q,d)=\frac{\psi(n_1^2/\lambda^2)\varphi(n_1/\lambda)}{\psi(n_1^2)}(-T_{\trace}+T_{\id}-T_{\hyp}+T_{\dual}),
\end{align*}
with
\begin{align*}
T_{\trace}&=\frac{1}{\varphi(n_1)}\Tr(T_q \langle d \rangle \mid S_k(\Gamma(n_1,\lambda))),\\
T_{\id} &=\frac{k-1}{24}q^{k/2-1}\psi(n_1 \lambda) \left( \delta_{n_1}\left(q^{1/2},d^{-1} \right) +(-1)^k \delta_{n_1}(q^{1/2},-d^{-1}) \right),\\
 T_{\hyp}& =\frac{1}{4}\sum_{i=0}^v \min (p^i, p^{v-i})^{k-1} \sum_{\substack{ \tau \mid n_1 \lambda \\ g\vcentcolon=(\tau, n_1\lambda/\tau)|p^i-p^{v-i}}} \frac{\varphi(g)\varphi(n_1(\lambda,g)/g)}{\varphi(n_1)} \\
 & \times \left( \delta_{n_1(\lambda,g)/g}(y_i,d^{-1})+(-1)^k\delta_{n_1(\lambda,g)/g}(y_i,-d^{-1}) \right),\\
 T_{\dual}&=\frac{\sigma(q)}{\varphi(n_1)}\delta(k,2).
\end{align*}

We emphasize that $\Gamma(N, M)$ is not $\Gamma_1(M, N)$; that $q = p^v$ in this section, even though $v$ denotes a prime of $K$ in section \ref{sec:counting}; that $d$ denotes a Hecke operator rather than the degree of a number field; and that $\phi$ should be distinguished from $\phi_{\Gamma}$ defined in section \ref{sec:Prestack} and from the test function $\phi$ in section \ref{sec:average}.

We define
\begin{align} \label{eqn: def bE}
	\bE_q(a_E^R \Phi_A) \vcentcolon= \frac{1}{q} \sum_{\substack{E \in  \cY(\bF_q) \\ A \hookrightarrow E(\bF_q)}} \frac{a_q(E)^R}{|\Aut_{\bF_q}(E)|}
\end{align}
where $A$ is a finite abelian group.
Let $U_j$ be the Chebyshev polynomials of the second kind, defined by
\begin{align} \label{eqn:defU}
    U_0(t) = 1, \qquad U_1(t) = 2t, \qquad U_{j+1}(t) = 2tU_{j}(t) - U_{j-1}(t)
\end{align}
for $j \geq 1$, and we also define the normalized form
\begin{align*}
    U_{k-2}(t, q) \vcentcolon = q^{\frac{k}{2}-1}U_{k-2}(t/2\sqrt{q}).
\end{align*}

\begin{theorem}[{\cite[Theorem 3]{KP2}}] \label{thm:KPthm3}
Let $A$ be a finite abelian group of rank at most $2$ and let $n_i(A)$ be invariant factors of $A$ for $i = 1, 2$ with $n_1(A)\geq n_2(A)$. 
Suppose $(q,|A|)=1$ and $k\geq 2$. If $q\equiv 1 \pmod{n_2(A)}$ we have
\begin{align*}
    \mathbb{E}_q(U_{k-2}(a_E,q)\Phi_A)&=\frac{1}{q\varphi(n_1/n_2)} \sum_{\nu \mid \frac{(q-1,n_1)}{n_2}}
\phi(\nu)\left(T_{n_1,n_2\nu}(q,1)-p^{k-1}T_{n_1,n_2\nu}(q/p^2,p) \right) \\
&+q^{k/2-1}\frac{(p-1)(k-1)}{24q}\left(\delta_{n_1}(q^{1/2},1)+(-1)^k\delta_{n_1}(q^{1/2},-1) \right)
\end{align*}
and if $q \not\equiv 1$ (mod $n_2(A)$), then $ \mathbb{E}_q(U_{k-2}(a_E,q)\Phi_A)=0$.
\end{theorem}

From the definitions, it is clear that for given $n_1$, $\lambda$, $q=p^v$, and $d$, 
\begin{align} \label{bound1}
    T_{\id}, T_{\hyp} \ll_{n_1, \lambda, k} q^{\frac{k-1}{2}}, \quad \text{and} \quad T_{\dual}=0
\end{align}
for $k \geq 3$.
For a natural number $m$, Hecke operator $T_m$, and the number of divisors $d(m)$, we have
\begin{align} \label{Deligne bound}
    \Tr( T_m \langle d \rangle | S_k(\Gamma(M,N))) \leq \frac{k-1}{12}\varphi(N)\psi(NM)d(m)m^{\frac{k-1}{2}}
\end{align}
by Deligne's bound (cf. \cite[(1-6)]{Petrow18}). Using this bound, we have
\begin{align} \label{bound2}
    T_{\text{trace}} \ll_{n_1,\lambda,k} v\cdot q^{\frac{k-1}{2}}.
\end{align}
By (\ref{bound1}) and (\ref{bound2}), we obtain 
\begin{align*}
T_{n_1,\lambda}(q,d) \ll_{n_1,\lambda,k} v\cdot q^{\frac{k-1}{2}}.
\end{align*}
Therefore, by Theorem \ref{thm:KPthm3}
\begin{align} \label{eqn:tr-bound}
\mathbb{E}_q(U_{k-2}(a_E,q)\Phi_A) \ll_{n_1, n_2, k} v \cdot q^{\frac{k-3}{2}}, \qquad (k \geq 3).
\end{align}
Using this estimate, we obtain bounds for the first and second moments of traces of Frobenius, weighted by $H_{\Gamma}(a, q)$. 
From now on, for $q=p^v$, we assume that the exponent $v$ is bounded. 
This assumption is harmless, since in section \ref{sec:average} $q$ will be the norm of a prime ideal $\fp$ in a fixed number field $K$, and hence $v$ is at most the degree of $K$.

\begin{theorem} \label{thm:Hmoment12}
Let $M, N \geq 1$ be integers, let $\Gamma = \Gamma_1(M,N)$ a congruence subgroup such that $\cX_{\Gamma} \cong \bP^1$ over $\Spec \bZ[1/6MN]$, and let $q$ be a prime power satisfying $(q, 6MN) = 1$.
Then,
\begin{align*}
\sum_{|a|\leq 2\sqrt{q}}aH_{\Gamma}(a, q) =O_{\Gamma}(1), \qquad
\sum_{|a|\leq 2\sqrt{q}}a^2H_{\Gamma}(a, q) = q +O_{\Gamma}(q^{\frac{1}{2}}).
\end{align*}
\end{theorem}
\begin{proof}
We define $A_{\Gamma, i}$ as groups satisfying
\begin{align*}
\bZ/MN\bZ \times \bZ/M\bZ  \leq A_{\Gamma,i} \leq \bZ/MN\bZ \times \bZ/MN\bZ
\end{align*}
and $j <i$ if and only if $A_{\Gamma, j} < A_{\Gamma, i}$. 
For example when $M = 1$ and $N=6$, we may choose
$A_{\Gamma, i_1} = \bZ/6\bZ, A_{\Gamma, i_2} = \bZ/6\bZ \times \bZ/2\bZ, A_{\Gamma, i_3} = \bZ/6\bZ \times \bZ/3\bZ$, and $ A_{\Gamma, i_4} = \bZ/6\bZ \times \bZ/6\bZ$ with non-linear ordering 
$i_1 \leq i_k$ for $k = 2, 3, 4$ and $i_2, i_3 \leq i_4$.
We also define
\begin{align*}
    \widetilde{\omega}_{\Gamma, i} = |\phi_{\Gamma}^{-1}(E_z)| |\Aut_{\bF_q}(E_z)| \quad \textrm{ if } \quad E_z(\bF_q)[MN] \cong A_{\Gamma, i}
\end{align*}
and $\omega_{\Gamma, i} = \widetilde{\omega}_{\Gamma, i} - \sum_{j < i}\omega_{\Gamma, j}$.
By Proposition \ref{prop:keymod}, $|\phi_{\Gamma}^{-1}(E_z)| |\Aut_{\bF_q}(E_z)|$ depends only on the group $E_z(\bF_q)[MN]$, even if the elliptic curves $E_z$ are not isomorphic.
Hence $\widetilde{\omega}_{\Gamma, i}$ is well-defined and $O_{\Gamma}(1)$.

By (\ref{eqn:HGamma def rep}) and the definition of $\widetilde{\omega}_{\Gamma, i}$, we have
\begin{align*}
\frac{q^2}{q-1}\sum_{|a| \leq 2\sqrt{q}} a^R H_{\Gamma}(a, q) = \sum_{|a| \leq 2\sqrt{q}} a^R \sum_{a_q(E_z) = a} |\phi_{\Gamma}^{-1}(z)| = \sum_{|a| \leq 2\sqrt{q}} a^R \sum_{A_{\Gamma, i}}
\sum_{ \substack{a_q(E_z) = a \\ E_z(\bF_q)[MN] \cong A_{\Gamma, i} }}  \frac{\widetilde{\omega}_{\Gamma, i}}{|\Aut_{\bF_q}(E_z)|},
\end{align*}
for an integer $R \geq 0$.
By definition (\ref{eqn: def bE}), the sum is
\begin{align*}
    \sum_{A_{\Gamma, i}} \sum_{ \substack{ E_z(\bF_q)[MN] \cong A_{\Gamma, i} }}  \frac{a_q(E_z)^R\widetilde{\omega}_{\Gamma, i}}{|\Aut_{\bF_q}(E_z)|}
    =\sum_{A_{\Gamma, i}}\sum_{ \substack{ E_z(\bF_q)[MN] \geq A_{\Gamma, i} }}  \frac{a_q(E_z)^R \omega_{\Gamma, i}} {|\Aut_{\bF_q}(E_z)|} 
    = \sum_{A_{\Gamma, i}} q \omega_{\Gamma, i} \bE_q(a_E^R \Phi_{A_{\Gamma, i}}).
\end{align*}
Consequently, we obtain
\begin{align} \label{eqn: Hmoment proof}
    \frac{q}{q-1}\sum_{|a| \leq 2\sqrt{q}} a^R H_{\Gamma}(a, q) =\sum_{A_{\Gamma, i}} \omega_{\Gamma, i} \bE_q(a_E^R \Phi_{A_{\Gamma, i}}).
\end{align}
From (\ref{eqn:tr-bound}) for $k=3$, we have $\bE_{q}(U_1(a_E, q)\Phi_A) = \bE_q(a_E \Phi_A) = O(1).$
Hence,  (\ref{eqn: Hmoment proof}) for $R = 1$ gives
\begin{align*}
    \sum_{|a| \leq 2\sqrt{q}} a H_{\Gamma}(a, q) = O\lbrb{\sum_{A_{\Gamma, i}} \omega_{\Gamma, i}} = O(1).
\end{align*}

Together with Lemma \ref{lem:Hmoment0} and (\ref{eqn: Hmoment proof}) for $R = 0$, we have
\begin{align*}
    \sum_{A_{\Gamma, i}} \omega_{\Gamma, i} \bE_q(\Phi_{A_{\Gamma, i}}) = 1 + O(q^{-1}).
\end{align*}
From (\ref{eqn:tr-bound}) for $k = 4$, we have $\bE_{q}(U_2(a_E, q)\Phi_A) =  O\left(q^{\frac{1}{2}}\right).$
Since $t^2 = U_2(t, q) + q$,
\begin{align*}
    \sum_{A_{\Gamma, i}} \omega_{\Gamma, i} \bE_q(a_E^2 \Phi_{A_{\Gamma, i}})
    &=\sum_{A_{\Gamma, i}} \omega_{\Gamma, i} \bE_q((U_2(a_E, q) + q) \Phi_{A_{\Gamma, i}}) =\sum_{A_{\Gamma, i}} \omega_{\Gamma, i} \lbrb{ \bE_q(U_2(a_E, q)\Phi_{A_{\Gamma, i}}) + \bE_q(q \Phi_{A_{\Gamma, i}}) } \\
    & = q + O\lbrb{q^{\frac{1}{2}}} .
\end{align*}
Hence (\ref{eqn: Hmoment proof}) for $R = 2$ gives the second estimate.
\end{proof}
The first moment in Theorem \ref{thm:Hmoment12}, which comes from Deligne's bound, is not sufficient to deduce Theorem \ref{mainthm:rankbound}.
Therefore, we retain the traces of Hecke operators on the space of cusp forms, so that we can exploit their cancellation later.

\begin{proposition} \label{prop:Hmoment1ref}
Let $A$ be an abelian group of rank $2$ with invariant factors $n_1, n_2$ with $n_1 \geq n_2$. 
Let $p$ be a prime that does not divide $|A|$. If $p \equiv 1 \pmod{n_2}$, then there are explicit constants $b(n_1, n_2, \nu)$ such that
\begin{align*}
\bE_p(a_E \Phi_{A})= 
\frac{1}{p}\sum_{\nu \mid \frac{(p-1,n_1)}{n_2}}
b(n_1,n_2,\nu)\Tr\left(T_p|S_3(\Gamma(n_1,n_2\nu))\right)  + O_{n_1,n_2}\lbrb{p^{-1}}. 
\end{align*}
If not, $\bE_p(a_E \Phi_{A})=0$.
\end{proposition} 
\begin{proof}
When $q=p$ and $k=3$, Theorem \ref{thm:KPthm3} says that
\begin{align*}
 \bE_p(a_E \Phi_{A})=\frac{1}{p\varphi(n_1/n_2)}\sum_{\nu \mid \frac{(p-1,n_1)}{n_2}}\phi(\nu)T_{n_1,n_2\nu}(p,1).
\end{align*}
Since \begin{align*}
    & T_{n_1,n_2 \nu}(p,1)=\frac{\psi(n_1^2/(n_2\nu)^2)\varphi(n_1/n_2\nu)}{\psi(n_1^2)}(-T_{\trace}+T_{\id}-T_{\hyp}+T_{\dual}), \\
& T_{\trace}=\frac{1}{\varphi(n_1)}\Tr(T_p|S_3(\Gamma(n_1,n_2\nu))), \qquad T_{\id}=T_{\dual}=0, \qquad T_{\hyp}=O_{n_1,n_2}(1), 
\end{align*}
the claim follows.
\end{proof}

\begin{remark} \label{nu condition}
In Proposition \ref{prop:Hmoment1ref}, the coefficients $b(n_1, n_2, \nu)$ depend only on $n_1, n_2$ and $\nu$. However, $\nu$ depends on the $p$. For the given prime $p \equiv 1$ (mod $n_2$), the values of $\nu$ are the divisors of $m$, where $(p-1,n_1)=m \cdot n_2$. If $n_2=1$, then $p$ can be any prime co-prime to $n_1$.
\end{remark}

\section{Average analytic rank} \label{sec:average}

\subsection{Statement}
Let $E$ be an elliptic curve defined over a number field $K$ of degree $d = [K:\bQ]$. We normalize the $L$-function associated to the elliptic curve so that the central point is not $1$ but $1/2$. 
Let $N_{K/\bQ}$ be the usual norm map.
Then, it has the Euler product for $\Re(s)>1$ as follows:
\begin{align*}
 & L(E/K, s)=\prod_{\fp}\left(1-\frac{\alpha_E(\fp)}{N_{K/\bQ}(\fp)^s}\right)^{-1}\left(1-\frac{\beta_E(\fp)}{N_{K/\bQ}(\fp)^s} \right)^{-1},
\end{align*}
for some $\alpha_E(\fp), \beta_E(\fp) \in \bC$ where $\fp$ runs over all the prime ideals of the field $K$.
The local parameters $\alpha_E(\fp)$ and $\beta_E(\fp)$ are related to the trace of Frobenius, as described below. 
Let $q$ be the cardinality of the residue field $\kappa(\fp)$, so $q = N_{K/\bQ}(\fp)$. 
When we vary the prime ideal $\fp$, we will use the more precise notation $q_{\fp}$.
The trace $a_E(\fp)$ of Frobenius at $\fp$ is defined by 
$$ 
    a_E(\fp)=\left\{ \begin{array}{cl}
    q+1-\# E(\kappa(\fp)) & \text{if $E$ has good reduction at $\fp$,} \\
    1 & \text{if $E$ has split multiplicative reduction at $\fp$,}\\
    -1 & \text{if $E$ has non-split multiplicative reduction at $\fp$,}\\
    0 & \text{if $E$ has additive reduction at $\fp$.}
\end{array} \right.
$$
If $E$ has good reduction at $\fp$, then $\alpha_E(\fp)$ and $\beta_E(\fp)$ are uniquely determined (up to order) by the relations
$$ \alpha_E(\fp)+\beta_E(\fp) = \frac{a_E(\fp)}{\sqrt{q }}, \quad \alpha_E(\fp)\beta_E(\fp)=1.$$
If $E$ has split (resp. non-split) reduction at $\fp$, we have 
\begin{align*}
    \alpha_E(\fp)=\frac{1}{\sqrt{q}} \quad \textrm{(resp. $\alpha_E(\fp)=\frac{-1}{\sqrt{q}}$ ) }, \quad  \textrm{and} \quad   \beta_E(\fp)=0.
\end{align*}
If $E$ has additive reduction at $\fp$, we have 
\begin{align*}
    \alpha_E(\fp)=\beta_E(\fp)=0.
\end{align*}
Let $\widehat{a}_{E}(\fp) \vcentcolon = a_E(\fp)/q_{\fp}$.
Then $\widehat{a}_E(\fp^k)=\alpha_E(\fp)^k+\beta_E(\fp)^k$ and we see that the logarithmic derivative of $L(E/K,s)$ is given by
\begin{align*}
-\frac{L'}{L}(E/K, s)=\sum_{\fp}\sum_{k=1}^\infty \frac{\widehat{a}_E(\fp^k) \log q_{\fp}}{q_{\fp}^{ks}}
\end{align*}
from the Euler product.

Let $f(E/K)$ be the conductor of $E/K$, $D_K$ the absolute value of the discriminant of $K$, and $\Gamma(s)$ the usual Gamma function.
We define the complete $L$-function of $L(E/K,s)$ by
\begin{align*}
    \Lambda(E/K,s)=A_E^{\frac{s+1/2}{2}}\Gamma_K(s)L(E/K, s)
\end{align*}
where 
\begin{align*}
    A_E=D_K^2 N_{K/\bQ}(f(E/K)), \qquad 
    \Gamma_K(s)=\left( \frac{1}{(2 \pi)^{(s+1/2)}}\Gamma\left(s+\frac{1}{2}\right)\right)^d.
\end{align*}
We assume the following standard conjecture (cf. \cite[\S 16.3]{Hus}).
\begin{conjecture}[Hasse--Weil] \label{H-W} 
The complete $L$-function $\Lambda(E/K,s)$ has analytic continuation to the whole complex plane, and there is a root number $\omega_E \in \lcrc{\pm 1}$ satisfying
\begin{align*}
    \Lambda(E/K,s)=\omega_E \Lambda(E/K,1-s).
\end{align*}
\end{conjecture}

If an $L$-function has analytic continuation, then the analytic rank of the $L$-function, denoted by $r_E$, is defined by the order of zeros at the central point.
We further assume the generalized Riemann hypothesis for $L(E/K, s)$.
Then every non-trivial zero can be denoted by $\frac{1}{2} + i\gamma_E$ where $\gamma_E$ is a real number.
In this paper, we use the following test function
\begin{align*}
    \phi(x) \vcentcolon= \frac{\sin^2(2 \pi x \sigma/2)}{(2 \pi x)^2}, \qquad \text{and} \qquad
    \widehat{\phi}(u) \vcentcolon= \frac{1}{2}\lbrb{\frac{1}{2} \sigma - \frac{1}{2}|u|},
\end{align*}
where $\sigma$ is a positive constant and the support of $\widehat{\phi}$ is $[-\sigma, \sigma]$.
We note that $\phi(0) = \frac{\sigma^2}{4}$ and $\widehat{\phi}(0) = \frac{\sigma}{4}$.
Since the function $\phi$ is non-negative valued, we have the trivial inequality
\begin{align*}
    r_E  \cdot \phi(0) \leq \sum_{\gamma_E} \phi\left( \gamma_E \frac{\log X}{2\pi}\right).
\end{align*}
This yields the following upper bound for the average analytic rank:
\begin{align} \label{rank_ineq}
    \frac{1}{|\cE_{K, \Gamma}(X)|} \sum_{E \in \cE_{K, \Gamma}(X)} r_E
    \leq  \frac{1}{\left| \cE_{K,\Gamma}(X) \right|}\sum_{E \in \cE_{K, \Gamma}(X)} 
    \frac{1}{\phi(0)}\sum_{\gamma_E}\phi \left( \gamma_E \frac{\log X }{2 \pi }\right).
\end{align}

By Ogg's formula, the conductor of the elliptic curve divides the minimal discriminant \cite[Corollary 11.2]{Sil2}.
Let $D(E/K)$ be the minimal discriminant of $E/K$.
Then we have
\begin{align*}
    A_E =D_K^2 N_{K/\bQ}(f(E/K)) \leq D_K^2 N_{K/\bQ}(D(E/K)).
\end{align*}

Let
$\Delta = 16(4A^3 + 27B^2)$ for $A, B$ such that $E_{A, B}$ is isomorphic to $E$.
Since $\Delta\subseteq D(E/K)$ by the minimality of $D(E/K)$, we have $N_{K/\bQ}(D(E/K)) \leq N_{K/\bQ}(\Delta)$.
In this case, we have
\begin{align*}
    N_{K/\mathbb{Q}}(D(E/K))&\leq N_{K/\mathbb{Q}}(16(4A^3+27B^2))\\
    &=\prod_{v\in M_{K,\infty}}\left|16(4A^3+27B^2)\right|_v\\
    &\leq\prod_{v\in M_{K,\infty}}|16|_v(|4|_v|A^3|_v+|27|_v|B^2|_v)\\
    &\leq\prod_{v\in M_{K,\infty}}|16|_v|54|_v\max\left\{|A^3|_v,|B^2|_v\right\}.
\end{align*}
By Lemma \ref{lem: Height inf}, there is $(A, B)$ satisfying
\begin{align*}
    H_{(4, 6), K}([A, B]) \geq C_K \prod_{v \in M_{K, \infty}}|(A, B)|_{(4, 6), v}
    = C_K \prod_{v \in M_{K, \infty}} \max\lcrc{|A|_v^3, |B|_v^2}^{\frac{1}{12}}.
\end{align*}
Therefore, there is a constant $C_{K, 1}$ depending on $K$ such that
\begin{align*}
    A_E \leq C_{K, 1} \cdot H_{(4, 6), K}([A, B])^{12}.
\end{align*}

Hence, if the height of the elliptic curve under weight $(4,6)$ is bounded by $X$, then 
\begin{align} \label{cond_ineq}
\log A_E \leq 12 \log H_{(4, 6), K}([A, B]) + \log C_{K, 1} \leq 12 \log X + O(1).
\end{align}
By Weil's explicit formula \cite[Lemma 3.1]{CFLS}, we have
 \begin{align*}
	& \frac{1}{\left| \cE_{K,\Gamma}(X) \right|}\sum_{E \in \cE_{K, \Gamma}(X)} \sum_{\gamma_E}\phi \left( \gamma_E \frac{\log X }{2 \pi }\right) \\
	& =\frac{ \widehat{\phi}(0)}{ \left| \cE_{K,\Gamma}(X) \right|}\sum_{E \in \cE_{K,\Gamma}(X)} \frac{ \log A_E }{\log X} + \frac{2}{\pi}\int_{-\infty}^\infty \phi \left( \frac{\log X \cdot r}{2 \pi } \right) \Re \frac{\Gamma_K'}{\Gamma_K}\left(\frac 12 +ir\right)dr \\
	&-\frac{2}{\log X \left| \cE_{K,\Gamma}(X) \right| }\sum_{\fp}\sum_{k=1}^\infty \frac{\log q_\fp}{\sqrt{ q_\fp^k}}\widehat{\phi}\left( \frac{\log q_\fp^k}{\log X}\right) \sum_{E \in \cE_{K,\Gamma}(X)}\widehat{a}_E(\fp^k) 
    \end{align*}
    Let us define 
    \begin{align*}
S_1 &=\frac{2}{\log X \left| \cE_{K,\Gamma}(X) \right| }\sum_{\fp} \frac{\log q_\fp}{\sqrt{q_\fp}}\widehat{\phi}\left( \frac{\log q_\fp}{\log X}\right) \sum_{E \in \cE_{K,\Gamma}(X)}\widehat{a}_E(\fp), \\
S_2 &=\frac{2}{\log X \left| \cE_{K,\Gamma}(X) \right| }\sum_{\fp} \frac{\log q_\fp}{q_\fp}\widehat{\phi}\left( \frac{2\log q_\fp}{\log X}\right) \sum_{E \in \cE_{K, \Gamma}(X)}\widehat{a}_E(\fp^2).
\end{align*}
We can see that the sum over the terms $q_\fp^k$ with $k\geq 3$ converges absolutely. Using the estimate of the digamma function $\frac{\Gamma'}{\Gamma}(z)=\log z + O(1)$ for a fixed $\Re(z)>0$, the integral term is bounded by $O_{\widehat{\phi}}\left(\frac{1}{\log X} \right)$.   
Hence, by \eqref{cond_ineq} and the comments above, we have 
    \begin{align*}
	\frac{1}{\left| \cE_{K,\Gamma}(X) \right|}\sum_{E \in \cE_{K, \Gamma}(X)} \sum_{\gamma_E}\phi \left( \gamma_E \frac{\log X }{2 \pi }\right) \leq 12\widehat{\phi}(0) - S_1 -S_2 +  O_K\left( \frac{1}{\log X}\right).
\end{align*}

From now on, we will focus on showing that 
\begin{align} \label{est_of_s1_1}
& S_1 \ll \frac{\log \log X}{\log X } + X^{\frac{3\sigma}{2} - \frac{1}{de(\Gamma)}}, \\ \label{est_of_s2}
&S_2 = -\frac{1}{2} \phi(0) + O\lbrb{X^{ \frac{\sigma}{2}  - \frac{1}{de(\Gamma)}} }.
\end{align}
Let $\Gamma$ be a congruence subgroup such that $\cX_{\Gamma} \cong \bP^1$.
If $(\ref{est_of_s1_1})$ and $(\ref{est_of_s2})$ are true, we have
\begin{align*}
\frac{1}{|\cE_{K, \Gamma}(X)|} \sum_{E \in \cE_{K, \Gamma}(X)} r_E &\leq \frac{1}{2}+ \frac{12\widehat{\phi}(0)}{\phi(0)} +o(1)  = \frac{1}{2} + \frac{12}{\sigma} + o(1)\\
& = \frac{1}{2} + 18 de(\Gamma) +o(1)
\end{align*}
by taking $\sigma$ arbitrarily close to $\frac{2}{3de(\Gamma)}$.
Hence, we obtain the following under (\ref{est_of_s1_1}) and (\ref{est_of_s2}).

\begin{theorem} \label{thm:rankbound01}
    Let $K$ be a number field of degree $d$, and let $\Gamma = \Gamma_1(M, N)$ be a congruence subgroup such that $\cX_{\Gamma} \cong \bP^1$ over $\Spec \bZ[1/6MN]$.
    Under the Hasse--Weil conjecture and the generalized Riemann hypothesis of $L$-function of elliptic curves over $K$, we have
    
    \begin{align*}
        \limsup_{X\to \infty} \frac{1}{|\cE_{K, \Gamma}(X)|} \sum_{E \in \cE_{K, \Gamma}(X)}r_E
        \leq 18e(\Gamma)d + \frac{1}{2}.
    \end{align*}
\end{theorem}

\begin{remark} \label{rmk:estpowcompare}
Our bound on the average analytic ranks is better than the cases when the order of the torsion groups is $\geq 5$  in \cite{CJ2}.
The reason is that the estimation of \cite[Corollay 3.10]{CJ2}
\begin{align*}
    |\cE_{G,p}^a(X)|&= c(G)
\frac{H_G(a,p)}{p^2-1} X^{\frac{1}{d(G)}}+O\left(H_G(a,p)X^\frac{1}{e(G)}+\frac{H_G(a,p)}{p}X^{\frac{1}{e(G)}} \log X \right)
\end{align*}
has an additional error term $H_G(a, p)X^{\frac{1}{e(G)}}$, compared with the our estimation (\ref{eqn:EKgamaref}).
One can check that $H_G(a, p)X^{\frac{1}{e(G)}}$ gives the main error term in the proof of \cite[Theorem 1]{CJ2}.
\end{remark}

\subsection{Estimate of $S_1$}
To estimate $S_1$, first, we need to control the inner sum of $S_1$.
\begin{lemma} \label{lem:aEp}
Let $K$ be a number field, $\Gamma$ the congruence subgroup of level $N$ for which  $\cX_{\Gamma} \cong \bP^1$ over $\bZ[1/6N]$.
Then for a prime $\fp$ not dividing $6N$,
\begin{align*}
\sum_{E \in \cE_{K,\Gamma}(X)} \widehat{a}_E(\fp) =
\frac{\kappa}{\sqrt{q_\fp}} \lbrb{ \sum_{|a| \leq 2\sqrt{q_\fp}}aH_{\Gamma}(a, q_\fp)} X^{\frac{2}{e(\Gamma)}}
&+O_{\Gamma}\lbrb{q_\fp  X^{\frac{2d - 1}{de(\Gamma)}}\log X + q_\fp^{-\frac{5}{2}} X^{\frac{2}{e(\Gamma)}}}.
\end{align*}
\end{lemma}
\begin{proof}
We have
\begin{align*}
	\sum_{E \in \cE_{K, \Gamma}(X)} \widehat{a}_E(\fp)= \sum_{|a| \leq 2\sqrt{q_\fp}} 
	\sum_{\substack{E \in \cE_{K,\Gamma}(X) \\  a_E(\fp) =a  }} \widehat{a}_E(\fp)
	+ \sum_{\substack{E \in \cE_{K, \Gamma}(X) \\ \text{$E$ mult at $\fp$}}}\widehat{a}_E(\fp).
\end{align*}
By Lemma \ref{lem:cuspGam} and (\ref{mult count}), we have 
\begin{align*}
    \left| 
    \sum_{\substack{E \in \cE_{K, \Gamma}(X) \\ E, \textrm{ mult at } \fp}} \widehat{a}_E(\fp)
    \right|
    &\leq q_\fp^{-\frac{1}{2}} |\cE_{K, \Gamma, \fp}^{\mathrm{mult}}(X)| = q_\fp^{-\frac{1}{2}}
    \frac{|\cX^{\mathrm{cusp}}_{\Gamma}(\bF_q)|}{q+1} \cdot  \kappa \cdot X^{\frac{2}{e(\Gamma)}} 
    + O\lbrb{ q_\fp^{-\frac{1}{2}}  X^{\frac{2d - 1}{de(\Gamma)}} \log X} \\
    &\ll_\Gamma  q_\fp^{-\frac{3}{2}}X^{\frac{2}{e(\Gamma)}} + q_\fp^{-\frac{1}{2}}X^{\frac{2d-1}{de(\Gamma)}}\log X.
\end{align*} 
Also by (\ref{eqn:EKgamaref}),  
\begin{align*}
&\sum_{|a| \leq  2\sqrt{q_\fp}} \sum_{\substack{E \in \cE_{K,\Gamma}(X) \\ a_E(\fp)=a}}\widehat{a}_E(\fp)
= \sum_{|a| \leq  2\sqrt{q_\fp}} \frac{a}{\sqrt{q}} |\cE_{K, \Gamma, \fp}^a(X)| \\
&= \sum_{|a| \leq 2\sqrt{q_\fp}} \frac{a}{\sqrt{q_\fp}}  \cdot   \frac{q_\fp^2}{q_\fp^2-1} \cdot \kappa \cdot 
    H_{\Gamma}(a, q_\fp)\cdot X^{\frac{2}{e(\Gamma)}} 
 + O\lbrb{\sum_{|a| \leq 2\sqrt{q_\fp}} \frac{|a|}{\sqrt{q_\fp}}  \lbrb{\sum_{a_{\fp}(E_z) = a} 1} X^{\frac{2d - 1}{de(\Gamma)}}\log X}  \\
&=  \frac{\kappa}{\sqrt{q_\fp}} \frac{q_\fp^2}{q_\fp^2-1} \lbrb{ \sum_{|a| \leq 2\sqrt{q_\fp}}aH_{\Gamma}(a, q_\fp)} X^{\frac{2}{e(\Gamma)}}
+O\lbrb{ q_\fp X^{\frac{2d - 1}{de(\Gamma)}}\log X},
\end{align*}
where in the error term of the last identity we applied the inequality
\begin{align*}
 \sum_{|a| \leq 2\sqrt{q_\fp}} \frac{|a|}{\sqrt{q_\fp}}  
 \sum_{\substack{a_{\fp}(E_z) = a \\ E_z \in \cY(\bF_{q_\fp})}} 1 
 \leq 2 |\cY(\bF_{q_\fp})| \ll q_\fp.
\end{align*}
By Theorem \ref{thm:Hmoment12} and $\frac{q_\fp^2}{q_\fp^2-1} = 1 + O(q_\fp^{-2})$, we also obtain $q_\fp^{-\frac{5}{2}}X^{\frac{2}{e(\Gamma)}}$ in the error term.
\end{proof}

By Lemma \ref{lem:aEp}, we have
\begin{align*}
S_1 &=\frac{2}{\log X \left| \cE_{K,\Gamma}(X) \right| }\sum_{\fp} \frac{\log q_{\fp}}{\sqrt{q_{\fp}}}\widehat{\phi}\left( \frac{\log q_{\fp}}{\log X}\right)\\
& \times \left[ \frac{\kappa}{\sqrt{q_{\fp}}} \lbrb{ \sum_{|a| \leq 2\sqrt{q_{\fp}}}aH_{\Gamma}(a, q_{\fp})} X^{\frac{2}{e(\Gamma)}}
+O_{\Gamma}\lbrb{ q_{\fp}  X^{\frac{2d - 1}{de(\Gamma)}}\log X + q_{\fp}^{-\frac{5}{2}} X^{\frac{2}{e(\Gamma)}}} \right].
\end{align*}
Then, the contribution from the error term is, by the prime number theorem and partial summation,
\begin{align*}
    &\ll \frac{1}{\log X \left| \cE_{K,\Gamma}(X) \right| }\sum_{\fp} \frac{\log q_{\fp}}{\sqrt{q_{\fp}}}\widehat{\phi}\left( \frac{\log q_{\fp}}{\log X}\right)  \cdot 
    \lbrb{ q_{\fp}
    X^{\frac{2d - 1}{de(\Gamma)}}\log X  + q_{\fp}^{-\frac{5}{2}} X^{\frac{2}{e(\Gamma)} } } \\
   &\ll_d X^{-\frac{1}{de(\Gamma)}} \left( \sum_{p \leq X^\sigma} \sqrt{p} \log p + \sum_{k=2}^{\log X} \sum_{p \leq X^{\frac{\sigma}{k} }} k p^{ \frac{k}{2}  } \log p \right) + \sum_{ n \leq X^\sigma} \frac{\log n}{n^3 \log X} \\
    &\ll_d X^{\frac{3\sigma}{2} - \frac{1}{de(\Gamma)}}  + \frac{1}{\log X}.
\end{align*} 
The contribution from the main term is
\begin{align*}
    &\frac{2}{\log X \left| \cE_{K,\Gamma}(X) \right| }\sum_{\fp} \frac{ \log q_{\fp}}{\sqrt{q_{\fp}}}\widehat{\phi}\left( \frac{\log q_{\fp}}{\log X}\right)  \cdot 
    \frac{\kappa}{\sqrt{q_{\fp}}} \lbrb{ \sum_{|a| \leq 2\sqrt{q_{\fp}}}aH_{\Gamma}(a, q_{\fp})}
    X^{\frac{2}{e(\Gamma)}} \\
    &= \frac{2 \kappa X^{\frac{2}{e(\Gamma)}}}{\log X \left| \cE_{K,\Gamma}(X) \right|} \sum_{\fp} \frac{\log q_{\fp}}{q_{\fp}}
    \widehat{\phi} \lbrb{ \frac{\log q_{\fp}}{\log X}}
    \lbrb{ \sum_{|a| \leq 2\sqrt{q_{\fp}}}aH_{\Gamma}(a, q_{\fp})}.
\end{align*}

We note that the contribution of $\fp$ dividing the level of $\Gamma$ is negligible.
Let $f(\fp)$ be the residue degree denoted by $f(\fp|p)$ in section \ref{sec:counting}.
By Theorem \ref{thm:Hmoment12} on the first moment, it is easy to see that the contribution of prime ideals $\fp$ with $q_\fp$ not prime is $O\left(\tfrac{1}{\log X}\right)$ because 
\begin{align} \label{eqn: f>1 vanishes}
\frac{2}{\log X} \sum_{\fp, f(\fp)>1 } \frac{\log q_{\fp}}{q_{\fp}}
    \widehat{\phi} \lbrb{ \frac{\log q_{\fp}}{\log X}}
    \lbrb{ \sum_{|a| \leq 2\sqrt{q_{\fp}}}aH_{\Gamma}(a, q_{\fp})} \ll_{\Gamma, d } \frac{2}{\log X} \sum_{k=2}^{[K:\bQ]} \sum_{n=1}^\infty \frac{\log n}{n^k} \ll_{\Gamma, d } \frac{1}{\log X}.
\end{align}
Hence, we consider the prime ideals $\fp$ whose norm is a rational prime only.
From the equation $(\ref{eqn: Hmoment proof})$, we see that the main term contribution becomes
\begin{align} \label{eqn:S1estp}
    \frac{2 }{\log X}  \sum_{\fp, f(\fp)=1 } \frac{\log q_{\fp}}{q_{\fp}}
    \widehat{\phi} \lbrb{ \frac{\log q_{\fp}}{\log X}}
    \lbrb{  \lbrb{1 - \frac{1}{q_{\fp}}} \sum_{i} \omega_{\Gamma, i} \bE_{q_{\fp}}(a \Phi_{A_{\Gamma, i}}) } + O_{\Gamma,d }\lbrb{\frac{1}{\log X}}.
\end{align} 
For simplicity, we denote $A_i\vcentcolon=A_{\Gamma,i}$, $\omega_i \vcentcolon= \omega_{\Gamma, i}$ and invariant factors of $A_i$ by $n_{i,1},n_{i,2}$.
By Proposition \ref{prop:Hmoment1ref} and Remark \ref{nu condition}, (\ref{eqn:S1estp}) is 
\begin{align*}
    &\frac{2}{\log X} \sum_{\substack{\fp, f(\fp)=1}} \frac{\log q_{\fp}}{q_{\fp}}
    \widehat{\phi} \lbrb{ \frac{\log q_{\fp}}{\log X}} \\
    & \times \lbrb{
    \sum_{i} \omega_{i} \lbrb{1 - \frac{1}{q_{\fp}}}
    \lbrb{ \frac{1}{q_{\fp}} \sum_{\nu \mid \frac{( q_{\fp}-1, n_{1, i})}{n_{2,i}}}
    b(n_{1, i}, n_{2, i}, \nu)
    \Tr(T_{q_{\fp}} | S_3(\Gamma(n_{1, i}, n_{2,i}\nu) )) + O\lbrb{\frac{1}{q_{\fp}}} } } \\
    &= \frac{2}{\log X}\sum_{i} \omega_{i}\sum_{\nu \mid \frac{n_{1,i}}{n_{2,i}}} 
 b(n_{1, i}, n_{2, i}, \nu) \sum_{\substack{ q_{\fp} \equiv 1 \text{ (mod $n_{2,i}\nu$)} \\ q_{\fp} \leq X^\sigma}} \frac{\log q_{\fp}}{ q_{\fp}^2} \widehat{\phi} \lbrb{ \frac{\log  q_{\fp}}{\log X}} \Tr(T_{q_{\fp}} | S_3(\Gamma(n_{1, i}, n_{2,i}\nu) )) \\
 & \qquad  + O_{\Gamma, d}\left( \frac{1}{\log X} \right), 
\end{align*}
where the error term is justified by 
\begin{align*}
\frac{1}{\log X} \sum_{\substack{ f(\fp)=1, q_{\fp}\leq X^{\sigma}}} \frac{\log q_{\fp}}{q_{\fp}^2}
    \widehat{\phi} \lbrb{ \frac{\log q_{\fp} }{\log X}}
    \ll_d
    \frac{1}{\log X}  \sum_{p\leq X^{\sigma}}  \frac{\log p}{p^2} \ll_d
    \frac{1}{\log X},
\end{align*}
and
\begin{align*}
    & \frac{1}{\log X}  \sum_{f(\fp)=1, q_{\fp} \leq X^{\sigma}}  \frac{\log q_{\fp}}{q_{\fp}}
    \widehat{\phi} \lbrb{ \frac{\log q_{\fp}}{\log X}}
    \sum_i \frac{\omega_{i}}{q_{\fp}^2}
    \sum_{\nu \mid \frac{(q_{\fp}-1, n_{1, i})}{n_{2,i}}}
    b(n_{1, i}, n_{2, i}, \nu)
    \Tr(T_{q_{\fp}} | S_3(\Gamma(n_{1, i}, n_{2,i}\nu) )) \\
    & \ll_{\Gamma, d} \frac{1}{\log X}  \sum_{p\leq X^{\sigma}}  \frac{\log p}{p^3} \cdot p \\
    & \ll_{\Gamma, d} \frac{1}{\log X}
\end{align*}
by Deligne's bound $(\ref{Deligne bound})$. Hence, we need to give a bound on the sum
\begin{align*}
   \sum_{\substack{f(\fp)=1, q_{\fp} \leq X^\sigma \\ q_{\fp} \equiv 1 \text{ (mod $M$)} }} \frac{\log q_{\fp}}{q_{\fp}^2} \widehat{\phi} \lbrb{ \frac{\log p}{\log X}} \Tr(T_{q_{\fp}} | S_3(\Gamma(N, M) ))
\end{align*}
for fixed positive integers $N$ and $M$ with $M \mid N$.

Let $\mathfrak{B} = \mathfrak{B}(N,M)$ be an eigenform basis of $S_3(\Gamma(N, M))$, and let $a_f(n)$ be the $n$-th Fourier coefficient of $f$.
Then 
\begin{align*}
    \Tr(T_{q_{\fp}} | S_3(\Gamma(N, M) )) = \sum_{f \in \mathfrak{B}} a_{f}(q_{\fp}).
\end{align*}

In the following lemma, we show that for each $f \in \mathfrak{B}(N, M)$, the sum
\begin{align} \label{modular sum}
     \frac{1}{\log X}\sum_{\substack{f(\fp) = 1,q_{\fp} \leq X^\sigma  \\ q_{\fp} \equiv 1 \text{ (mod $M$)}  }}  \frac{a_{f}(q_{\fp})\log q_{\fp}}{q_{\fp}^2} \widehat{\phi} \lbrb{ \frac{\log q_{\fp}}{\log X}}
\end{align} vanishes.

\begin{lemma} \label{lem:Cheb}
    Let $K$ be a number field of degree $d$, and let $f$ be an eigenform in $\mathfrak{B}(N, M)$. Then,
    \begin{align*}
      \frac{1}{\log X}  \sum_{\substack{f(\fp) = 1,q_{\fp} \leq X^\sigma  \\ q_{\fp} \equiv 1 \text{ (mod $M$)}  }} \frac{ a_f( q_{\fp}) \log q_{\fp}}{ q_{\fp}^2}\widehat{\phi}\left( \frac{\log q_{\fp}}{\log X}\right) = O_{d, \sigma, M,N, \widehat{\phi}}\left( \frac{\log \log X}{\log X}\right).
    \end{align*}
\end{lemma}
\begin{proof}
The Dedekind zeta function $\zeta_K(s)$ has the Euler product 
$$ \zeta_K(s)=\prod_{\fp} \left( 1- \frac{1}{q_\fp^s} \right)^{-1}$$
for $\Re(s) > 1$.
Let
\begin{align*}
  \Lambda_K(\mathfrak{n})= \left\{ \begin{array}{cl}
    \log q_{\fp} & \text{ if $\mathfrak{n}=\fp^k$ for some prime ideal $\fp$,}\\
    0 & \text{otherwise.}
    \end{array} \right.
\end{align*}
From the Euler product, we find that the negation of the logarithmic derivative of $\zeta_K(s)$ is 
\begin{align*}
    -\frac{\zeta_K'(s)}{\zeta_K(s)}=\sum_{\mathfrak{\fp}}\sum_{k=1}^\infty \frac{\Lambda_K(\fp^k)}{q_{\fp}^{ks}}.
\end{align*}
Let
\begin{align*}
     \psi(K,t)=\sum_{q_{\fp}^k \leq t} \Lambda_K(\fp^k).
\end{align*}
Since $\zeta_K(s)$ has a simple pole at $s=1$, there is a positive constant $c$ depending on the degree $d$, and $\beta_K$ a (potential) Siegel zero of $\zeta_K(s)$ satisfying
\begin{align*}
\psi(K,t)= t -\frac{t^{\beta_K}}{\beta_K} + O_K \left(t \exp(-c\sqrt{\log t}) \right).
\end{align*}
The term containing $\beta_K$ disappears if it does not exist. (See \cite[(5.52)]{IK}.)  Since the contribution from the terms $\fp^k$ with $q_{\fp}=p^{f(\fp)}$ for a rational prime $p$ and $f(\fp)\geq 2$ to $\psi(K,t)$ is at most $O_d(t^\frac{1}{2}\log t)$, we have
\begin{align*}
    \psi(K,t)=\sum_{\substack{f(\fp)=1 \\ q_{\fp} \leq t}} \log q_{\fp}+O_d(\sqrt{t} \log t) = t -\frac{t^{\beta_K}}{\beta_K} + O_K \left(t \exp(-c\sqrt{\log t}) \right)
\end{align*} 
and this implies that
\begin{align*}
   \theta(K,t) \vcentcolon= \sum_{\substack{f(\fp)=1 \\ q_{\fp} \leq t}}\log q_{\fp} =  t -\frac{t^{\beta_K}}{\beta_K} + O_K \left(t \exp(-c\sqrt{\log t}) \right).
\end{align*}

Let $f$ be a Hecke eigenform in $S_3(\Gamma(N,M))$, and let $a_f(p)$ be the $p$-th Fourier coefficient of $f$. 
We define a function $F(t) : [1,\infty) \rightarrow \mathbb{R}$ as follows. 
If $p$ is a rational prime such that $ p \equiv 1 \pmod{M} $ and $ p = q_{\fp}$  for some prime ideal $\fp$ of $K$, then we define
\begin{align*}
    F(p) = \frac{a_f(p)}{p^2} \widehat{\phi} \left( \frac{\log p}{\log X} \right).
\end{align*}
For $t$ with $|t-p| \leq 1$, $F(t)$ is given from the lines that connect the three points $(p-1,0)$, $(p,F(p))$, and $(p+1,0)$. For the other $t$, we define $F(t)=0$. 
By its construction, 
\begin{align*}
    F'(t) = \left\{
    \begin{array}{ll}
    \displaystyle \frac{a_f(p)}{p^2}\widehat{\phi}\left( \frac{\log p}{\log X}\right)    & \text{if } t \in (p-1, p), \\
    \displaystyle -\frac{a_f(p)}{p^2}\widehat{\phi}\left( \frac{\log p}{\log X}\right)    & \text{if } t \in (p, p+1).
    \end{array}
    \right.
\end{align*}
Therefore, $F(t), F'(t) \ll 1/t$.
Then, we have, by partial summation,
\begin{align*} 
    &\sum_{\substack{f(\fp) = 1,q_{\fp} \leq X^\sigma  \\ q_{\fp} \equiv 1 \text{ (mod $M$)}  }} \frac{ a_f( q_{\fp}) \log q_{\fp}}{ q_{\fp}^2}\widehat{\phi}\left( \frac{\log q_{\fp}}{\log X}\right) \\
    &=\sum_{\substack{f(\fp)=1 \\ q_{\fp} \leq X^\sigma }} (\log q_{\fp} )F(q_{\fp}) 
    =\int_{1}^{X^\sigma}F(x)d\theta(K,x) \\
    &= -\int_{1}^{X^\sigma}  F'(x)\left(x-\frac{x^{\beta_K}}{\beta_K}+O(x\exp(-c\sqrt{\log x})\right)dx +O(1) \\
   & = -F(X^\sigma)\left( X^\sigma-\frac{X^{\sigma \cdot \beta_K}}{\beta_K} + O\left( X^\sigma \exp(-c\sqrt{\sigma \log X} )\right)\right)  \\
    & \qquad + \int_1^{X^\sigma} F(x)\lbrb{1 - \frac{1}{x^{1-\beta_K}} +O\left( \exp(-c\sqrt{\log x})\right) } dx   + O(1) \\
     &=\int_1^{X^\sigma}F(x)dx + O(1)=\sum_{p^* < X^\sigma}\frac{a_f(p)}{p^2}\widehat{\phi}\left( \frac{\log p}{\log X}\right)+ O(1) \ll_{d, \widehat{\phi}, M, N} \sum_{p^* < X^\sigma } \frac{1}{p} + O(1) \\
     &\ll_{d, \widehat{\phi}, M, N} \log \log X, 
\end{align*}
where $p^*$ is a rational prime such that $p \equiv 1 \pmod{M}$ and $p = q_{\fp}$ for some $\fp$, and the claim follows.
\end{proof}
We have reached (\ref{est_of_s1_1}) by our discussions above.

\subsection{Estimate of $S_2$}
To estimate $S_2$, first, we need to control the inner sum of $S_2$.  

\begin{lemma} \label{lem:aEpsquare}
Let $K$ be a number field, and let $\Gamma = \Gamma_1(M,N)$ be the congruence subgroup such that $\cX_{\Gamma} \cong \bP^1$ over $\Spec \bZ[1/6MN]$. 
Then, for a prime $\fp$ not dividing $6MN$,
\begin{align*}
\sum_{E \in \cE_{K,\Gamma}(X)} \widehat{a}_E(\fp^2) 
= - \kappa X^{\frac{2}{e(\Gamma)}}
+ O_{\Gamma}\!\left(q_\fp^{-\frac{1}{2}}X^{\frac{2}{e(\Gamma)}} 
+ q_\fp X^{\frac{2}{e(\Gamma)}- \frac{1}{de(\Gamma)}} \log X \right).
\end{align*}
\end{lemma}

\begin{proof}
First, we note that
\begin{align*}
	\sum_{E \in \cE_{K, \Gamma}(X)} \widehat{a}_E(\fp^2)= \sum_{|a| \leq 2\sqrt{q_\fp}} 
	\sum_{\substack{E \in \cE_{K,\Gamma}(X) \\  a_E(\fp) =a  }} \widehat{a}_E(\fp^2)
	+ \sum_{\substack{E \in \cE_{K, \Gamma}(X) \\ \text{$E$ mult at $\fp$}}}\widehat{a}_E(\fp^2)
\end{align*}
and $\widehat{a}_E(\fp^2)= q_\fp^{-1}$ when $E$ has multiplicative reduction at $\fp$.
Hence the same argument as Lemma \ref{lem:aEp} gives
\begin{align*}
     \left| \sum_{\substack{E \in \cE_{K, \Gamma}(X) \\ \text{$E$ mult at $\fp$}}}\widehat{a}_E(\fp^2) \right|
     \ll q_\fp^{-2} X^{\frac{2}{e(\Gamma)}}  + q_\fp^{-1}X^{\frac{2d-1}{de(\Gamma)}} \log X   .
\end{align*}
For the good reduction case, we use the identity $\widehat{a}_{E}(\fp^2) = \widehat{a}_E(\fp)^2 - 2$. The computation in the proof of Lemma \ref{lem:aEp} gives
\begin{align*}
&\sum_{|a| \leq  2\sqrt{q_\fp}} \sum_{\substack{E \in \cE_{K,\Gamma}(X) \\ a_E(\fp)=a}}\widehat{a}_E(\fp^2)
=\sum_{|a| \leq  2\sqrt{q_\fp}} \sum_{\substack{E \in \cE_{K,\Gamma}(X) \\ a_E(\fp)=a}}
( \widehat{a}_E(\fp)^2 - 2)
= \sum_{|a| \leq  2\sqrt{q_\fp}} \lbrb{\frac{a^2}{q_\fp}-2}|\cE_{K, \Gamma, \fp}^a(X)| \\
&=\kappa \lbrb{  \sum_{|a| \leq 2\sqrt{q_\fp}} \frac{a^2}{q_\fp} \frac{q_\fp^2}{q_\fp^2-1} H_{\Gamma}(a, qq_\fp } X^{\frac{2}{e(\Gamma)}} -2 \kappa \lbrb{  \sum_{|a| \leq 2\sqrt{q_\fp}} \frac{q_\fp^2}{q_\fp^2-1} H_{\Gamma}(a, q_\fp) } X^{\frac{2}{e(\Gamma)}}
+O_{\Gamma}\lbrb{ q_\fp  X^{\frac{2d - 1}{de(\Gamma)}}\log X} \\
&= \frac{\kappa}{q_\fp}
\lbrb{\sum_{|a| \leq 2\sqrt{q_\fp}} a^2 H_{\Gamma}(a, q_\fp) } X^{\frac{2}{e(\Gamma)}}
-2 \kappa \lbrb{\sum_{|a| \leq 2\sqrt{q_\fp}} H_{\Gamma}(a, q_\fp) } X^{\frac{2}{e(\Gamma)}}
+O_{\Gamma}\lbrb{q_\fp^{-2}X^{\frac{2}{e(\Gamma)}} + q_\fp X^{\frac{2d - 1}{de(\Gamma)}}\log X }.
\end{align*}
Again we note that the contribution of $\fp \mid N$ is negligible, so by  Lemma \ref{lem:Hmoment0} and Theorem \ref{thm:Hmoment12} the sum is
\begin{align*}
    - \kappa  X^{\frac{2}{e(\Gamma)}}
+O_{\Gamma}\lbrb{q_\fp^{-\frac{1}{2}}X^{\frac{2}{e(\Gamma)}} + q_\fp X^{\frac{2}{e(\Gamma)} - \frac{1}{de(\Gamma)}}\log X}.
\end{align*}
Thus, the lemma follows.
\end{proof}

By Lemma \ref{lem:aEpsquare},
\begin{align*}
    S_2 & = \frac{2}{\log X \left| \cE_{K,\Gamma}(X) \right| }\sum_{\fp} \frac{ \log q_{\fp}}{q_{\fp}}\widehat{\phi}\left( \frac{2\log q_{\fp}}{\log X}\right) \cdot 
    \lbrb{- \kappa  X^{\frac{2}{e(\Gamma)}}  +O\lbrb{q_{\fp}^{-\frac{1}{2}}X^{\frac{2}{e(\Gamma)}} + q_{\fp} X^{\frac{2d - 1}{de(\Gamma)}}\log X}} \\
    & = \frac{2}{\log X }\sum_{\fp} 
    \frac{ \log q_{\fp}}{q_{\fp}}
    \widehat{\phi}\left( \frac{2\log q_{\fp}}{\log X}\right)
    \lbrb{-1  +O\lbrb{q_{\fp}^{-\frac{1}{2}} + q_{\fp} X^{-\frac{1}{de(\Gamma)}}\log X}} \\
    & = -\frac{2}{\log X }\sum_{\fp} 
    \frac{ \log q_{\fp}}{q_{\fp}}
    \widehat{\phi}\left( \frac{2\log q_{\fp}}{\log X}\right) + O\lbrb{ \frac{1}{\log X} \sum_{ q_{\fp} \leq X^{\frac{\sigma}{2}}}
    \frac{\log q_{\fp}}{q_{\fp}^{\frac 32}} + X^{- \frac{1}{de(\Gamma)}}\sum_{q_{\fp} \leq X^{\frac{\sigma}{2}}}\log q_{\fp} }.
\end{align*}
Since
\begin{align*}
\sum_{ q_{\fp} \leq X^{\frac{\sigma}{2}}}\log q_{\fp} \leq \psi(K, X^{\frac{\sigma}{2}})\ll_d X^{\frac{\sigma}{2}} \textrm{ and }
    \sum_{ q_{\fp} \leq X^{\frac{\sigma}{2}}}
    \frac{\log q_{\fp} }{q_{\fp} ^{\frac 32}} \ll_d \sum_{n=1}^\infty \frac{\log n}{n^\frac{3}{2}} \ll_d 1, 
\end{align*}
(\ref{est_of_s2}) follows from the lemma below. 

\begin{lemma} \label{lem: S2 chev}
Let $\widehat{\phi}$ be an even continuous function supported in $[-\sigma, \sigma]$, and let $\phi$ be its Fourier transform. Then, we have
    \[\frac{2}{\log X }\sum_{\fp} 
    \frac{ \log q_{\fp}}{q_{\fp}}
    \widehat{\phi}\left( \frac{2\log q_{\fp}}{\log X}\right)=\frac{1}{2}\phi(0)+O_{d, \widehat{\phi}} \left( \frac{1}{\log X} \right). \]
\end{lemma}
\begin{proof}
First, we observe that it is enough to consider prime ideals $\fp$ with $f(\fp)=1$ only because the contribution from prime ideals with $f(\fp)\geq 2$ is $O_{d, \widehat{\phi}} \left( \frac{1}{\log X} \right)$, by (\ref{eqn: f>1 vanishes}).
Recall that
\begin{align*}
    \theta(K, t)= \sum_{\substack{f(\fp)=1 \\ q_{\fp} \leq t}} \log  q_{\fp}.
\end{align*}
Let $\tilde{\theta}(K,t)=\theta(K,t)-t$ for $t \geq 1$.
Since $\widehat{\phi}$ is supported in $[-\sigma, \sigma]$, we have, by partial summation,
\begin{align} \label{s2 sum}
    &\frac{2}{\log X }\sum_{\fp} 
    \frac{ \log q_{\fp}}{q_{\fp}}
    \widehat{\phi}\left( \frac{2\log q_{\fp}}{\log X}\right)  = \int_{1}^{X^{ \frac{\sigma}{2} }} \frac{2}{t \log X}\widehat{\phi}\left( \frac{2 \log t}{\log X} \right) d\theta(K,t) \nonumber\\ 
    &= \int_{1}^{X^{\frac{\sigma}{2}}} \frac{2}{t \log X}\widehat{\phi}\left( \frac{2 \log t}{\log X} \right) dt
    +\int_1^{X^{\frac{\sigma}{2}}} \frac{2}{t \log X}\widehat{\phi}\left( \frac{2 \log t}{\log X} \right) d\tilde{\theta}(K,t).
    \end{align}
We extract the main term as below:
\begin{align} \label{s2 main}
   & \int_1^{X^{\frac{\sigma}{2}}} \frac{2}{t \log X}\widehat{\phi}\left( \frac{2 \log t}{\log X} \right) dt= \int_{0}^{\sigma}\widehat{\phi}(t)dt =\frac{1}{2}\int_{-\sigma}^{\sigma}\widehat{\phi}(t)dt. 
\end{align}
We control the second integral in $(\ref{s2 sum})$ as below.
\begin{align} \label{s2 error}
     &\int_1^{X^{\frac{\sigma}{2}}} \frac{2}{t \log X}\widehat{\phi}\left( \frac{2 \log t}{\log X} \right) d\tilde{\theta}(K,t)
     = -  \int_1^{X^{\frac{\sigma}{2}}} \tilde{\theta}(K,t) \frac{d}{dt}\left(  \frac{2}{t \log X}\widehat{\phi}\left( \frac{2 \log t}{\log X} \right) \right) dt  \nonumber\\
     &= -\int_1^{X^{\frac{\sigma}{2}}} \tilde{\theta}(K,t) \left( -\frac{2}{t^2 \log X}\widehat{\phi}\left( \frac{2 \log t}{\log X}\right) 
     + \frac{4}{t^2 \log^2 X}\widehat{\phi}'\left( \frac{2 \log t}{\log X} \right) \right) dt \nonumber \\
     & \ll_{d, \widehat{\phi}} \frac{1}{\log X}
 \end{align}
 because $\tilde{\theta}(K,t) = -\frac{t^{\beta_K}}{\beta_K}+O_d(t\exp(-c\sqrt{\log t}))$. Hence, the claim follows from $(\ref{s2 main})$ and $(\ref{s2 error})$.
\end{proof}

\section{Proof of Theorem \ref{prop:phi411}. } \label{sec: phi refine}

In this section, we use $\# S$ for the cardinality of the set $S$, and $B$ for the parameter $X$, unlike in the previous section. 
Let $f : \cP(u) \to \cP(w)$ be a morphism between weighted projective spaces, $i_{\fp} : K \to K_{\fp}$ the natural embedding, and $\Omega_{\fp}$ a subset of $\cP(u)(K_{\fp})$.
The main goal of this section is to prove Theorem \ref{prop:phi411}, which estimates
\begin{align} \label{eqn: 411 sketch}
    \#\lcrc{y \in f(\cP(u)(K)) : H_{w, K}(y) \leq B, i_{\fp}(y) \in \Omega_{\fp}}
\end{align}
with an error term whose implied constant does not depend on $q_{\fp}$.
Bruin--Manterola Ayala \cite{BM} gave an estimate of 
\begin{align*}
    \#\lcrc{y \in f(\cP(u)(K)) : H_{w, K}(y) \leq B },
\end{align*}
which is an estimate of (\ref{eqn: 411 sketch}) without a local condition.
Phillips \cite{Phi1} considered the same problem (\ref{eqn: 411 sketch}), although the error term in \cite{Phi1} depends on $q_{\fp}$.

We follow the ideas and strategies in \cite{BM, Phi1} through this section.
However, certain aspects of the arguments seem to require further clarification or modification.
See Remarks \ref{rmk: BM Vpad def}, \ref{rmk: BM cN} for \cite{BM} and Remarks \ref{rmk: phi proj local cond}, \ref{rmk: Phi1 mfad},  \ref{rmk: on Phi M} for \cite{Phi1}.
In light of these considerations, the principal technical distinctions are as follows: Proposition \ref{prop:refinePhi3}, which extends \cite[Proposition 3.2.7]{Phi1},
an error estimate with an implied constant independent of $q_{\fp}$, and a structural discrepancy in the objects being counted, as detailed in Remark \ref{rmk: both compact}.

\subsection{Local conditions} \label{subsec: local cond Omega}

For $u = (u_0, \dots, u_n)$ an $(n+1)$-tuple of positive integers and $k \in K$ (or $K_\fp$), we define the $u$-weighted action
\begin{align*}
    k*_u(x_0, \cdots, x_n) \vcentcolon = (k^{u_0}x_0,  \cdots , k^{u_n}x_n ).
\end{align*}
For a subset $X_{\fp} \subset \cP(u)(K_{\fp})$, we define
\begin{align*}
    X_{\fp}^{\aff} &\vcentcolon = \lcrc{ (x_{\fp, 0}, \cdots, x_{\fp, n}) \in K_{\fp}^{n+1} \setminus \lcrc{0} : [x_{\fp, 0}, \cdots, x_{\fp, n}] \in X_{\fp} } \\
    &= \lcrc{k_{\fp} *_{u} (x_{\fp, 0}, \cdots, x_{\fp, n}) \in K_{\fp}^{n+1} \setminus \lcrc{0} : [x_{\fp, 0}, \cdots, x_{\fp, n}] \in X_{\fp}, k_{\fp} \in K_{\fp}^\times }.
\end{align*}
We denote by $i_{\fp}: K \to K_{\fp}$ the natural embedding and the induced maps.
Let $S$ be a set of primes of $K$.
For $\lcrc{\Omega_{\fp}}_{\fp \in S}$ with $\Omega_{\fp} \subset \cP(u)(K_{\fp})$, we define
\begin{align*}
    \Omega_S \vcentcolon= \bigcap_{\fp \in S} i_{\fp}^{-1}(\Omega_{\fp}),
\end{align*}
which is a subset of $\cP(u)(K)$. Then the quotient map $\mathbb{A}^{n+1}\setminus\{0\}\rightarrow\mathcal{P}(u)$ under the $\ast_{u}$-action gives a commutative cube:
\begin{align*}
    \xymatrix@!0@R=4pc@C=8pc{
    & \Omega_S^\mathrm{aff} \ar@{_(->}[dl] \ar[dd]|\hole \ar[rr] & & \displaystyle\prod_{\mathfrak{p}\in S}\Omega_\mathfrak{p}^\mathrm{aff} \ar@{_(->}[dl] \ar[dd] \\
    K^{n+1}\setminus\{0\} \ar[dd] \ar[rr]^(.35){(i_\mathfrak{p})_{\mathfrak{p}\in S}} & & \displaystyle\prod_{\mathfrak{p}\in S}(K_\mathfrak{p}^{n+1}\setminus\{0\}) \ar[dd] & \\
    & \Omega_S \ar@{_(->}[dl] \ar[rr]|\hole & & \displaystyle\prod_{\mathfrak{p}\in S}\Omega_\mathfrak{p} \ar@{_(->}[dl] \\
    \mathcal{P}(u)(K) \ar[rr] & & \displaystyle\prod_{\mathfrak{p}\in S}\mathcal{P}(u)(K_\mathfrak{p}) &
    }
\end{align*}
whose left, right, and bottom faces are Cartesian squares by definition. Hence, the top face becomes a Cartesian square by the pasting law for Cartesian squares. This is equivalent to saying that
\begin{align*}
    \Omega_S^\mathrm{aff}=\bigcap_{\mathfrak{p}\in S}i_\mathfrak{p}^{-1}(\Omega_\mathfrak{p}^\mathrm{aff}).
\end{align*}
Also, we note that $[x]\in\mathcal{P}(u)(K)$ satisfies $\Omega_S$ if and only if one (equivalently, every) $\widetilde{x}\in K^{n+1}\setminus\{0\}$ with $[x]=[\widetilde{x}]$ satisfies $\Omega_S^\mathrm{aff}$.

In the remaining of the section, we use $q = q_{\fp} =  N_{K/\bQ}(\fp)$ and identify the residue field $\cO_{K, \fp}/\fp\cO_{K, \fp}$ with $\bF_q$.
The following lemma illustrates an example.

\begin{lemma} \label{lem: local cond first exam}
Let $[\widetilde{\alpha_{0}}, \widetilde{\alpha_1}] \in \cP(1,1)(\bF_q)$ and let
\begin{align*}
    \Omega_{\fp} = \lcrc{[x_{\fp, 0}, x_{\fp, 1}] \in \cP(1, 1)(K_{\fp}) : \psi_{\fp}([x_{\fp, 0}, x_{\fp, 1}]) = [\widetilde{\alpha_0}, \widetilde{\alpha_1}]}.
\end{align*}
Then, $\Omega_{\fp}^{\aff} \cap \cO_{K,\fp}^2$ is the disjoint union, over $k \geq 1$, of translates of $\fp^k \times \fp^k$.
\end{lemma}
\begin{proof}
There is a unique integer $k$ such that $\pi_{\fp}^k x_{\fp, 0}, \pi_{\fp}^k x_{\fp, 1} \in \cO_{K, \fp}$ and
\begin{align*}
    \min\lcrc{\ord_{\fp}(\pi_{\fp}^k x_{\fp, 0}), \ord_{\fp}(\pi_{\fp}^k x_{\fp, 1})} = 0.
\end{align*}
By Remark \ref{reduction-P1}, 
\begin{align*}
    \psi_{\fp}([x_{\fp,  0}, x_{\fp, 1}]) = [ (\pi_{\fp}^k x_{\fp, 0}, \pi_{\fp}^k x_{\fp, 1}) \pmod{\fp}].
\end{align*}
Therefore, when both $\widetilde{\alpha_0}$ and $\widetilde{\alpha_1}$ are non-zero, 
\begin{align*}
    \Omega_{\fp} 
    &= \lcrc{[x_{\fp, 0}, x_{\fp, 1}] \in \cP(1, 1)(K_{\fp}) : x_{\fp, 0}, x_{\fp, 1} \in \cO_{K, \fp} \setminus \fp\cO_{K, \fp}, (x_{\fp, 0}, x_{\fp, 1}) \equiv (\zeta\widetilde{\alpha_0}, \zeta \widetilde{\alpha_1}) \textrm{ for some } \zeta \in \bF_q^\times  },
\end{align*}
and
\begin{align*}
    \Omega_{\fp}^{\aff}
    &= \lcrc{ (u\pi_{\fp}^k x_{\fp, 0}, u \pi_{\fp}^k x_{\fp, 1}) : x_{\fp, 0}, x_{\fp, 1} \in \cO_{K, \fp} \backslash \fp \cO_{K, \fp},(x_{\fp, 0}, x_{\fp, 1}) \equiv (\zeta\widetilde{\alpha_0}, \zeta\widetilde{\alpha_1}), u \in \cO_{K, \fp}^\times, k \in \bZ, \zeta \in \bF_q^\times } \\
    &= \bigsqcup_{\zeta \in \bF_q^\times } \bigsqcup_{k \in \bZ}
    \lcrc{ (\pi_{\fp}^k x_{\fp, 0}, \pi_{\fp}^k x_{\fp, 1}) : x_{\fp, 0}, x_{\fp, 1} \in \cO_{K, \fp}\backslash \fp \cO_{K, \fp} ,(x_{\fp, 0}, x_{\fp, 1}) \equiv (\zeta\widetilde{\alpha_0}, \zeta\widetilde{\alpha_1}) }.
\end{align*}
Hence, there is a bijection
\begin{align*}
\Omega_{\fp}^{\aff} \cap \cO_{K, \fp}^{2} \to 
\bigsqcup_{\zeta \in \bF_q^\times }\bigsqcup_{k \geq 0}\lcrc{ (\pi_{\fp}^k x_{\fp, 0}, \pi_{\fp}^k x_{\fp, 1}) : x_{\fp, 0}, x_{\fp, 1} \in \cO_{K, \fp}\backslash \fp \cO_{K, \fp}, (x_{\fp, 0}, x_{\fp, 1}) \equiv (\zeta\widetilde{\alpha_0}, \zeta\widetilde{\alpha_1}) }.
\end{align*}
Let $a_{\fp, 0}, a_{\fp, 1} \in \cO_{K, \fp}$ satisfying $(a_{\fp, 0}, a_{\fp, 1}) \equiv (\widetilde{\alpha_0}, \widetilde{\alpha_1})$ modulo $\fp$ and let $\xi \in \cO_{K, \fp}$, the image of Teichm\"uller character of $\zeta$.
Then the codomain can be written as
\begin{align}
\bigsqcup_{\zeta \in \bF_q^\times}
\bigsqcup_{k \geq 0} \lbrb{
\pi_{\fp}^k *_{(1,1)} \lbrb{
    \lcrc{x_{\fp, 0} \in \cO_{K, \fp} : |x_{\fp, 0} - \xi a_{\fp, 0}|_{\fp} \leq \frac{1}{q}} \times 
    \lcrc{x_{\fp, 1} \in \cO_{K, \fp} : |x_{\fp, 1} - \xi a_{\fp, 1}|_{\fp} \leq \frac{1}{q}}  } }.
\end{align}
Hence for
\begin{align*}
    \Omega_{\fp, 0}^{\aff} = \prod_{j=0}^1 \lcrc{x_{\fp, j} \in \cO_{K, \fp} : |x_{\fp, j} - a_{\fp, j}|_{\fp}  \leq \frac{1}{q}}
    = \prod_{j=0}^1 (a_{\fp, j} + \fp\cO_{K, \fp}),
\end{align*}
we have a bijection 
\begin{align*}
    \xymatrix{\Omega_{\fp}^{\aff} \cap \cO_{K, \fp}^{2} \ar[r]^-\sim 
    &   \displaystyle \bigsqcup_{k \geq 0}\bigsqcup_{\substack{ \zeta \in \bF_q^\times}} \lbrb{\xi\pi_{\fp}^k*_{(1,1)} \Omega_{\fp, 0}^{\aff} } 
    }.
\end{align*}
We note that $\pi_{\fp}^k*_{(1, 1)}\Omega_{\fp, 0}^{\aff}$ is a union of translates of $\fp^{k+1} \times \fp^{k+1}$.
If one of $\widetilde{\alpha_0}, \widetilde{\alpha_1}$ is zero, then one of $a_{\fp, j}$ is zero, and the conclusion remains unchanged.
\end{proof}



For $x_{\fp}\in K_{\fp}^{n+1} \setminus \lcrc{0}$ and a tuple of positive integers $u = (u_0, \cdots, u_n)$, we define
\begin{align*}
    &\epsilon_{u, \fp}(x_{\fp}) \vcentcolon= \min_{0 \leq j \leq n} \left\lfloor \frac{\ord_{\fp}(x_{\fp, j})}{u_j} \right\rfloor.
\end{align*}
For a subset $\Omega_{\fp} \subset \cP(u)(K_{\fp})$ and an integer $k \geq 0$, we define
\begin{align*}
    \Omega_{\fp, k}^{\aff} &\vcentcolon = \lcrc{\pi_{\fp}^k*_{u} x_{\fp} : x_{\fp} \in \cO_{K, \fp}^{n+1}, \epsilon_{u, \fp}(x_{\fp}) = 0,  [x_{\fp}] \in \Omega_{\fp} }.
\end{align*}

We note that it does not depend on the choice of the uniformizer, and that $\Omega_{\fp, k}^{\aff}$ is stable under $u$-weighted $\cO_{K, \fp}^\times$-action.
As in the previous example, for any subset $\Omega_{\fp}$, we have a bijection
\begin{align} \label{eqn: Omega bijec Omega_k}
    \xymatrix{\Omega_{\fp}^{\aff} \cap \cO_{K, \fp}^{n+1} \ar[r]^-\sim & \displaystyle\bigsqcup_{\substack{k\geq 0}} \bigcup_{\zeta \in \mu(K_{\fp})} (\zeta\pi_{\fp}^k*_{u}\Omega_{\fp, 0}^{\aff})
    = \bigsqcup_{\substack{k\geq 0}} \bigcup_{\zeta \in \mu(K_{\fp})} (\zeta*_u\Omega_{\fp, k}^{\aff}).}
\end{align}
When $u \neq (1, \cdots, 1)$, the union over $\zeta \in \mu(K_{\fp})$ may not be disjoint.

\begin{definition} \label{def: irred proj local}
(i) A subset $\Omega_{\fp}\subset \cP(u)(K_{\fp})$ is called a projective local condition at $\fp$, and we say that $[x] \in \cP(u)(K)$ satisfies $\Omega_{\fp}$ if $i_{\fp}([x]) \in \Omega_{\fp}$. \\
(ii) A projective local condition $\Omega_{\fp}$ is called irreducible\footnote{The term ``irreducible'' may not be ideal, as an irreducible local condition can still be written as a union of other irreducible local conditions. However, we adopt this terminology to remain consistent with previous works. A reader may think of an irreducible local condition as a generalization of a box-type set.} if 
\begin{align*}
    \Omega_{\fp, 0}^{\aff} = \prod_{j=0}^{n} \lcrc{x_{\fp, j} \in \cO_{K, \fp} : |x_{\fp, j} - a_{\fp, j}|_{\fp} \leq \omega_{\fp, j}}
    = \prod_{j=0}^n (a_{\fp, j} + \fp^{r_{\fp, j}}\cO_{K, \fp})
\end{align*}
for some $a_{\fp, j} \in \cO_{K, \fp}$, $r_{\fp, j} \geq 0$, and $\omega_{\fp, j} = q^{-r_{\fp, j}}$.
\end{definition}

\begin{remark} \label{rmk: phi proj local cond}
Our definition is based on {\cite[Definition 4.0.3]{Phi2}}, but the order of presentation differs: we first define $\Omega_{\fp, k}^{\aff}$ and then use it to describe the irreducible projective local condition.
Equation (\ref{eqn: Omega bijec Omega_k}) may also serve as a more precise replacement for the one in \cite[Definition 4.0.3]{Phi2}.
See also (\ref{eqn: Omega decomp Omega_k disj}).
\end{remark}


The following is elementary, but will be used repeatedly throughout the remainder of the paper.

\begin{lemma} \label{lem: pm cap xpn}
Let $\fp$ be a prime ideal of $K$, let $m \leq n$ be non-negative integers, and let $x \in \mathcal{O}_{K, \fp}$. Then,
\begin{align*}
\mathfrak{p}^m\mathcal{O}_{K,\mathfrak{p}}\cap(x+\mathfrak{p}^n\mathcal{O}_{K,\mathfrak{p}})=\left\{\begin{array}{ll}
    x+\mathfrak{p}^n\cO_{K, \fp} & \textrm{if $x\in\mathfrak{p}^m\mathcal{O}_{K,\mathfrak{p}}$}, \\
    \emptyset & \textrm{otherwise.}
    \end{array}\right.
\end{align*}
\end{lemma}

\subsection{Counting points in a lattice with local conditions}

Let $\Lambda$ be a free $\bZ$-module of rank $d$, and
\begin{align*}
    \Lambda_{\infty} \vcentcolon= \Lambda \otimes_{\bZ} \bR, \qquad 
    \Lambda_{p} \vcentcolon= \Lambda \otimes_{\bZ} \bZ_p.
\end{align*}

For each prime $p$, we choose a Haar measure $m_p$ on $\Lambda_p$, normalized so that $m_p(\Lambda_p) = 1$ for all but finitely many $p$, 
and choose a Haar measure $m_{\infty}$ on $\Lambda_{\infty}$.

Let $i_p$ be the natural embedding $\Lambda \to \Lambda_p$.
Suppose that a non-canonical embedding $i_{\Lambda, p} : \Lambda_p \to \bZ_p^{d}$ is given, which is usually not an isomorphism $\Lambda_p \cong \bZ_p^d$ which comes from the choice of basis of $\Lambda$, and let $a_{p, j} \in \bZ_p$ and $\omega_{p, j} \in \lcrc{p^k : k \in \bZ_{\leq 0}}$.
An irreducible local condition on $\Lambda$ at $p$ with respect to $i_{\Lambda, p}$ is given by
\begin{align} \label{eqn: Omega p def}
\Omega_p = \prod_{i=1}^{d} B_i \quad \textrm{where} \quad 
    B_i =  \lcrc{x_{p, i} \in \bZ_p : |x_{p, i} - a_{p, i}|_p \leq \omega_{p, i}}.
\end{align}
In other words, $x \in \Lambda$ satisfies $\lcrc{\Omega_p}_{p \in S}$ if and only if
\begin{align*}
    i_{\Lambda, p} (i_{p}(x)) \in \Omega_{p}
\end{align*}
for each $p \in S$.


We say that a subset $R \subset \Lambda_{\infty}$ is \emph{definable} if it is definable in some o-minimal structure (cf. \cite[\S 3.1]{Phi2}).
For infinite places, we consider $\Omega_{\infty}\subset \Lambda_{\infty}$ that is bounded and definable.
For a region $R \subset \Lambda_{\infty} \cong \bR^{n+1}$, an $(n+1)$-tuple of positive integers $w = (w_0, \cdots, w_n)$, and a positive number $B$, we define 
\begin{align*}
    B*_w R \vcentcolon=\lcrc{(t^{w_0}x_0, \cdots, t^{w_n}x_n ) : 0 \leq t \leq B, (x_0, \cdots, x_n) \in R \subset \bR^{n+1} }.
\end{align*}
Also for $w = (w_0, \cdots, w_n)$, we define $|w|$ (resp. $w_{\min}, w_{\max}$) by the sum (resp. the minimum, the maximum) of $w_j$.

\begin{lemma}[{\cite[Lemma 3.2.4]{Phi2}}] \label{lem: Phi2 3.2.4}
Let $s, \Lambda, \Lambda_p, \Lambda_{\infty}, m_p, m_{\infty}$, and $\Omega_{\infty}$ be as above.
For a finite set of primes $S$ and irreducible local conditions $\lcrc{\Omega_p}_{p \in S}$ and a positive real $B$, we have
\begin{align*}
    &\#\lcrc{x \in \Lambda \cap (B*_w \Omega_{\infty}) : i_p(x) \in \Omega_p \textrm{ for all } p \in S} \\
    & \qquad = \frac{m_{\infty}(\Omega_{\infty})}{m_{\infty}(\Lambda_{\infty}/\Lambda)} 
    \prod_{p \in S} \frac{m_p(\Omega_p \cap \Lambda_p)}{m_p(\Lambda_p)} B^{|w|}
    + O\lbrb{ \max_{0 \leq j \leq n} \lcrc{\prod_{p \in S} \omega_{p, j}^{-1}} \frac{m_p(\Omega_p \cap \Lambda_p)}{m_p(\Lambda_p)} B^{|w| - w_{\min}} }.
\end{align*}
The implied constant only depends on $\Lambda, \Omega_{\infty}$, and $w$.
\end{lemma}
\begin{proof}
In \cite[Lemma 3.2.4]{Phi1}, the local condition is defined by the subset of $\Omega_p$, but our local condition is more generally a subset of $\bZ_p^{\rk(\Lambda)}$.
Hence, we should exchange $m_{p}(\Omega_p)/m_p(\Lambda_p)$ by $m_{p}(\Omega_p \cap \Lambda_p)/m_p(\Lambda_p)$.
The other part of the proof is unchanged.
\end{proof}

Let $K_{\infty} \vcentcolon= K \otimes_{\bQ}\bR$ and let $d = [K:\bQ]$. Then an ideal in $\cO_K$ can be regarded as a lattice of rank $d$ in $K_{\infty}$. Furthermore,
\begin{align*}
    \fa^u = \fa^{u_0} \times \fa^{u_2} \times \cdots \times \fa^{u_n} \subset K_{\infty} \times K_{\infty} \times \cdots \times K_{\infty}
\end{align*}
is also a lattice of rank $(n+1)d$ in $K_{\infty}^{n+1}$.

One can also define an irreducible local condition on the affine space, applicable to $\fa^u$, or more generally, to a lattice in $K^{n+1}$.
Let $a_{\fp, j} \in \cO_{K, \fp}$, $r_{\fp, j} \geq 0$, and let $\omega_{\fp, j} = q_{\fp}^{-r_{\fp, j}}$.
A collection of sets $\lcrc{\Omega_{\fp}^a}_{\fp \in S}$, where
\begin{align} \label{eqn: Omega fp def}
    \Omega_{\fp}^a = \prod_{j=0}^n  \lcrc{x_{\fp, j} \in \cO_{K, \fp}  : |x_{\fp, j} - a_{\fp, j}|_{\fp} \leq \omega_{\fp, j}} = \prod_{j=0}^n (a_{\fp, j} + \fp^{r_{\fp, j}}\cO_{K, \fp}) \subset \cO_{K, \fp}^{n+1}
\end{align}
is called an irreducible affine local condition.
Let
\begin{align*}
    \fa_\fp \vcentcolon= \fa \cO_{K, \fp} = (\fp\cO_{K, \fp})^{\ord_{\fp} (\fa) }.
\end{align*}
Then there is a natural embedding $i_{\fp} : \fa^u \to \fa_{\fp}^u$, and $x \in \fa^u$ satisfies the local condition $\Omega_{\fp}^a$ if and only if $i_{\fp}(x) \in \Omega_{\fp}^a$.
We also note that the elements of $\fa^u$ satisfying the local condition $\Omega_{\fp}^a$ are given by $\fa^u \cap i_{\fp}^{-1}(\Omega_{\fp}^a)$, not merely by $i_{\fp}^{-1}(\Omega_{\fp}^a)$.
For $\Omega = \lcrc{\Omega_{\fp}^a}_{\fp \in S}$, we define
\begin{align*}
    \Omega_S \vcentcolon = \bigcap_{\fp \in S} i_{\fp}^{-1}(\Omega_{\fp}^a)
\end{align*}
in analogy with the definition in the projective local condition case.

To summarize, we have defined four types of irreducible local conditions.
We first introduced an irreducible local condition $\Omega_{\fp} \subset \cP(u)(K_{\fp})$ on projective space and its preimage $\Omega_{\fp}^{\aff} \subset K_{\fp}^{n+1}$ in affine space.
For a lattice $\Lambda$ of rank $d$ with an embedding $i_{\Lambda, p} : \Lambda_p \to \bZ_p^{d}$, we defined an irreducible local condition $\Omega_p \subset \bZ_p^{d}$.
For a lattice $\Lambda \subset K^{n+1}$, we define an irreducible affine local condition $\Omega_{\fp}^a \subset K_{\fp}^{n+1}$.
We note that only the third, $\Omega_p$, is non-canonical and depends on the choice of $i_{\Lambda, p}$.

Let $m_\mathfrak{p}$ be the usual Haar measure on $\cO_{K, \fp}$, normalized so that $m_{\fp}(\cO_{K, \fp}) = 1$, and let 
\begin{align*}
    \iota_{C, p} : \cO_K \otimes_{\bZ} \bZ_p \to \prod_{\fp \mid p} \cO_{K, \fp}
\end{align*}
be the canonical isomorphism from the Chinese remainder theorem.
In this section, we always use $\iota$ for an isomorphism and $i$ for an embedding.
We define a Haar measure on $\cO_{K} \otimes \bZ_p$ using $\iota_{C, p}$.
In other words, if $\iota_{C, p}(S) = \prod_{\fp \mid p} S_{\fp}$ for $S \subset \cO_K \otimes_{\bZ} \bZ_p$, we define
\begin{align*}
    m_p(S) \vcentcolon= \prod_{\fp} m_{\fp}(S_{\fp}).
\end{align*}
It also induces a Haar measure on $(\cO_K \otimes_{\bZ}\bZ_p)^{n+1}$ which is denoted by $m_p$.
Then, $m_{p}(\fa^u_p)$ is well-defined via
\begin{align*}
    \fa^u_p  = \fa^u\otimes_{\bZ} \bZ_p \subset (\cO_{K} \otimes_{\bZ}\bZ_p)^{n+1}.
\end{align*}
We have $m_p(\fa_p^u) = 1$ if no prime $\fp$ above $p$ divides $\fa$.
Also, the standard Haar measure on $\bR$ and $\bC$ induces the Haar measure on $K_\infty$ and $K_\infty^{n+1}$.

In the next lemma, we compare the local conditions $\Omega_p$ for $\Lambda$ and $\Omega_{\fp}^a$ for $\fa^u$, when $\Lambda = \fa^u$.
A choice of basis of $\cO_K$ gives an isomorphism $\iota_K : \cO_K \to \bZ^d$ and $\iota_{K, p} \vcentcolon = \iota_{K} \otimes_{\bZ}\bZ_p$, and also induces an isomorphism
\begin{align*}
\Phi : 
\xymatrix{ \displaystyle
    \prod_{\fp \mid p}\cO_{K, \fp} \ar[r]^-{\iota_{C, p}^{-1}} & \cO_{K} \otimes_{\bZ} \bZ_p \ar[r]^-{\iota_{K, p}} &  \bZ_p^d
    }.
\end{align*}
For a prime $\fp$ of $K$ dividing $p$, recall that we denote its residue degree and the ramification index by $f(\fp | p)$, and $e(\fp | p)$.
Let $\iota_{K, \fp} : \cO_{K, \fp} \to \bZ_p^{e(\fp | p)f(\fp |p)}$ be a usual isomorphism.
Then, $\iota_{K, \fp}$ preserves the Haar measure, so does $\Phi = \prod_{\fp \mid p} \iota_{K, \fp}$.

\begin{lemma} \label{lem: Omega fp p}
We fix an isomorphism $\iota_K: \cO_K \to \bZ^d$, thereby also fixing the induced map $\Phi$.
\\
(i) For each $\fp \mid p$, let $a_{\fp}$ be an element of $\cO_{K, \fp}$, and let $r_{\fp}\geq 0$ be an integer.
Then, there exist $a_{p, i} \in \bZ_p$ and $k_i \geq 0$ for $i = 1, \cdots, d$ such that
\begin{align*}
    \Phi\lbrb{\prod_{\fp \mid p} (a_{\fp} + \fp^{r_{\fp}}\cO_{K, \fp}) } = \prod_{i=1}^d (a_{p, i} + p^{k_i}\bZ_p), \quad \textrm{and} \quad
    \sum_{\fp \mid p} r_{\fp}f(\fp|p) = \sum_{i=1}^d k_i.
\end{align*}
(ii) Given a set of irreducible local conditions $\Omega = \lcrc{\Omega_{\fp}^a}_{\fp \mid p}$ for $\fa^u$, there is an irreducible local condition $\Omega_p$ and an embedding $i_{\Lambda, p} : \fa^u_p \to \bZ_p^{d(n+1)}$ such that $x \in \fa^u$ satisfies $\Omega = \lcrc{\Omega_{\fp}^a}_{\fp \mid p}$ if and only if $x$ satisfies $\Omega_p$, and 
\begin{align*}
m_p(\Omega_p) = \prod_{\fp \mid p}m_{\fp}(\Omega_{\fp}^a \cap \fa_{\fp}^u).
\end{align*}
\end{lemma}
\begin{proof}
(i) We first assume that $a_{\fp} = 0$ for all $\fp \mid p$.
Let $\fb = \prod_{\fp \mid p} \fp^{r_{\fp}}$. Then
\begin{align*}
    \Phi\lbrb{\prod_{\fp \mid p} \fp^{r_{\fp}} \cO_{K, \fp} } = \iota_{K, p}(\fb \otimes_{\bZ} \bZ_p) = \prod_{i=1}^d p^{k_i} \bZ_p
\end{align*}
for some $k_i \geq 0$. Hence,
\begin{align*}
\left. \prod_{\fp \mid p} \cO_{K, \fp}  \middle/ \prod_{\fp \mid p} \fp^{r_{\fp}}\cO_{K, \fp} \right. \cong
\left. \bZ_p^d \middle/  \prod_{i=1}^d p^{k_i} \bZ_p \right. .
\end{align*}
It gives the result, when $a_\fp = 0$ for $\fp \mid p$.
In the general case, a collection of translations by elements $ a_{\mathfrak{p}} \in \mathcal{O}_{K,\mathfrak{p}} $ at pairwise relatively prime primes $\fp$ corresponds, via the Chinese remainder theorem, to a translation by a single element of $\cO_K$.
Under the isomorphism $\Phi$, the translation in $\cO_K$ induces a translation in $\bZ_p^d$.

(ii)
Tensoring with $\bZ_p$ for a given $\iota_K : \cO_K \to \bZ^d$, we obtain $\iota_{K, p} : \cO_K \otimes_{\bZ}\bZ_p \to \bZ_p^d$, and similarly, an isomorphism $\cO_K^{n+1} \otimes_{\bZ}\bZ_p \to \bZ_p^{(n+1)d}$, which is also denoted by $\iota_{K, p}$.
This construction ensures compatibility between $\iota_K$ and $\iota_{K, p}$.
Therefore, the right square in the diagram below commutes.
\begin{align} \label{eqn: Omega_fp to Omega_p}
    \begin{aligned}
        \xymatrixcolsep{3.5pc}\xymatrix{
    \displaystyle
      \prod_{\fp \mid p} \cO_{K, \fp}^{n+1} \ar[r]^-{\cong} & \cO_K^{n+1} \otimes_{\bZ} \bZ_p \ar[r]^-{\iota_K} & (\bZ^d)^{n+1}\otimes_{\bZ}\bZ_p \ar[r]^-{\cong} & \bZ_p^{(n+1)d} \\
      \fa^u = \Lambda \ar[u]^-{\prod_{\fp \mid p} i_{\fp}} \ar[r]^-{-\otimes_{\bZ}\bZ_p} \ar@{^{(}->}[ur]^-{-\otimes_{\bZ}\bZ_p} & \Lambda_p \ar[rr]^-{\iota_{K, p}} \ar[u]^-{\cup} & & \iota_{K, p}(\Lambda_p)  \ar@{^{(}->}[u]
    }
    \end{aligned}
\end{align}
The two triangles on the left side naturally commute, and the morphism from $\prod_{\fp \mid p} \cO_{K, \fp}^{n+1}$ to $\bZ_p^{(n+1)d}$ is precisely the copy of $\Phi$.
We also denote $\Phi$ for this map.

For the irreducible affine local conditions $\Omega = \lcrc{\Omega_{\fp}^a}_{\fp \mid p}$, we recall that $x \in \fa^u$ satisfies $\Omega$ if and only if its image in $\prod_{\fp \mid p} \cO_{K, \fp}^{n+1}$ is lies in $\prod_{\fp \mid p} \Omega_{\fp}^a$.
Since the diagram commutes, we have
\begin{align*}
    \Phi \lbrb{ \prod_{\fp \mid p} (\Omega_{\fp}^a \cap \fa_\fp^{u}) } \subset \iota_{K, p}(\Lambda_p).
\end{align*}
We denote it by $\Omega_p$, a local condition for $\Lambda = \fa^u$, with an embedding $\Lambda \to \iota_{K, p}(\Lambda_p) \subset \bZ_p^{(n+1)d}$.
Then, $x \in \fa^u$ satisfies $\lcrc{\Omega_{\fp}^a}_{\fp \mid p}$ if and only if it satisfies $\Omega_p$.
By (i) and the definition of $m_p$, and since $\Phi$ preserves the Haar measure, it follows that
\begin{align*}
    \prod_{\fp \mid p} m_{\fp} (\Omega_{\fp}^a \cap \fa_\fp^{u})
    = m_p \lbrb{\prod_{\fp \mid p} (\Omega_{\fp}^a \cap \fa_{\fp}^u )  }
    = m_p \lbrb{\Phi \lbrb{ \prod_{\fp \mid p} (\Omega_{\fp}^a \cap \fa_\fp^{u}) }}
    =m_{p}(\Omega_p).
\end{align*}
This proves (ii).
\end{proof}

For a bounded definable $\Omega_{\infty} \subset K_{\infty}^{n+1}$, there is a natural weight
\begin{align*}
\widetilde{w} \vcentcolon= (w_0, \cdots, w_0, w_1, \cdots, w_1, \cdots, w_n, \cdots, w_n) \in \bZ^{(n+1)d}   
\end{align*}
induced by $w$. We note that $|\widetilde{w}| = d|w|$, which appears in the following propositions.
Following \cite{Phi1, Phi2}, we define
\begin{align*}
    B*_{\widetilde{w}} \Omega_{\infty} = \lcrc{(t^{w_0}x_0, \cdots, t^{w_0}x_d, t^{w_1}x_{d+1}, \cdots, t^{w_n}x_{(n+1)d}) : 0\leq t \leq B,  x = (x_0, \cdots, x_{(n+1)d}) \in \Omega_{\infty}  }.
\end{align*}
For an irreducible affine local condition $\Omega = \lcrc{\Omega_{\fp}^a}_{\fp \in S}$, we denote by $\omega_{\fp, j}(\Omega)$ the Haar measure of the $j$-th component of $\Omega_{\fp}^a$, which agrees with $\omega_{\fp, j}$ in (\ref{eqn: Omega fp def}).

\begin{proposition}[{\cite[Proposition 3.2.7]{Phi2}}] \label{prop:phi3 327}
Let $u$ and $w$ be $(n+1)$-tuples of positive integers.
Let $\Omega_\infty \subset K_\infty^{n+1}$ be a bounded definable subset, let $B>0$, and let $\Omega = \lcrc{\Omega_{\fp}^a}_{\fp \in S}$ be a set of irreducible local conditions over a finite set $S$. Then we have
\begin{align*}
    &\# \lcrc{ x\in  \fa^{u} \cap (B*_{\widetilde{w}}\Omega_{\infty}) :  i_\fp(x) \in \Omega_\fp^a \textrm{ for all primes } \fp \in S  }
    = \kappa(\fa^u, \Omega) B^{d|w|} + O\lbrb{\epsilon(\fa^u, \Omega) B^{d|w| - w_{\min}}},
\end{align*}
where
\begin{align*}
    &\kappa(\fa^u, \Omega) = \frac{m_{\infty}(\Omega_{\infty})}{m_{\infty}(K_{\infty}^{n+1}/\fa^u)} \prod_{\fp \in S} 
 \frac{m_{\fp}(\Omega_\fp^a \cap \fa_{\fp}^u) }{m_{\fp}(\fa_{\fp}^u)}, \,\,
    \epsilon(\fa^u, \Omega) =  \max_{0 \leq j\leq n} \lcrc{\prod_{\fp \in S} \omega_{\fp, j}^{-1}(\Omega) }\prod_{\fp \in S} 
 \frac{m_{\fp}(\Omega_\fp^a \cap \fa_{\fp}^u ) }{m_{\fp}(\fa_{\fp}^u)}.
\end{align*}
\end{proposition}
\begin{proof}
It suffices to assume that all primes $\fp \in S$ lie above a fixed rational prime $p$. 
Moreover, if there exists a prime $\fp' \mid p$ with $\fp' \notin S$, we set $\Omega_{\fp'}^a = \cO_{K,\fp'}$ and accordingly enlarge $S$ and $\Omega$. 
Since we are counting only integral points, these processes are harmless.
By Lemma  \ref{lem: Omega fp p}, there exists a local condition $\Omega_p$ such that $x \in \fa^u$ satisfies $\Omega_{\fp}^a$ for all $\fp \mid p$ if and only if it satisfies $\Omega_p$, and
\begin{align*}
m_p(\Omega_p) = \prod_{\fp \mid p} m_{\fp}(\Omega_{\fp}^a \cap \fa_{\fp}^u).
\end{align*}

We will use Lemma \ref{lem: Phi2 3.2.4} for $\Lambda = \fa^u$, $\Lambda_{\infty} = K_{\infty}^{n+1}$, its bounded definable subset $\Omega_{\infty}$, and local conditions $\Omega_p$.
We have $|\widetilde{w}| = d|w|$, $\widetilde{w}_{\min} = w_{\min}$, and 
\begin{align*}
    m_p(\Lambda_p) = m_p(\fa^u\otimes_{\bZ} \bZ_p) = m_p\lbrb{\prod_{\fp \mid p} \fa_{\fp}^u} = \prod_{\fp \mid p} m_{\fp}(\fa_\fp^u).
\end{align*}
Therefore, 
\begin{align*}
    \frac{m_p(\Omega_p \cap \Lambda_p)}{m_p(\Lambda_p)} 
    = \prod_{\fp \mid p} \frac{m_{\fp}(\Omega_{\fp}^a \cap \fa_{\fp}^u) }{m_{\fp}(\fa_\fp^u)}
    = \prod_{\substack{\fp \in S }} \frac{m_{\fp}(\Omega_{\fp}^a \cap \fa_{\fp}^u) }{m_{\fp}(\fa_\fp^u)}
\end{align*}
since $\Omega_{\fp'}^a = \cO_{K, \fp'}$ if $\fp' \not\in S$.
\end{proof}

When it is necessary to distinguish between two types of prime ideals, those appearing in local conditions and those denoting general primes, we write $\fp$ for the former and $\fq$ for the latter.\footnote{When local conditions are not involved, we typically use $\fp$ to denote a general prime ideal. We write $p$ for the rational prime below $\fp$, but note that $q = q_{\fp}$ denotes the norm $N_{K/\bQ}(\fp)$, not a rational prime below $\fq$.}
Let $v_{\fq} = (v_{\fq, 0}, \cdots, v_{\fq, n})$ be an $(n+1)$-tuple of non-negative integers satisfying $v_{\fq} = (0, \cdots, 0)$ except finitely many $\fq$, and let
\begin{align*}
    \Lambda_{\fq} = \prod_{j=0}^n \fq^{v_{\fq, j}}\cO_{K, \fq} \subset \cO_{K, \fq}^{n+1}, \quad \textrm{and} \quad
    \Lambda = \bigcap_{\fq} i_{\fq}^{-1}(\Lambda_{\fq}).
\end{align*}
We need to generalize Proposition \ref{prop:phi3 327} to this setting of $\Lambda$.
In this case, $\Lambda$ is a lattice of rank $(n+1)d$, and
\begin{align*}
    \Lambda_p = \Lambda \otimes_{\bZ} \bZ_p = \prod_{\fp \mid p} \Lambda_{\fp}
\end{align*}
satisfies $m_p(\Lambda_p) = 1$ for all but finitely many $p$.

An element $x \in \Lambda$ satisfies an irreducible local condition $\Omega_{\fp}^a$ defined in (\ref{eqn: Omega fp def}) if and only if $i_{\fp}(x) \in \Lambda_{\fp}$ is in $\Omega_{\fp}^a$, which is equivalent to $i_{\fp}(x_{j}) \equiv a_{\fp, j} \pmod{\fp^{r_{\fp, j}}}$, where $N_{K/\bQ}(\fp^{r_{\fp, j}}) = \omega_{\fp, j}$.
On the other hand, a non-canonical embedding $i_{\Lambda, p} : \Lambda_p \to \bZ_p^{(n+1)d}$, together with the boxes (\ref{eqn: Omega p def}), defines the local condition $\Omega_p$.
More precisely, $x \in \Lambda$ satisfies the local condition $\Omega_p$ if and only if $i_{\Lambda, p}i_p(x) \in \Omega_{p}$. 

\begin{proposition} \label{prop:refinePhi3}
Let $v_{\fq} = (v_{\fq, 0}, \cdots, v_{\fq, n})$ be an $(n+1)$-tuple of non-negative integers satisfying $v_{\fq} = (0, \cdots, 0)$ except finitely many $\fq$, let $B>0$, and let $\Omega_{\infty}$ be a bounded definable subset of $K_{\infty}^{n+1}$.
For
\begin{align*}
    \Lambda = \bigcap_{\fq} i_{\fq}^{-1} \lbrb{\prod_{j=0}^n \fq^{v_{\fq, j}}\cO_{K, \fq}},
\end{align*}
a finite set of irreducible local conditions $\Omega = \lcrc{\Omega_{\fp}^a}_{\fp \in S}$, 
and a fixed vector $v \in K_{\infty}^{n+1}$, we have
\begin{align*}
    \# \lcrc{x \in (v + \Lambda) \cap (B*_{\widetilde{w}}\Omega_{\infty}) : i_{\fp}(x) \in \Omega_\fp^a \textrm{ for all } \fp \in S}
    = \kappa(\Lambda, \Omega) B^{d|w|} + O\lbrb{ \epsilon(\Lambda, \Omega) B^{d|w| - w_{\min}}},
\end{align*}
where
\begin{align*}
    \kappa(\Lambda, \Omega) = \frac{m_{\infty}(\Omega_{\infty})}{m_{\infty}(K_{\infty}^{n+1}/\Lambda)}
    \prod_{\fp \in S} \frac{m_{\fp}(\Omega_{\fp}^a \cap \Lambda_{\fp})}{m_{\fp}(\Lambda_{\fp})},
    \quad
    \epsilon(\Lambda, \Omega) = \max_{0 \leq j \leq  n} \lcrc{\prod_{\fp \in S} \omega_{\fp, j}^{-1}(\Omega ) } \prod_{\fp \in S} \frac{m_{\fp}(\Omega_{\fp}^a \cap \Lambda_{\fp})}{m_{\fp}(\Lambda_{\fp})}.
\end{align*}
The implied constant depends on $\Lambda, \Omega_{\infty}$, and $v$, but not on $q_{\fp}$.
\end{proposition}
\begin{proof}
The diagram (\ref{eqn: Omega_fp to Omega_p}) also works for $\Lambda \subset \cO_{K}^{n+1}$.
Hence, there is
\begin{align*}
    \Omega_p \vcentcolon = \Phi\lbrb{\prod_{\fp \mid p} (\Omega_{\fp}^a \cap \Lambda_{\fp})} \subset i_{\Lambda, p}(\Lambda_p)
\end{align*}
such that $x \in \Lambda$ satisfies $\lcrc{\Omega_{\fp}^a}_{\fp \mid p}$ if and only if $x$ satisfies $\Omega_p$, and
\begin{align*}
    m_p(\Omega_p) = \prod_{\fp \mid p} m_{\fp} (\Omega_{\fp}^a \cap \Lambda_{\fp}).
\end{align*}
We can apply Lemma \ref{lem: Phi2 3.2.4} to $\Lambda$, $\Omega_{\infty}$, and $\Omega_{p}$ since $m_p(\Lambda_p) = 1$ for all but finitely many $p$.
The remainder of the proof proceeds exactly as in Proposition \ref{prop:phi3 327}.
\end{proof}

When $\Lambda$ is explicitly given, we will use $\kappa(\Omega), \epsilon(\Omega)$ instead of $\kappa(\Lambda, \Omega), \epsilon(\Lambda, \Omega)$.

\subsection{Counting points in the weighted projective space with local conditions}
We review some definitions and properties for counting rational points in weighted projective space.
Let $u, w$ be an $(n+1)$-tuple of positive integers, and let $M_{K, 0}$ (resp. $M_{K, \infty}$) be the set of finite (resp. infinite) places of $K$.
For $x \in K^{n+1} \setminus \lcrc{0}$, we define 
\begin{align*}
    \mathfrak{I}_w(x)\vcentcolon=\bigcap_{\substack{\mathfrak{a}\subseteq K \\ x\in\mathfrak{a}^{w_0}\times\cdots\times\mathfrak{a}^{w_n}}}\mathfrak{a}
\end{align*}
where $\mathfrak{a}\subset K$ runs over fractional ideals (cf. \cite[\S 3]{BN}).
We note that $\fI_w(x)$  is characterized by
\begin{align*}
    \fI_w(x)^{-1} = \lcrc{a \in K : a^{w_j}x_j \in \cO_K \textrm{ for } j= 0, \cdots, n}.
\end{align*}
In other words, an element $a \in \fI_w(x)^{-1}$ satisfies $w_j \ord_v(a) + \ord_v(x_j) \geq 0$ for any $v \in M_{K, 0}$ and all $j = 0, \cdots, n$.
Hence, we have a prime factorization
\begin{align*}
    \fI_w(x) = \prod_{v \in M_{K, 0}} \fp_v^{ \min\limits_{0 \leq j \leq n } \left\lfloor \frac{\ord_v(x_j)}{w_j}\right\rfloor},
\end{align*}
where $\fp_v$ is the prime ideal corresponding to $v \in M_{K, 0}$.
For $x_{\fp}\in K_{\fp}^{n+1} \setminus \lcrc{0},$ we have defined
\begin{align*}
    &\epsilon_{w, \fp}(x_{\fp}) \vcentcolon= \min_{0 \leq j \leq n} \left\lfloor \frac{\ord_{\fp}(x_{\fp, j})}{w_j} \right\rfloor.
\end{align*}
The local size function $\fI_{w, \fp}$ is defined by 
\begin{align*}
    \fI_{w, \fp}(x_{\fp}) \vcentcolon=\fp^{\epsilon_{w, \fp}(x_{\fp})}\cO_{K, \fp}.
\end{align*}

For $f : \cP(u) \to \cP(w)$, we define the reduced degree $e(f)$ following \cite[Definition 4.2]{BN}. In the case of $f = \phi_{\Gamma}$, this coincides with the definition given earlier in (\ref{eqn:def eGam}).
The argument of \cite[Lemma 4.1]{BN} says that we may write
\begin{align*}
    \xymatrix{f:\mathcal{P}(u) \ar[r] & \mathcal{P}(w) & [x_0,\cdots,x_n] \ar@{|->}[r] & [f_0(x),\cdots,f_n(x)]}
\end{align*}
where $f_0,\cdots,f_n\in K[x_0,\cdots,x_n]$ are weighted homogeneous polynomials of degree $e(f)w_j$, and the ideal $\sqrt{(f_0,\cdots,f_n)}$ contains $(x_0,\cdots,x_n)$.
For $x \in K^{n+1}$, we define a defect ideal as
\begin{align*}
    \delta_f(x) \vcentcolon= \fI_{w}(f(x)) \fI_{u}(x)^{-e(f)}.
\end{align*}
By \cite[Lemma 1.6]{BM} and $\deg f_j = e(f)w_j$, we have
\begin{align} \label{eqn: weighted action fI f delta}
    \fI_w(\lambda*_wx) = \lambda \fI_w(x), \qquad
    f(\lambda*_u x) = \lambda^{e(f)}*_wf(x), \qquad \delta_f(\lambda*_u x) = \delta_f(x).
\end{align}
Hence, $\delta_f$ is a well-defined fractional ideal when $x \in \cP(u)(K)$ (cf. \cite[Remark 2.3]{BM}).

We denote the set of defects by
\begin{align*}
    \mathfrak{D}_f \vcentcolon = \lcrc{\delta_f(x) : x \in \cP(u)(K)}.
\end{align*}
If $\mathfrak{D}_f$ is finite, we say that $f$ has a finite defect.
We denote by $S_f$ the set of prime ideals dividing some ideal in $\fD_f$.
Recall that the height function $H_w = H_{w, K}$ satisfies
\begin{align*}
    H_{w}(x) = \frac{H_{w, \infty}(x)}{N_{K/\bQ}(\fI_{w}(x))}, \qquad
    \textrm{where} \qquad 
    H_{w, \infty}(x)\vcentcolon= \prod_{v \in M_{K, \infty}} \max_{0 \leq j \leq n}|x_j|_{v}^{\frac{1}{w_j}}
\end{align*}
(cf. \cite[Definition 1.7]{BM}).
Let
\begin{align*}
    V_\mathfrak{p}^{\mathfrak{a},\mathfrak{d}}\vcentcolon=\left\{x_\mathfrak{p}\in K_\mathfrak{p}^{n+1}\setminus\{0\} : \epsilon_{u,\mathfrak{p}}(x_\mathfrak{p})\geq\mathrm{ord}_\mathfrak{p}(\mathfrak{a}),\ \mathrm{ord}_\mathfrak{p}(\mathfrak{I}_{w,\mathfrak{p}}(f(x_\mathfrak{p}))) -\mathrm{ord}_\mathfrak{p}(\mathfrak{I}_{u,\mathfrak{p}}(x_\mathfrak{p}))^{e(f)}) \geq \mathrm{ord}_\mathfrak{p}(\mathfrak{d})\right\}.
\end{align*}
Then
\begin{align*}
    V_{\fp}^{\fa, \fd} = \lcrc{x_{\fp} \in K_{\fp}^{n+1} \setminus \lcrc{0} :
    \epsilon_{u, \fp}(x_{\fp}) \geq \ord_{\fp}(\fa), \epsilon_{w, \fp}(f(x_{\fp})) \geq e(f)\epsilon_{u, \fp}(x_{\fp}) + \ord_{\fp}(\fd)}.
\end{align*}
We also define the global object
\begin{align*}
    V^{\fa, \fd} \vcentcolon = \bigcap_{\fp} i_{\fp}^{-1}(V_{\fp}^{\fa, \fd}).
\end{align*}

For a prime ideal $\fp$ and $t \geq 0$, we define
\begin{align*}
    V_{\fp, t}^{\fa, \fd} \vcentcolon = \lcrc{x_{\fp} \in \cO_{K, \fp}^{n+1} \setminus \lcrc{0} : \epsilon_{u, \fp}(x_{\fp}) = \ord_{\fp}(\fa) + t, \,
    \epsilon_{w, \fp}(f(x_{\fp})) \geq e(f) \epsilon_{u, \fp}(x_{\fp}) + \ord_{\fp}(\fd)  }.
\end{align*}
Then, we have
\begin{align*}
    V_{\fp}^{\fa, \fd} = \bigsqcup_{t \geq 0} V_{\fp, t}^{\fa, \fd}.
\end{align*}

\begin{lemma} \label{lem: congruence equiv cond}
Let $R$ be a ring and $I$ an ideal of $R$. For $V \subset R$, the following are equivalent:\\
(i) The set $V$ is determined by the congruence conditions modulo $I$. \\
(ii) There is a subset $\fS$ of $R/I$ satisfying $V = \rho^{-1}(\fS)$, where $\rho$ is the projection $R \to R/I$. \\
(iii) There is a subset $X$ of $R$ satisfying $V = \bigsqcup_{v \in X} (v + I)$ and $\# X = \# \fS$. \\
(iv) The set $V$ is $I$-stable under addition.
\end{lemma}
\begin{proof}
This is straightforward.
\end{proof}

Let $*$ denote the componentwise product on $(n+1)$-tuples of ideals.
In other words,
\begin{align*}
    \lbrb{\fa_0 \times \cdots \times \fa_n } * \lbrb{\fb_0 \times \cdots \times \fb_n }
    = (\fa_0\fb_0 \times \cdots \times \fa_n\fb_n)
\end{align*}
so that $(\fa \fb)^{w} = \fa^w*\fb^w$ and $\fa*_u \fb^w = \fa^u*\fb^w$.

\begin{lemma} \label{lem: concrete fmfa}
Let $f : \cP(u) \to \cP(w)$ be a morphism.\\
(i) For integral ideals $\fa$ and $\fd$, 
\begin{align*}
    V^{\fa, \fd} = \lcrc{x \in \cO_K^{n+1}\setminus \lcrc{0} : \fI_u(x) \subset \fa, \delta_f(x) \subset \fd} = \lcrc{x \in \fa^{u} \setminus \lcrc{0} : \delta_f(x) \subset \fd}.
\end{align*}
\noindent
(ii) 
Let $\widehat{w} = \widehat{w}(f, t)$ be an $(n+1)$-tuple of positive integers whose definition is given in (\ref{eqn: w hat def}) and let
\begin{align*}
    \fm_{\fp, t}^{\fa, \fd} = \fa_{\fp}^u*\lbrb{\fp^{\widehat{w}_0}\cO_{K, \fp} \times \cdots \times \fp^{\widehat{w}_n}\cO_{K, \fp}}.
\end{align*}
Then, there is a finite set $X_{\fp, t}^{\fa, \fd}$ such that
\begin{align*}
    V_{\fp, t}^{\fa, \fd} = \bigsqcup_{v \in X_{\fp, t}^{\fa, \fd}} (v + \fm_{\fp, t}^{\fa, \fd}).
\end{align*}
\noindent 
(iii) 
We have $\# X_{\fp, t}^{\fa, \fd} = \# X_{\fp, t+1}^{\fa, \fd}$ for $t \geq 0$.
\end{lemma}
\begin{proof}
(i)
For each $x\in K^{n+1}\setminus\{0\}$, we have
\begin{align*}
    \mathfrak{I}_u(x)=\prod_{v\in M_{K,0}}\mathfrak{p}_v^{\epsilon_{u,\mathfrak{p}}(x)},
\end{align*}
so $\mathfrak{I}_u(x)\subseteq\mathfrak{a}$ if and only if $\epsilon_{u,\mathfrak{p}}(x)\geq\mathrm{ord}_\mathfrak{p}(\mathfrak{a})$ for every $\mathfrak{p}$.
This is also equivalent to $x \in \fa^{u}$, because $\fI_{u}(x) \subset \fa$ implies $\ord_{\fp}(x_{j}) \geq u_j \ord_{\fp}(\fa)$ for all $j = 0, \cdots, n$.
Similarly, $\delta_f(x) \supset \fd$ if and only if $\ord_{\fp}(\delta_f(x)) \geq \ord_{\fp}(\fd)$ for every $\fp$, which is equivalent to
\begin{align*}
    \epsilon_{w, \fp}(f(x_{\fp})) \geq e(f)\epsilon_{u, \fp}(x_{\fp}) + \ord_{\fp}(\fd)
\end{align*}
for every $\fp$, when $x_{\fp} = i_{\fp}(x)$. This gives (i).

(ii) 
The defect condition on $V_{\fp, t}^{\fa, \fd}$ is
\begin{align*}
    \min_{0 \leq j \leq n} \left\lfloor \frac{\ord_{\fp}(f(x_{\fp})_j)}{w_j} \right\rfloor \geq  e(f)\min_{0 \leq j \leq n} \left\lfloor \frac{\ord_{\fp}(x_{\fp,j})}{u_j} \right\rfloor+ \ord_{\fp}(\fd),
\end{align*}
which can be rewritten as 
\begin{align*}
    \ord_{\fp}(f(x_{\fp})_j) & \geq  \lbrb{e(f)\epsilon_{u, \fp}(x_{\fp}) + \ord_{\fp}(\fd)}w_j \qquad \textrm{ for all } 0 \leq j \leq n.
\end{align*}
Equivalently, this is expressed by the congruence \begin{align} \label{eqn: f residue cond}
\begin{aligned}
    f(x_{\fp})_j &\equiv 0 \pmod{ \lbrb{ \fI_{u, \fp}(x_{\fp})^{e(f)} \fd_{\fp} }^{w_j} } 
    \qquad \textrm{for all } 0\leq j \leq n.
\end{aligned}
\end{align}
Since each $f_j$ is a weighted homogeneous polynomial, $f(x_{\fp})_j$ modulo $\lbrb{\fI_{u,\fp}(x_{\fp})^{e(f)}\fd_{\fp}}^{w_j}$ is determined by $x_{\fp, j}$  modulo $\lbrb{\fI_{u,\fp}(x_{\fp})^{e(f)}\fd_{\fp}}^{w_j}$.
By Lemma \ref{lem: congruence equiv cond}, there is a finite set $\fS_{\fp}^{f,j} \subset \cO_{K, \fp}/\lbrb{\fI_{u,\fp}(x_{\fp})^{e(f)}\fd_{\fp}}^{w_j}$ such that (\ref{eqn: f residue cond}) holds if and only if $x_{\fp, j}$ modulo $\lbrb{\fI_{u,\fp}(x_{\fp})^{e(f)}\fd_{\fp}}^{w_j}$ is in $\fS_{\fp}^{f,j}$.

We first assume that $\fp \nmid \fa$.
By (i), if $x \in V_{\fp, t}^{\fa, \fd}$, then  $\fI_{u, \fp}(x_{\fp}) = \fp^t\cO_{K, \fp}$.
Since $\fI_{u, \fp}(x_{\fp}) = \fp^t \cO_{K, \fp}$ is equivalent to
\begin{align*}
    x_{\fp} \in (\fp^t \cO_{K, \fp})^u \setminus (\fp^{t+1}\cO_{K, \fp})^u,
\end{align*}
the condition $\fI_{u, \fp}(x_{\fp}) = \fp^t \cO_{K, \fp}$ is determined by $x_{\fp, j}$ modulo $\fp^{(t+1)u_j}\cO_{K, \fp}$. 
Therefore, $x_{\fp} \in \cO_{K, \fp}^{n+1}$ is also in $V_{\fp, t}^{\fa, \fd}$ if and only if
\begin{itemize}
    \item $x_{\fp, j} \equiv 0 \pmod{\fp^{tu_j}}$ and $x_{\fp, j} \not\equiv 0 \pmod{\fp^{(t+1)u_j}}$ for each $j$, and
    \item $x_{\fp, j}$ modulo $\lbrb{\fI_{u,\fp}(x_{\fp})^{e(f)}\fd_{\fp}}^{w_j}$ is in $\fS_{\fp}^{f,j}$.
\end{itemize}
Therefore for
\begin{align}
    \widehat{w}_j &\vcentcolon= \max \lcrc{ (\ord_{\fp}(\fd) + te(f))w_j, (t+1)u_j}, \label{eqn: w hat def}
    \\ \nonumber
    \fm_{\fp, t}^{\fa, \fd} & \vcentcolon= (\fp\cO_{K, \fp})^{\widehat{w}} = \fp^{\widehat{w}_0}\cO_{K, \fp} \times \cdots \times \fp^{\widehat{w}_n}\cO_{K, \fp},
\end{align}
there is a finite set $\fS_{\fp, t}^{\fa, \fd} \subset \cO_{K, \fp}^{n+1}/\fm_{\fp, t}^{\fa, \fd}$ such that $x_{\fp} \in V_{\fp, t}^{\fa, \fd}$ if and only if $x_{\fp}$ modulo $\fm_{\fp, t}^{\fa, \fd}$ is in $\fS_{\fp, t}^{\fa, \fd}$.
By Lemma \ref{lem: congruence equiv cond}, we obtain the result when $\fp \nmid \fa$.
We note that the set $X_{\fp, t}^{\fa, \fd}$ defined by Lemma \ref{lem: congruence equiv cond} is stable under $u$-weighted $\cO_{K, \fp}^\times$-action, since $V_{\fp, t}^{\fa, \fd}$ is also stable by (\ref{eqn: weighted action fI f delta}) and (i).

For a prime $\fp$, we choose a uniformizer $\pi_{\fp}$. Then,
\begin{align*}
    \fa_{\fp}^u \to (\fp \fa)_{\fp}^u \qquad x \mapsto \pi_{\fp}*_u x
\end{align*}
is a bijection whose inverses is $x \mapsto \pi_{\fp}^{-1}*_u x$.
By (\ref{eqn: weighted action fI f delta}),
\begin{align*}
    \epsilon_{w, \fp}(f(\pi_{\fp}*_ux_{\fp})) - e(f)\epsilon_{u, \fp}(\pi_\fp *_u x_{\fp}) 
    & = \epsilon_{w, \fp}(f(x_{\fp})) - e(f) \epsilon_{u, \fp}(x_{\fp}).
\end{align*}
So if
\begin{align*}
    \ord_{\fp}(\fI_{w, \fp}(f(x_{\fp}))) - e(f) \ord_{\fp}(\fI_{u, \fp}(x_{\fp})) \geq \ord_{\fp}(\fd)
\end{align*}
then
\begin{align*}
    \ord_{\fp}(\fI_{w, \fp}(f(\pi_{\fp}*_ux_{\fp}))) - e(f) \ord_{\fp}(\fI_{u, \fp}(\pi_{\fp}*_u x_{\fp})) \geq \ord_{\fp}(\fd).
\end{align*}
Therefore, we have $\pi_{\fp} *_u V_{\fp, t}^{\fa, \fd} \subset V_{\fp, t}^{\fp \fa, \fd}$.
Considering the inverse, we have $\pi_{\fp}*_uV_{\fp, t}^{\fa, \fd} = V_{\fp, t}^{\fp \fa, \fd}$.

By the $(\fp \nmid \fa)$-case, there is a set $\fS_{\fp, t}^{\fa, \fd}$ such that $\rho(V_{\fp, t}^{\fa, \fd}) = \fS_{\fp, t}^{\fa, \fd}$ where $\rho : \cO_{K, \fp}^{n+1} \to \cO_{K, \fp}^{n+1}/(\fp\cO_{K, \fp})^{\widehat{w}}$.
So in this case, we have a commutative diagram
\begin{align*}
    \xymatrix{
    0 \ar[r] & V_{\fp, t}^{\fa, \fd} \ar[r] \ar[d]^-{\pi_\fp*_u-} & \cO_{K,\fp}^{n+1} \ar[r] \ar[d]^-{\pi_\fp*_u-}& \fS_{\fp, t}^{\fa, \fd} \ar[r]\ar[d]^-{\pi_\fp*_u-} & 0 \\
    0 \ar[r] & V_{\fp, t}^{\fp\fa, \fd} \ar[r] & (\fp\cO_{K, \fp})^u \ar[r] & \pi_{\fp}*_u\fS_{ \fp , t}^{\fa, \fd} \ar[r] & 0
    }
\end{align*}
Therefore, $V_{\fp, t}^{\fp\fa, \fd}$ is determined by a congruence condition modulo $(\fp\cO_{K, \fp})^u*(\fp\cO_{K, \fp})^{\widehat{w}}$.
Repeating this procedure yields (ii).

(iii) We note that $V_{\fp, t}^{\fp\fa, \fd} = V_{\fp, t+1}^{\fa, \fd}$.
In the above construction in the proof of (ii), we can choose $X_{\fp, t + 1}^{\fa, \fd}$ as $\pi_{\fp}*_uX_{\fp, t}^{\fa, \fd}$ so that $\# X_{\fp, t}^{\fa, \fd} = \# X_{\fp, t+1}^{\fa, \fd}$.
\end{proof}

Since $V_{\fp, t}^{\fa, \fd}$ is stable under $u$-weighted $\cO_{K, \fp}^\times$-action, $\fp *_u V_{\fp, t}^{\fa, \fd}$ make sense and $\fp *_u V_{\fp, t}^{\fa, \fd} = V_{\fp, t}^{\fp\fa, \fd} = V_{\fp, t+1}^{\fa, \fd}$.
Therefore, we also have $\fp*_u V_{\fp}^{\fa, \fd} = V_{\fp}^{\fp\fa, \fd}$, which will be used later.

In a later discussion, we need $\widehat{w}_j$ to be determined by $w_j$, rather than $u_j$.
Therefore, we require the following technical condition on $f$:
\begin{align} \label{eqn: widehat w condition}
{\textrm{for any } \fp \mid \fd,}  \qquad 
    (\ord_{\fp}(\fd) + te(f))w_j \geq (t+1)u_j.
\end{align}
We emphasize that this is not a restrictive assumption. For example, if $w_j \geq u_j$ after reordering then (\ref{eqn: widehat w condition}) holds since $e(f)w_j \geq u_j$ by \cite[Proposition 2.4.8]{Phi1}.

\medskip

The notations $V_{\fp}^{\fa, \fd}$ and $\fm_{\fp}^{\fa, \fd}$ are from \cite{BM}, and the idea of concrete computation of $\fm_{\fp}^{\fa, \fd}$ is given in \cite{Phi1}, whose notation is $\Lambda_{\fp}(\fa)$ in \cite{Phi1}.
However, we should warn that $V_{\fp}^{\fa, \fd}$ and $\fm_{\fp}^{\fa, \fd}$ in this paper are \emph{different} from those of \cite{BM, Phi1}.
The following remarks explain the reasons.
We recall that the condition $\ord_p(n) = k$ on integers $n$ is determined by the congruence modulo $p^{k+1}$, not $p^k$.

\begin{remark} \label{rmk: BM Vpad def}
In \cite{BM}, the authors introduced
\begin{align*}
    &R_{\fp}^{\fa} \vcentcolon= \lcrc{x_{\fp} \in K_{\fp}^{n+1} \setminus \lcrc{0} : \epsilon_{u, \fp}(x_{\fp}) = \ord_{\fp}(\fa) }, \\
    &R_{\fp}^{\fa, \fd} \vcentcolon= \lcrc{x_{\fp} \in K_{\fp}^{n+1} \setminus \lcrc{0} : \epsilon_{u, \fp}(x_{\fp}) =\ord_{\fp}(\fa), \ord_{\fp}(\fI_{w, \fp}(f(x_{\fp}))) -  \ord_{\fp}(\fI_{u, \fp}(x_{\fp})^{e(f)}) = \ord_{\fp}(\fd) }.
\end{align*}
\cite[Lemma 2.8]{BM} claims that $R_{\fp}^{\fa, \fd}$ is determined by a congruence condition modulo $\fm_{\fp}^{\fa, \fd}$ and we can choose $\fm_{\fp}^{\fa, \fd}$ by $\fa_{\fp}^{u}$ if $\ord_{\fp}(\fd) = 0$.
Here is a counterexample to the last statement concerning the case $\ord_{\fp}(\fd) = 0$:
Suppose that $f = \mathrm{id}$ so that $\fD_f = \lcrc{\cO_K}$ and $\ord_{\fp}(\fd) = 0$.
Then, we have $R_{\fp}^{\fa, \fd}=R_{\fp}^{\fa} = \lcrc{x \in K_{\fp}^{n+1} \setminus \lcrc{0} : \epsilon_{u, \fp}(x) = \ord_{\fp}(\fa)}$. 
Since
\begin{align*}
    \min_{0 \leq j \leq n} \left\lfloor \frac{\ord_{\fp}(x_j)}{u_j} \right\rfloor = \ord_{\fp}(\fa),
\end{align*}
$\epsilon_{u, \fp}(x) = \ord_{\fp}(\fa)$ is equivalent to 
``$\ord_{\fp}(x_j) \geq \ord_{\fp}(\fa) u_j$ for all $0 \leq j \leq n$ but at least one of $j$ satisfies $\ord_{\fp}(x_j) < (\ord_{\fp}(\fa)+ 1)u_j$.''
This is equivalent to ``$x_j \in \fp^{\ord_{\fp}(\fa)u_j}$ for all $0 \leq j \leq n$ but at least one of $j$ satisfies $x_j \not\in \fp^{(\ord_{\fp}(\fa)+1)u_j}$.''
This is determined by the congruence condition modulo $\fp^{(\ord_{\fp}(\fa)+1)u_j}$, not $\fp^{\ord_{\fp}(\fa)u_j}$!
Hence, $R_{\fp}^{\fa, \fd}$ is determined by $(\fp^{\ord_{\fp}(\fa)+1})^{u} = (\fp\fa_{\fp})^{u}$, not $\fa_{\fp}^{u}$.
\end{remark}

We suggest considering the condition $\epsilon_{u, \fp}(x_{\fp}) \geq \ord_{\fp}(\fa)$, instead of the equality in both $R_{\fp}^{\fa}$ and $R_{\fp}^{\fa, \fd}$.
If one follows our suggestion, then $R^{\fa}_{\fp}$ is simply $\fa_{\fp}^{u}$, and the notation becomes unnecessary.


\begin{remark} \label{rmk: Phi1 mfad}
For the construction of $\fm_{\fp}^{\fa, \fd}$, we basically followed the proof of \cite[Lemma 4.1.2]{Phi1} with some modifications.
First, the condition
\begin{align*}
    \min_{0 \leq j \leq n} \left\lfloor \frac{\ord_{\fp}(f(x_{\fp})_j)}{w_j} \right\rfloor = e(f)\min_{0 \leq j \leq n} \left\lfloor \frac{\ord_{\fp}(x_{\fp,j})}{u_j} \right\rfloor+ \ord_{\fp}(\fd),
\end{align*}
is a congruence condition modulo $\lbrb{\fa_{\fp}^{e(f)}\fp^{\ord_{\fp}(\fd)+1}\cO_{K, \fp}}^{w_j}$, not $\lbrb{\fa_{\fp}^{e(f)}\fp^{\ord_{\fp}(\fd)}\cO_{K, \fp}}^{w_j}$.
Second, we replace the $\fD_f$-condition on $\Lambda_{\fp}(\fa)$ in \cite[Lemma 4.1.2]{Phi1} by $\fp \mid \fd$.
A more delicate point is that the residue class condition on $f(x_i)$ may be modulo $(\fI_{w'}(x)^{e(f)}\fd_{\fp})^{w_i}$, not $(\fa_{\fp}^{e(f)}\fd_{\fp})^{w_i}$ following the notations of \cite{Phi1}.
\end{remark}

We also use the condition 
\begin{align*}
    \epsilon_{w, \fp}(f(x_{\fp})) \geq e(f)\epsilon_{u, \fp}(x_{\fp}) + \ord_{\fp}(\fd),
\end{align*}
on defects instead of equality.
Because of this replacement, we need one more M\"obius inversion arguments later (see (\ref{eqn: M'=sum M}), (\ref{eqn: M''= sum M}) and (\ref{eqn: mu M' form}) and compare with \cite[(2)]{BM}, \cite[p.26]{Phi1} ).

\begin{remark} \label{rmk: both compact}
Following our notation, the main difficulty is that $V_{\fp}^{\fa, \fd}$ is neither compact nor a union of translates of a lattice.
Since $V_{\fp, t}^{\fa, \fd}$ is a union of translates of a lattice, we approximate $V_{\fp}^{\fa, \fd}$ by such a union, which is one of the main differences between our paper and \cite{BM, Phi1}.
\end{remark}

One can show that the closure of $V_\mathfrak{p}^{\mathfrak{a},\mathfrak{d}}$ in $K_\mathfrak{p}^{n+1}$ is $V_\mathfrak{p}^{\mathfrak{a},\mathfrak{d}}\cup\{0\}$ and this is compact.
Here is an example.

\begin{example}
Let $\phi = \id_{\cP(1, 1)}$ so that the defect is trivial. Then,
\begin{align*}
    V_{\fp, t}^{\fa, \cO_K} \cup \lcrc{0} = \lcrc{ x_{\fp} \in \cO_{K, \fp}^2 : \fI_{(1,1), \fp}(x_{\fp}) = \fa_{\fp} \cdot \fp^t \cO_{K, \fp} }.
\end{align*}
Hence,
\begin{align*}
    V_{\fp, t}^{\fa, \cO_K} = \lbrb{\fp^{\ord_{\fp}(\fa) + t}\cO_{K, \fp} \times \fp^{\ord_{\fp}(\fa) + t}\cO_{K, \fp} }
     \setminus \lbrb{\fp^{\ord_{\fp}(\fa) + t+1}\cO_{K, \fp} \times \fp^{\ord_{\fp}(\fa) + t+1}\cO_{K, \fp}}
\end{align*}
and
\begin{align*}
\bigsqcup_{t \geq 0} V_{\fp, t}^{\fa, \cO_K} = \lbrb{\fa_{\fp}\times \fa_{\fp}} \setminus \lcrc{0},
\qquad 
    V_{\fp}^{\fa, \cO_K} \cup \lcrc{0} = \fa_{\fp} \times \fa_{\fp}.
\end{align*}
So $V_{\fp}^{\fa, \cO_K}\cup \lcrc{0} $ is compact.
\end{example}


Let $B$ be a real number, and let $t_{\fp} = t_{\fp}(B)$ be the maximal integer satisfying $q_{\fp}^{t_{\fp}} \leq B$.
When $\fp \mid \fd$, we define
\begin{align*}
    V_{\fp}^{\fa, \fd}(B) \vcentcolon = \bigsqcup_{0 \leq t \leq t_{\fp}} V_{\fp, t}^{\fa, \fd}.
\end{align*}

\begin{lemma} \label{lem: V pad lattice}
(i) If $\fp \mid \fd$, $V_{\fp}^{\fa, \fd}(B)$ is a union of translates of $\fm_{\fp, t_{\fp}}^{\fa, \fd}$ and it satisfies
\begin{align*}
V_{\fp}^{\fa, \fd} \setminus V_{\fp}^{\fa, \fd}(B) \subset (\fp^{t_{\fp}+1}\cO_{K, \fp})^u \quad \textrm{and} \quad
    m_{\fp}(V_{\fp}^{\fa, \fd} \setminus V_{\fp}^{\fa, \fd}(B)) 
    \ll B^{-|u|}.
\end{align*}
(ii) If $\fp \nmid \fd$, we have $V_{\fp}^{\fa, \fd} \cup \lcrc{0} = \fa_{\fp}^u$.
\end{lemma}
\begin{proof}
(i) By the definition and Lemma \ref{lem: concrete fmfa} (ii),
\begin{align*}
    V_{\fp}^{\fa, \fd}(B) = \bigsqcup_{0 \leq t \leq t_{\fp}} V_{\fp, t}^{\fa, \fd} = \bigsqcup_{0 \leq t \leq t_{\fp}} \bigsqcup_{v \in X_{\fp, t}^{\fa, \fd}}(v + \fm_{\fp, t}^{\fa, \fd}).
\end{align*}
Since $v + \fm_{\fp,t}^{\fa, \fd}$ is a union of translates of $\fm_{\fp, t'}^{\fa, \fd}$ for $t' \geq t$, the first part follows.

Since $X_{\fp, t+1}^{\fa, \fd} = \pi_{\fp}*_u X_{\fp, t}^{\fa, \fd}$, $v \in X_{\fp, t}^{\fa, \fd}$ is also in $(\fp^t\cO_{K, \fp})^u$.
By Lemma \ref{lem: concrete fmfa} (ii), $\fm_{\fp, t}^{\fa, \fd} \subset (\fp^{t+1}\cO_{K, \fp})^u$.
Hence, 
\begin{align*}
    \bigsqcup_{v \in X_{\fp, t}^{\fa, \fd}}(v + \fm_{\fp, t}^{\fa, \fd})
    \subset (\fp^t\cO_{K, \fp})^u.
\end{align*}
Therefore,
\begin{align*}
    m_{\fp}(V_{\fp}^{\fa, \fd} \setminus V_{\fp}^{\fa, \fd}(B))
    &= m_{\fp}\lbrb{ \bigsqcup_{t > t_{\fp}} \bigsqcup_{v \in X_{\fp, t}^{\fa, \fd}}(v + \fm_{\fp, t}^{\fa, \fd}) }
    = \sum_{t > t_{\fp}} m_{\fp}\lbrb{ (\fp^t\cO_{K, \fp})^u } 
    = \sum_{t > t_{\fp}} q_{\fp}^{-t|u|} \leq q_{\fp}^{-t_{\fp}|u|}.
\end{align*}

(ii) If $\fp \nmid \fd$, there is no defect condition.
We have
\begin{align*}
    V_{\fp}^{\fa, \fd} \cup \lcrc{0} = \lbrb{\bigsqcup_{t \geq 0} V_{\fp, t}^{\fa, \fd}} \cup \lcrc{0} 
    = \lcrc{x \in \cO_{K, \fp}^{n+1} : \epsilon_{u, \fp}(x_{\fp}) \geq \ord_{\fp}(\fa)  }
    = \fa_{\fp}^u
\end{align*}
by the same argument of the previous example.
\end{proof}
 
Motivated by Lemma \ref{lem: V pad lattice}, let
\begin{align} \label{eqn: def fm_fp^fa^fd V_fp^fa^fd}
    \fm_{\fp}^{\fa, \fd}(B) \vcentcolon = 
    \left\{
    \begin{array}{ll}
    \fm_{\fp, t_\fp}^{\fa, \fd}     &  \textrm{ if } \fp \mid \fd, \\
    \fa_{\fp}^u     & \textrm{ if } \fp \nmid \fd,
    \end{array}
    \right. \qquad 
    V_{\fp}^{\fa, \fd}(B) \vcentcolon = 
    \left\{
    \begin{array}{ll}
    \bigsqcup_{0 \leq t \leq t_\fp} V_{\fp, t}^{\fa, \fd}     &  \textrm{ if } \fp \mid \fd, \\
    V_{\fp}^{\fa, \fd} \cup \lcrc{0} = \fa_{\fp}^u     & \textrm{ if } \fp \nmid \fd.
    \end{array}
    \right.
\end{align}
If $\fp \mid \fd$, let $X_{\fp}^{\fa, \fd}(B)$ be the finite set satisfying
\begin{align*}
    V_{\fp}^{\fa, \fd}(B) = \bigsqcup_{v \in X_{\fp}^{\fa, \fd}(B)} (v + \fm_{\fp}^{\fa, \fd}(B))
\end{align*}
and if $\fp \nmid \fd$, let $X_{\fp}^{\fa, \fd}(B) \vcentcolon = \lcrc{0}$.
Then this equality also holds for $\fp \nmid \fd$.
By Lemma \ref{lem: congruence equiv cond}, there is a set $\fS_{\fp}^{\fa, \fd}(B)$ such that $V_{\fp}^{\fa, \fd}(B) = \rho^{-1}(\fS_{\fp}^{\fa, \fd}(B))$.

For the global objects, we define
\begin{align*}
    \fm^{\fa, \fd}(B) \vcentcolon &= \lbrb{ \bigcap_{q_{\fp} \leq B} i_{\fp}^{-1}(\fm_{\fp}^{\fa, \fd}(B)) } \bigcap \lbrb{ \bigcap_{q_{\fp} \geq B} i_{\fp}^{-1}(\cO_{K,\fp}^{n+1})  }, \\
    V^{\fa, \fd}(B) \vcentcolon &= \lbrb{ \bigcap_{q_{\fp} \leq B } i_{\fp}^{-1}(V_{\fp}^{\fa, \fd}(B)) } \bigcap \lbrb{ \bigcap_{q_{\fp} \geq B} i_{\fp}^{-1}(\cO_{K,\fp}^{n+1})  }.
\end{align*} 
Then, $V^{\fa, \fd}(B)$ is determined by the congruence conditions modulo $\fm^{\fa, \fd}(B)$ by the Chinese remainder theorem.
By Lemma \ref{lem: congruence equiv cond}, there are $X^{\fa, \fd}(B)$ and $\fS^{\fa, \fd}(B)$ such that 
\begin{align} \label{eqn: Vad periodic}
    V^{\fa, \fd}(B)  = \bigsqcup_{v \in X^{\fa, \fd}(B)} (v + \fm^{\fa, \fd}(B)) = \rho^{-1}(\fS^{\fa, \fd}(B)), \qquad
    \# \fS^{\fa, \fd}(B) = \# X^{\fa, \fd}(B).
\end{align}

Following \cite[\S 3]{Den}, let  $r_1$ and $r_2$ be the number of real and complex places of $K$ so that $|M_{K, \infty}| = r_1 + r_2$.
For a vector $d = (d_{v_i})_{i=1}^{r_1 + r_2}$ where $d_{v_i} = 1$ for real $v_i$ and $2$ for complex $v_i$, we denote by $H$ the hyperplane of $\bR^{r_1 + r_2}$ which is orthogonal to $d$.
There is the natural projection map $\mathrm{pr} : \bR^{r_1 + r_2} \to H.$

Let $\mu(K)$ be the set of roots of unity in $K$. Then by Dirichlet's unit theorem, $\cO_K^\times/\mu(K)$ is isomorphic to a lattice $\Lambda$ in the hyperplane $H$.
By choosing a basis of $\Lambda$, we have a fundamental domain $\widetilde{\cF}$ of $H/\Lambda$.
For an infinite place $v$ of $K$, we define
\begin{align*}
    \eta_v : K_v^{n+1} \setminus \lcrc{0} \to \bR, \qquad x_v = (x_{v, 0}, \cdots, x_{v, n}) \mapsto
    \log \lbrb{\max_{0 \leq j \leq n} |x_{v, j}|_v^{\frac{1}{u_j}}}
\end{align*}
which naturally induces
\begin{align*}
\eta : \prod_{v\in M_{K, \infty}}\lbrb{K_v^{n+1} \setminus \lcrc{0}} \to \bR^{r_1 + r_2}, \qquad   \eta = \prod_{v\in M_{K, \infty}} \eta_v.
\end{align*}
Let 
\begin{align*}
    \cF \vcentcolon= (\mathrm{pr} \circ \eta)^{-1}(\widetilde{\cF}).
\end{align*}
For $f: \cP(u) \to \cP(w)$, we define
\begin{align*}
    \cD(B) &\vcentcolon= 
    \lcrc{(x_v)_{v\in M_{K, \infty}} \in \prod_{v\in M_{K, \infty}} (K_v^{n+1} \setminus \lcrc{0}) : H_{w, \infty}(f(x))\leq B}
\end{align*}
and
\begin{align*}
    \cF(B) \vcentcolon= \cF \cap \cD(B).
\end{align*}

Let $i_{\infty} : K^{n+1} \to K_{\infty}^{n+1}$ be the diagonal map. Then, $i_{\infty}(x)$ is in $\cD(B)$ if and only if $H_{w, \infty}(f(x)) \leq B$.
We note that $u$-weighted $\cO_K^\times$-action on $K^{n+1} \setminus \lcrc{0}$ also induces $\cO_K^\times$-action on $\prod_{v \in M_{K, \infty}}(K_v^{n+1} \setminus \lcrc{0})$.

\begin{lemma} \label{lem:Phi41456}
(i) $\cD(B)$ (resp. $\cF$) is stable under the $u$-weighted $\cO_K^\times$-action (resp. the $u$-weighted $\mu(K)$-action).
Furthermore, a natural embedding $i : \cF \to \prod_{v \in M_{K, \infty}}(K_v^{n+1} \setminus \lcrc{0})$ induces a bijection
\begin{align*}
    \cF/\mu(K) \to \left.\lbrb{\prod_{v \in M_{K, \infty}} (K_v^{n+1} \setminus \lcrc{0})}\middle/\cO_K^\times\right. .
\end{align*}
(ii) $\cF(B) = B^{\frac{1}{e(f)d}}*_{\widetilde{u}}\cF(1)$ for all $B > 0$. \\
(iii) $\cF(1)$ is bounded and definable.
\end{lemma}
\begin{proof}
The definition of $\widetilde{u}$ is given before Proposition \ref{prop:phi3 327}.
These statements are \cite[Lemmas 4.1.4, 4.1.5, 4.1.6]{Phi1}. See also \cite[Lemma 3.2, Lemma 3.5, Lemma 3.7]{BM}.
\end{proof}

We temporarily reintroduce the distinction between 
$x \in K^{n+1}, \bar{x} \in K^{n+1}/\cO_K^\times, [x ] \in \cP(u)(K)$, as this level of precision is important in the current discussion.
For a set of projective local conditions $\lcrc{\Omega_{\fp}}_{\fp \in S}$ and integral ideals $\fb \subset \fa$, we define
\begin{align*}
    M(\fb, \fa, \fd, \Omega, B) 
    &\vcentcolon= \# \lcrc{ \overline{x} \in V^{\fa, \fd} /\cO_K^\times : \
    \begin{array}{ll}
    \fI_{u}(x) = \fb,  \\
    \delta_f(x) = \fd,
    \end{array} \ 
    \begin{array}{ll}
    H_{w, \infty}(f(x)) \leq B,   \\
    i_{\fp}([x]) \in \Omega_{\fp} \textrm{ for all } \fp \in S
    \end{array}
     }, \\
    M'(\fb, \fa, \fd, \Omega, B) 
    &\vcentcolon= \# \lcrc{ \overline{x} \in V^{\fa, \fd} /\cO_K^\times : \
    \begin{array}{ll}
    \fI_{u}(x) = \fb,  \\
    \delta_f(x) \subset \fd,
    \end{array} \ 
    \begin{array}{ll}
    H_{w, \infty}(f(x)) \leq B,   \\
    i_{\fp}([x]) \in \Omega_{\fp} \textrm{ for all } \fp \in S
    \end{array}
     }, \\
    M''(\fb, \fa, \fd, \Omega, B) 
    &\vcentcolon= \# \lcrc{ \overline{x} \in V^{\fa, \fd} /\cO_K^\times : \
    \begin{array}{ll}
    \fI_{u}(x) \subset \fb,  \\
    \delta_f(x)\subset \fd,
    \end{array} \ 
    \begin{array}{ll}
    H_{w, \infty}(f(x)) \leq B,   \\
    i_{\fp}([x]) \in \Omega_{\fp} \textrm{ for all } \fp \in S
    \end{array}
     }.
\end{align*}
Since $\delta_f$-condition on $M''$ coincides with those on $V^{\fa, \fd}$, it is redundant, but is retained here to facilitate comparison.

By (\ref{eqn: weighted action fI f delta}), $\fI_u(x)$ and $\delta_f(x)$ are stable under $u$-weighted $\cO_{K}^\times$-action.
Since the absolute value of roots of unity is $1$, $H_{w, \infty}(f(x))$ is also stable under $\cO_K^\times$-action.
Finally, $i_{\fp}([x]) \in \Omega_{\fp}$ if and only if one of $\widetilde{x} \in K^{n+1}\setminus \lcrc{0}$ with $[x] = [\widetilde{x}]$ satisfies $i_{\fp}([\widetilde{x}]) \in \Omega_{\fp}$.
Therefore, the above sets are well-defined.
Furthermore, we have
\begin{align} \label{eqn: M'=sum M}
    M'(\fb, \fa, \fd, \Omega, B) = \sum_{\fd' \subset \cO_K} M(\fb, \fa, \fd'\fd, \Omega, B), \\
    M''(\fb, \fa, \fd, \Omega, B) = \sum_{\fc \subset \cO_K} M'(\fb\fc, \fa, \fd, \Omega, B).
    \label{eqn: M''= sum M}
\end{align}
The sum in the first equation (\ref{eqn: M'=sum M}) is finite if $f$ has finite defect, and the second one is always infinite.
We note that they are counterparts of 
\begin{align*}
    \mathcal{N}_{\phi}(\fb, \fa, \fd, T) &\vcentcolon = \widehat{\#}\lcrc{x \in (V^{\fa, \fd} \setminus \lcrc{0})/\cO_K^\times : \fI_{u}(x) = \fb, H_{w, \infty}(f(x)) \leq T}, \\
    \bar{\mathcal{N}}_{\phi}(\fb, \fa, \fd, T) &\vcentcolon = \widehat{\#}\lcrc{x \in (V^{\fa, \fd} \setminus \lcrc{0})/\cO_K^\times : \fI_{u}(x) \subset \fb, H_{w, \infty}(f(x)) \leq T},
\end{align*}
where $\widehat{\#} X = \sum_{x \in X} \frac{1}{\#\Aut x}$ in \cite[\S 3.1]{BM}, but we need to introduce one more notation because of Remark \ref{rmk: BM Vpad def}.

\begin{remark} \label{rmk: on Phi M}
We note that the author defined
\begin{align*}
    &\cM(\fa, \fd)  = \# \lcrc{x \in K^{n+1} \setminus \lcrc{0} : \fI_{u}(x) \subset \fa, \delta_f(x) = \fd},\\
    &M(\Omega, \fa, \fd, B) =
    \# \lcrc{ x \in \cM(\fa, \fd)/\cO_K^\times : H_{w}(f(x)) \leq B, \fI_{u}(x) = \fa, \delta_f(x) = \fd, x \in \Omega^{\aff} }, \\
    &M'(\Omega, \fa, \fd, B) =
    \#\lcrc{ x \in \cM(\fa, \fd)/\cO_K^\times : H_{w}(f(x)) \leq B, \fI_{u}(x) \subset \fa, \delta_f(x) = \fd, x \in \Omega^{\aff} }.
\end{align*}    
in \cite{Phi1, Phi2}, and presented a relation similar to (\ref{eqn: M''= sum M}) in \cite[p.27]{Phi1}.
However, since replacing $\fa$ with $\fa\fb$ in $M'$ affects both the structure of $\cM(\fa, \fd)$ and the condition $\fI_{u}(x) \subset \fa$, the validity of the relation is not immediately evident.
\end{remark}

Since $\fI_{u}(x) \subset \fb$ is equivalent to $x \in \fb^{u}$, there is a natural relation between $M''(\fb, \fa, \fd, \Omega, B)$ and $i_{\infty}(V^{\fa, \fd} \cap \fb^{u} \cap \Omega_S^{\aff}) \cap \cD(B) $, as previously mentioned in \cite[Proposition 3.6]{BM} without local conditions.
From now on, we omit $i_{\infty}$ and write $V^{\fa, \fd} \cap \fb^{u} \cap \Omega_S^{\aff} \cap \cD(B)$ for $i_{\infty}(V^{\fa, \fd} \cap \fb^{u} \cap \Omega_S^{\aff}) \cap \cD(B) $.

We recall (\ref{eqn: Omega bijec Omega_k}), which says that there is a bijection
\begin{align*}
    \Omega_{\fp}^{\aff} \cap \cO_{K, \fp}^{n+1}
    \longrightarrow \bigsqcup_{\substack{k\geq 0}} \bigcup_{\zeta \in \mu(K_{\fp})} (\zeta*_u\Omega_{\fp, k}^{\aff}).
\end{align*}
Let $\mu(K_{\fp})^{(u)} \leq \mu(K_{\fp})$  be the set of roots of unity $\zeta$ satisfying $\zeta^{u_j} = 1$ for each $j = 0, \cdots, n$. Then, $\zeta*_u \Omega_{\fp, k}^{\aff} = \zeta'*_{u} \Omega_{\fp, k}^{\aff}$ if and only if $\zeta' \zeta^{-1} \in \mu(K_{\fp})^{(u)}$.
Hence, if we choose a non-canonical set $\mu_u(K_{\fp})$ as a set of representatives for $\mu(K_{\fp})/\mu(K_{\fp})^{(u)}$, we have a bijection
\begin{align} \label{eqn: Omega decomp Omega_k disj}
    \Omega_{\fp}^{\aff} \cap \cO_{K, \fp}^{n+1}
    \longrightarrow \bigsqcup_{\substack{k\geq 0}} \bigsqcup_{\zeta \in \mu_u(K_{\fp})} (\zeta*_u\Omega_{\fp, k}^{\aff}).
\end{align}
Similarly, one can also define $\mu_u(K)$.

\begin{lemma} \label{lem: M' step1}
(i) Let $Z$ be an $\cO_K^\times$-stable subset of $\prod_{v \in M_{K,\infty}}(K_v^{n+1} \setminus \lcrc{0})$. Then, 
\begin{align*}
    (Z \cap \cF)/\mu(K) \to Z/\cO_K^\times, \qquad x \mu(K) \mapsto x \cO_K^\times
\end{align*}
is a bijection.
\\
(ii) 
Let $\Omega = \lcrc{\Omega_{\fp}}_{\fp \in S}$ be a set of irreducible projective local conditions and let $\fb \subset \fa$ integral ideals of $\cO_K$. Then
\begin{align*}
    M''(\fb, \fa, \fd, \Omega, B) = \frac{1}{\#\mu_{u}(K)} \sum_{\zeta \in \mu_u(K_{\fp})}\sum_{k \geq 0}  \#\lbrb{ V^{\fa, \fd} \cap \fb^{u} \cap 
    \left( \bigcap_{\fp}i_{\fp}^{-1}(\zeta*_u\Omega_{\fp, k}^{\aff})\right)  \cap \cF(B)}.
\end{align*}
\end{lemma}
\begin{proof}
(i) By Lemma \ref{lem:Phi41456} for an $\cO_K^\times$-stable set $Z \subset \prod_{v \in M_{K, \infty}}(K_v^{n+1} \setminus \lcrc{0})$, we have a bijection $(Z \cap \cF)/\mu(K) \to Z/\cO_K^\times$.

(ii) 
By the discussion at the beginning of the section \ref{subsec: local cond Omega}, for $\overline{x} \in V^{\fa, \fd}/\cO_K^\times$,   $i_{\fp}([\overline{x}]) \in \Omega_{\fp}$ if and only if $i_{\fp}(x) \in \Omega_{\fp}^{\aff} \cap \cO_{K, \fp}^{n+1}$.
Therefore, by (\ref{eqn: Omega decomp Omega_k disj}), we have
\begin{align*}
    &M''(\fb, \fa, \fd, \Omega, B)  \\
    &=  \sum_{\zeta \in \mu_u(K_{\fp})} \sum_{k \geq 0} \# \lcrc{ \overline{x} \in V^{\fa, \fd}/\cO_{K}^\times : \fI_u(x) \subset \fb,  H_{w, \infty}(f(x)) \leq B, i_{\fp}(x) \in (\zeta*_u\Omega_{\fp, k}^{\aff}) \textrm{ for } \fp \in S }.
\end{align*}

The height condition in the definition of $M''$ is the same as that of $\cD(B)$. 
Therefore, each summand is
\begin{align*}
\left.
    \lbrb{V^{\fa, \fd} \cap \fb^{u} \cap 
    \left( \bigcap_{\fp}i_{\fp}^{-1}(\zeta*_u\Omega_{\fp, k}^{\aff}) \right) \cap \cD(B)} \middle/\cO_{K}^\times, \right.
\end{align*}
whose cardinality is equal to that of
\begin{align*}
    &\left.\lbrb{ V^{\fa, \fd} \cap \fb^{u} \cap 
    \left( \bigcap_{\fp}i_{\fp}^{-1}(\zeta*_u\Omega_{\fp, k}^{\aff}) \right) \cap \cF(B)} \middle/\mu(K)\right. ,
\end{align*}
as follows from (i) by taking $Z = V^{\fa, \fd} \cap \fb^{u} \cap \left( \bigcap_{\fp}i_{\fp}^{-1}(\zeta*_u\Omega_{\fp, k}^{\aff})\right) \cap \cD(B)$.
Consequently, we have
\begin{align*}
    M''(\fb, \fa, \fd, \Omega, B)
    &= \sum_{\zeta \in \mu_u(K_{\fp})} \sum_{k \geq 0} \#\lbrb{ V^{\fa, \fd} \cap \fb^{u} \cap 
    \left.\left( \bigcap_{\fp}i_{\fp}^{-1}(\zeta*_u\Omega_{\fp, k}^{\aff})\right) \cap \cF(B)} \middle/\mu(K) \right. 
\end{align*}
which gives (ii).
\end{proof}

We recall that $\fD_f$ is the set of defect ideals of $f$ and
$S_f$ is the set of prime ideals that divide an ideal in $\fD_f$. 
Let $\widehat{\fD_f}$ be the set of integral ideals in $\cO_K$ whose prime divisors are in $S_f$ and let $\widetilde{\fD_f}$ be the set of ideals in $\cO_K$ not divisible by any prime ideal in $S_f$.
Then, any integral ideal $I$ can be decomposed as the product of an ideal $\fc_0(I)$ in $\widehat{\fD_f}$ and an ideal $\fc_1(I)$ in $\widetilde{\fD_f}$.

Since $\fb \subset \fa$, there is the decomposition $\fa^{-1}\fb = \fc_0(\fa^{-1}\fb) \fc_1(\fa^{-1}\fb)$ which we denote by $\fc_0$ and $\fc_1$, respectively.
Following \cite[p.15]{BM}, we define
\begin{align*}
    \fm_{\fc_1, \fp}^{\fa, \fd} &\vcentcolon= \fm_{\fp}^{\fa, \fd} \cap (\fa \fc_1)_{\fp}^{u},
    &\quad V_{\fc_1, \fp}^{\fa, \fd} &\vcentcolon= V_{\fp}^{\fa, \fd} \cap (\fa \fc_1)_{\fp}^{u}, \\
    \fm_{\fc_1, \fp}^{\fa, \fd}(B) &\vcentcolon= \fm_{\fp}^{\fa, \fd}(B) \cap (\fa \fc_1)_{\fp}^{u},
    &\quad V_{\fc_1, \fp}^{\fa, \fd}(B) &\vcentcolon= V_{\fp}^{\fa, \fd}(B) \cap (\fa \fc_1)_{\fp}^{u}, \\
    \fm_{\fc_1}^{\fa, \fd}(B) &\vcentcolon= \fm^{\fa, \fd}(B) \cap (\fa \fc_1)^{u},
    &\quad V_{\fc_1}^{\fa, \fd}(B) &\vcentcolon= V^{\fa, \fd}(B) \cap (\fa \fc_1)^{u}.
\end{align*}
These definitions depend on $\fa$ and $\fb$, since $\fc_1$ is defined in terms of them.

By definition (\ref{eqn: Vad periodic}) and Lemma \ref{lem: congruence equiv cond}, $V^{\fa, \fd}(B)$ is $\fm^{\fa, \fd}(B)$-stable under addition.
Therefore, $V^{\fa, \fd}(B) \cap (\fa\fc_1)^{u}$ is closed under addition by $\fm_{\fc_1}^{\fa, \fd}(B)= \fm^{\fa, \fd}(B) \cap (\fa\fc_1)^{u}$.
Hence, there exist a set $\fS_{\fc_1}^{\fa, \fd}(B)$ and a subset $X_{\fc_1}^{\fa, \fd}(B) \subset V_{\fc_1}^{\fa, \fd}(B)$ such that
\begin{align} \label{eqn: Vc1ad periodic}
    V_{\fc_1}^{\fa, \fd}(B)
    = \bigsqcup_{v \in X_{\fc_1}^{\fa, \fd}(B)} (v + \fm_{\fc_1}^{\fa, \fd}(B)), \quad 
    \fS_{\fc_1}^{\fa, \fd}(B) = V_{\fc_1}^{\fa, \fd}(B)/\fm_{\fc_1}^{\fa, \fd}(B), \quad
    \# \fS_{\fc_1}^{\fa, \fd}(B) = \# X_{\fc_1}^{\fa, \fd}(B).
\end{align}
\begin{lemma} \label{lem: V^ad b^w translation}
For integral ideals $\fa, \fb$ satisfying $\fb \subset \fa$, there is a finite set $X_{\fb}^{\fa, \fd}(B)$ such that  $ \# \fS_{\fc_1}^{\fa, \fd}(B) = \# X_{\fb}^{\fa, \fd}(B)$ and
\begin{align*}
    V^{\fa, \fd}(B)  \cap \fb^{u} 
    = \bigsqcup_{v \in X_{\fb}^{\fa, \fd}(B)} \lbrb{v + \fm_{\fc_1}^{\fa, \fd}(B) \cap (\fa \fc_0)^{u} }.
\end{align*}
\end{lemma}
\begin{proof}

By the definition of $\fc_0, \fc_1$, we have $V^{\fa, \fd}(B) \cap \fb^{u} = V_{\fc_1}^{\fa, \fd}(B) \cap (\fa\fc_0)^{u}$.
Since $V_{\fc_1}^{\fa, \fd}(B) + \fm_{\fc_1}^{\fa, \fd}(B) \subset V_{\fc_1}^{\fa, \fd}(B)$, we have
\begin{align*}
    \lbrb{V_{\fc_1}^{\fa, \fd}(B) \cap (\fa \fc_0)^{u}} + \lbrb{\fm_{\fc_1}^{\fa, \fd}(B) \cap (\fa \fc_0)^{u}} \subset  V_{\fc_1}^{\fa, \fd}(B) \cap (\fa \fc_0)^{u}.
\end{align*}
Hence by Lemma \ref{lem: congruence equiv cond}, there are 
\begin{align*}
    \fS_{\fb}^{\fa, \fd}(B) = \frac{V^{\fa, \fd}_{\fc_1}(B) \cap (\fa\fc_0)^{u}}{\fm_{\fc_1}^{\fa, \fd}(B) \cap (\fa \fc_0)^{u}} \qquad \textrm{and} \qquad 
    X_{\fb}^{\fa, \fd}(B) \subset V^{\fa, \fd}_{\fc_1}(B) \cap (\fa\fc_0)^{u}
\end{align*}
satisfying
\begin{align*}
    V^{\fa, \fd}_{\fc_1}(B) \cap (\fa\fc_0)^{u} = \bigsqcup_{v \in X_{\fb}^{\fa, \fd}(B)} \lbrb{v + \fm_{\fc_1}^{\fa, \fd}(B) \cap (\fa \fc_0)^{u} }, \qquad \#\fS_{\fb}^{\fa, \fd}(B) = \# X_{\fb}^{\fa, \fd}(B).
\end{align*}
By definition (\ref{eqn: def fm_fp^fa^fd V_fp^fa^fd}), we have $\fm_{\fp}^{\fa, \fd}(B) = \fa_{\fp}^u$ when $\fp \nmid \fd$. Hence if $\fp \not\in S_f$, we have
\begin{align*}
    \fm_{\fc_1, \fp}^{\fa, \fd}(B) = \fm_{\fp}^{\fa, \fd}(B) \cap (\fa \fc_1)_{\fp}^u = \fa_{\fp}^u.
\end{align*}
If $\fp \in S_f$, we have $(\fa \fc_0)_{\fp}^u = \fa_{\fp}^u$.
Therefore, we have $\fm_{\fc_1, \fp}^{\fa, \fd}(B) + (\fa\fc_0)_{\fp}^u = \fa_{\fp}^u$ for any $\fp$.
Then the natural isomorphism
\begin{align*}
    \frac{\fa^u}{ \fm_{\fc_1}^{\fa, \fd}(B)} = \frac{(\fa\fc_0)^u + \fm_{\fc_1}^{\fa, \fd}(B)}{\fm_{\fc_1}^{\fa, \fd}(B)}
    \cong \frac{(\fa\fc_0)^u}{\fm_{\fc_1}^{\fa, \fd}(B) \cap  (\fa\fc_0)^u}
\end{align*}
gives a bijection between $\fS_{\fc_1}^{\fa, \fd}(B)$ and $\fS_{\fb}^{\fa, \fd}(B)$.
\end{proof}

\begin{lemma} \label{lem: temp fS}
Let $\lambda$ be an element of $K$ such that $\lambda\cO_K$ is relatively prime to the ideals in $S_f$. Then, $\#\fS_{\fc_1}^{\lambda\fa, \fd}(B) = \#\fS_{\fc_1}^{\fa, \fd}(B)$.
\end{lemma}
\begin{proof}

In the proof of Lemma \ref{lem: concrete fmfa}, we proved that
$\lambda*_{u}V^{\fa, \fd} = V^{\lambda \fa, \fd} = (\lambda\cO_K)^{u} * V^{\fa, \fd}$.
We will use $\lambda\fa$ and $\lambda\fb$ to define $V^{\lambda\fa, \fd}_{\fc_1}$ so that  $\fc_1 = \fc_1(\fa^{-1}\fb) = \fc_1((\lambda\fa)^{-1}(\lambda\fb))$.
Therefore,
\begin{align*}
    V_{\fc_1, \fp}^{\lambda \fa, \fd}(B) &= V_{\fp, t}^{\lambda \fa, \fd} \cap (\lambda \fa \fc_1( (\lambda \fa)^{-1} \lambda \fb ))^u
    =V_{\fp, t}^{\lambda \fa, \fd} \cap (\lambda \cO_{K, \fp} \fa \fc_1 )^u \\
    &= ( (\lambda \cO_{K, \fp})^u * V_{\fp,t}^{\fa, \fd}) \cap (( \lambda \cO_{K, \fp} )^u * (\fa\fc_1)^u ) = (\lambda \cO_{K, \fp})^u * (V_{\fp,t}^{\fa, \fd} \cap  (\fa\fc_1)^u ) \\
    &= (\lambda \cO_{K, \fp} )^u* V_{\fc_1, \fp}^{\fa, \fd}(B). 
\end{align*}

By (\ref{eqn: def fm_fp^fa^fd V_fp^fa^fd}), if $\fp \nmid \fd$, then
$\fm_{\fp}^{\lambda \fa, \fd}(B) = (\lambda \fa)_{\fp}^u = (\lambda \cO_K)_{\fp}^u * \fa_{\fp}^u = (\lambda \cO_K)_{\fp}^u * \fm_{\fp}^{\fa, \fd}(B)$.
Hence, 
\begin{align*}
    \fm_{\fc_1, \fp}^{\lambda \fa, \fd}(B) &= \fm_{\fp}^{\lambda \fa, \fd} (B)\cap (\lambda\cO_{K}  \fa \fc_1)_{\fp}^u  = [(\lambda\cO_K)_{\fp}^{u} * \fm_{\fp}^{\fa, \fd}(B)] \cap [(\lambda\cO_K)_{\fp}^{u} * (\fa \fc_1)_{\fp}^u ] = (\lambda\cO_K)_{\fp}^{u} * \fm_{\fc_1, \fp}^{\fa, \fd}(B).
\end{align*}
Therefore, for primes $\fp \nmid \fd$, we have
\begin{align*}
    \frac{V_{\fc_1, \fp}^{\lambda\fa, \fd}(B)}{\fm^{\lambda \fa, \fd}_{\fc_1, \fp}(B)}
    = \frac{(\lambda \cO_K )_{\fp}^{u} * V_{\fc_1, \fp}^{\fa, \fd}(B)}{(\lambda\cO_K)_{\fp}^{u} * \fm_{\fc_1, \fp}^{\fa, \fd}(B)  } \cong \frac{V_{\fc_1, \fp}^{\fa, \fd}(B)}{\fm_{\fc_1, \fp}^{\fa, \fd}(B) }.
\end{align*}

When $\fp \mid \fd$, we have 
$V_{\fc_1, \fp}^{\lambda \fa, \fd} = V_{\fc_1, \fp}^{\fa, \fd}$ since $\lambda\cO_K$ is not divisible by $\fp \in S_f$.
The constant $t_{\fp}$ depends on $B$ and $\fp$, and the definition of $\fm_{\fp, t_{\fp}}^{\fa, \fd}$ in Lemma \ref{lem: concrete fmfa} (ii) depends only on $\fa_{\fp}, \fd_{\fp}$ and $f$.
Hence, if $\fp \nmid \lambda \cO_K$, $\fm_{\fp}^{\lambda \fa, \fd}(B) = \fm_{\fp, t_{\fp}}^{\lambda\fa, \fd} = \fm_{\fp, t_{\fp}}^{\fa, \fd} = \fm_{\fp}^{\fa, \fd}(B)$.
Consequently, we have
\begin{align*}
    \# \fS_{\fc_1}^{\lambda \fa, \fd}(B)
    = \# \lbrb{ \frac{V_{\fc_1}^{\lambda\fa, \fd}(B)}{\fm^{\lambda \fa, \fd}_{\fc_1}(B)} }
    = \# \lbrb{ \frac{V_{\fc_1}^{\fa, \fd}(B)}{\fm^{\fa, \fd}_{\fc_1}(B)}} = \#\fS_{\fc_1}^{\fa, \fd}(B),
\end{align*}
which gives the result.
\end{proof}

We note that the idea of the proof of Lemma \ref{lem: temp fS} is in \cite[Remark 3.16]{BM}, but the assumption that $\lambda\cO_K$ is relatively prime to $S_f$ is not stated.

We recall that $D_K$ is the absolute value of the discriminant of a number field $K$.

\begin{lemma} \label{lem: determinant}
For $\Lambda = \fm_{\fc_1}^{\fa, \fd}(B) \cap (\fa \fc_0)^{u}$, we have
\begin{align*}
m_{\infty}(K_{\infty}^{n+1}/\Lambda) = 
    \det(\fm_{\fc_1}^{\fa, \fd}(B) \cap (\fa \fc_0)^{u}) = N_{K/\bQ}(\fc_0)^{|u|} [\cO_K^{n+1} : \fm_{\fc_1}^{\fa, \fd}(B)] \frac{D_K^{\frac{n+1}{2}}}{2^{(n+1)r_2}}.
\end{align*}
\end{lemma}
\begin{proof}
Since $m_{\infty}$ is the induced Haar measure on $K_{\infty}^{n+1}$ by the usual Haar measure on $\bR$ and $\bC$, $m_{\infty}(K_{\infty}^{n+1}/\Lambda) = \det(\Lambda)$. 
Then, the claim is followed by \cite[Lemma 3.10]{BM}.
\end{proof}

\begin{lemma} \label{lem: V ad truncate}
Let $\fa$ be an integral ideal and $\fd \in \fD_f$. 
Suppose $B', B >0$ are real numbers such that  $q_{\fp} \leq B'$ for each $\fp \mid \fd$, and denote by $\omega(\fd)$ the number of prime divisors of $\fd$. Then,
\begin{align*}
    \# \lbrb{ (V^{\fa, \fd} \setminus V^{\fa, \fd}(B')) \cap \cF(B) } \ll \frac{B^{\frac{|u|}{e(f)}}}{B'^{|u| \omega(\fd) }}.
\end{align*}
\end{lemma}
\begin{proof}
By Lemma \ref{lem: V pad lattice} (i), there is a positive integer $t_{\fp}$ such that $V_{\fp}^{\fa, \fd} \setminus V_{\fp}^{\fa, \fd}(B') \subset (\fp^{t_{\fp}+1}\cO_{K, \fp})^u $ and $m((\fp^{t_{\fp}+1}\cO_{K, \fp})^u) \leq B'^{-|u|}$.
Let
\begin{align*}
    \ft \vcentcolon= \lbrb{\bigcap_{\fp \mid \fd} i_{\fp}^{-1}((\fp^{t_{\fp}+1}\cO_{K, \fp})^u)} \bigcap \lbrb{\bigcap_{\fp \nmid \fd} i_{\fp}^{-1}(\cO_{K, \fp}^{n+1}) }.
\end{align*}
Then, $V^{\fa, \fd} \setminus V^{\fa, \fd}(B') \subset \ft$ by the definition of $V^{\fa, \fd}(B')$.
By Lemma \ref{lem:Phi41456} (ii), Lemma \ref{lem: determinant} for $\Lambda = \ft$, and Proposition \ref{prop:refinePhi3} for $\Omega_{\infty} = \cF(1)$ and  trivial local condition $\Omega$, we have
\begin{align*}
    \# \lcrc{ x \in \ft \cap \cF(B)} = \# \lcrc{ x \in \ft \cap B^{\frac{1}{de(f)}}*_{\widetilde{u}}\cF(1)} =
    O \lbrb{ \frac{ B^{ \frac{|u|}{e(f)} }}{|m_{\infty}(K_{\infty}^{n+1}/\ft)|} }
    =O \lbrb{ \frac{ B^{ \frac{|u|}{e(f)} }}{[\cO_K^{n+1} : \ft]} }.
\end{align*}
Using $m((\fp^{t_{\fp}+1}\cO_{K, \fp})^u) \leq B'^{-|u|}$, we have
\begin{align*}
    [\cO_K^{n+1} : \ft] = \prod_{\fp \mid \fd} [\cO_{K, \fp}^{n+1} : \ft_{\fp}]
    = \prod_{\fp \mid \fd} [\cO_{K, \fp}^{n+1} : (\fp^{t_{\fp}+1}\cO_{K, \fp})^u] \geq \prod_{\fp \mid \fd} B'^{|u|}.
\end{align*}
Hence, we obtain the result.
\end{proof}

\begin{lemma} \label{lem: fS lattice vol quotient}
Suppose that $f : \cP(u) \to \cP(w)$ has finite defect and satisfies (\ref{eqn: widehat w condition}), and $B' > q_{\fp}$ for any $\fp \in S_f$.
Let $\fb \subset \fa$ be integral ideals, $\fd \in \fD_f$, and $\Lambda = \fm^{\fa, \fd}(B') \cap \fb^u$.
Then, there is a constant $C(\fa, \fb, \fd, f)$ satisfying 
\begin{align*}
    \frac{\# \fS_{\fc_1}^{\fa, \fd}(B')}{m_{\infty}(K_{\infty}^{n+1}/\Lambda)} = C(\fa, \fb, \fd, f) + O(B'^{-|w|}),
\end{align*}
for sufficiently large $B'$.
\end{lemma}
\begin{proof}
We recall that  
\begin{align*}
    m_{\infty}(K_{\infty}^{n+1}/\Lambda) = [\cO_K^{n+1} : \Lambda] = \prod_{\fp} [\cO_{K, \fp}^{n+1} : \Lambda_{\fp}].
\end{align*}
By definition (\ref{eqn: def fm_fp^fa^fd V_fp^fa^fd}), $\fm_{\fp}^{\fa, \fd}(B') = \fa_{\fp}^u$  if $\fp \nmid \fd$.
So $\Lambda_{\fp} = \fm_{\fp}^{\fa, \fd}(B') \cap \fb_{\fp}^u = \fb_{\fp}^u$ and 
\begin{align*}
    [\cO_{K, \fp}^{n+1} : \Lambda_{\fp}] = [\cO_{K, \fp}^{n+1} : \fb_{\fp}^u] = q_{\fp}^{\ord_{\fp}(\fb)|u|}.
\end{align*}
If $\fp \mid \fd$ and  $B'$ is sufficiently large, then $\fm_{\fp, t_{\fp}}^{\fa, \fd}$ defined in (\ref{eqn: def fm_fp^fa^fd V_fp^fa^fd}) is a subset of $\fb_{\fp}^u$.
So we have $\Lambda_{\fp} = \fm_{\fp, t_\fp}^{\fa, \fd}$ in this case.
Recall that the assumption (\ref{eqn: widehat w condition}) says that
\begin{align*}
    (\ord_{\fp}(\fd) + t e(f) ) w_j \geq (t +1)u_j
\end{align*}
for each $j$. Therefore, we have
\begin{align} \label{eqn: side lattice volume}
    m_{\infty}(K_{\infty}^{n+1}/\Lambda) 
    = \prod_{\fp \nmid \fd} q_{\fp}^{\ord_{\fp}(\fb)|u|} \cdot 
    \prod_{\fp \mid \fd} m_{\fp} (\fm_{\fp, t_{\fp}}^{\fa, \fd})^{-1}
    = \prod_{\fp \nmid \fd} q_{\fp}^{\ord_{\fp}(\fb)|u|} \cdot 
    \prod_{\fp \mid \fd} q_{\fp}^{\ord_{\fp}(\fa)|u| + (\ord_{\fp}(\fd) + t_{\fp}e(f))|w|}.
\end{align}

We first compute $\#\fS_{\fp}^{\fa, \fd}(B')$ instead of $\#\fS_{\fc_1, \fp}^{\fa, \fd}(B')$.
When $\fp \mid \fd$, $\fm_{\fp}^{\fa, \fd}(B') = \fm_{\fp, t_{\fp}}^{\fa, \fd}$ and
\begin{align*}
    V_{\fp}^{\fa, \fd}(B') = \bigsqcup_{0 \leq t \leq t_\fp} V_{\fp, t}^{\fa, \fd} = \bigsqcup_{v \in X_{\fp}^{\fa, \fd}(B')}(v + \fm_{\fp, t_{\fp}}^{\fa, \fd}), \qquad
    V_{\fp, t}^{\fa, \fd} = \bigsqcup_{v \in X_{\fp, t}^{\fa, \fd}}(v + \fm_{\fp, t}^{\fa, \fd}).
\end{align*}
By Lemma \ref{lem: concrete fmfa} (ii) and the condition (\ref{eqn: widehat w condition}),
\begin{align*}
    \fm_{\fp, t+1}^{\fa, \fd}
    &= \fa_{\fp}^u * \lbrb{ \fp^{(\ord_{\fp}(\fd) + (t+1)e(f))w_0 } \cO_{K, \fp} \times \cdots \times \fp^{(\ord_{\fp}(\fd) + (t+1)e(f))w_n } \cO_{K, \fp} } \\
    &= \fp^{e(f)}\cO_{K, \fp} *_w \lbrb{ \fa_{\fp}^u * \lbrb{ \fp^{(\ord_{\fp}(\fd) + te(f))w_0 } \cO_{K, \fp} \times \cdots \times \fp^{(\ord_{\fp}(\fd) + te(f))w_n } \cO_{K, \fp} } } \\
    &= \fp^{e(f)} \cO_{K, \fp}*_w \fm_{\fp, t}^{\fa, \fd}.
\end{align*}
So each $v + \fm_{\fp, t}^{\fa, \fd}$ is covered by $q_{\fp}^{e(f)|w|}$ translates of $\fm_{\fp, t+1}^{\fa, \fd}$. Iterating this relation, we find that $V_{\fp, t}^{\fa, \fd}$, originally covered by $\# X_{\fp, t}^{\fa, \fd}$ translates of $\fm_{\fp, t}^{\fa, \fd}$, is covered by $\# X_{\fp, t}^{\fa, \fd} \cdot q_{\fp}^{(t_{\fp} - t)e(f)|w|}$ translates of $\fm_{\fp, t_{\fp}}^{\fa, \fd}$.
On the other hand, we have $\# X_{\fp, t}^{\fa, \fd} = \# X_{\fp, t+1}^{\fa, \fd}$ by Lemma \ref{lem: concrete fmfa} (iii).
Consequently,
\begin{align*}
    \# X_{\fp}^{\fa, \fd}(B') 
    = \# X_{\fp, 0}^{\fa, \fd} \sum_{0 \leq t \leq t_{\fp}}q_{\fp}^{(t_{\fp} - t)e(f)|w|} 
    = \# X_{\fp, 0}^{\fa, \fd}\sum_{0 \leq t \leq t_{\fp}}q_{\fp}^{te(f)|w|}
    = \# X_{\fp, 0}^{\fa, \fd}\frac{q_{\fp}^{(t_{\fp}+1)e(f)|w|} - 1}{q_{\fp}^{e(f)|w|} - 1}.
\end{align*}

Now we consider $\# \fS_{\fc_1, \fp}^{\fa, \fd}(B')$. 
Following (\ref{eqn: def fm_fp^fa^fd V_fp^fa^fd}), we define $X_{\fp, t}^{\fa, \fd}(B')$ as the set satisfying
\begin{align*}
    V_{\fp, t}^{\fa, \fd} = \bigsqcup_{v \in X_{\fp, t}^{\fa, \fd}(B') } (v + \fm_{\fp}^{\fa, \fd}(B')).
\end{align*}
Since $V_{\fp, t}^{\fa, \fd} = \fp^t\cO_{K, \fp}*_u V_{\fp, 0}^{\fa, \fd}$, we have $V_{\fp, t}^{\fa, \fd} \subset (\fa\fc_1)_{\fp}^u$ if $t \geq \ord_{\fp}(\fa\fc_1)$.
In other words, 
the effect of intersection with $(\fa\fc_1)_{\fp}^u$ on $V_{\fp}^{\fa, \fd}(B') \cap (\fa\fc_1)_{\fp}^u$ is only non-trivial for $t < \ord_{\fp}(\fa\fc_1)$ in $V_{\fp}^{\fa, \fd}(B') = \bigsqcup_{0 \leq t \leq t_{\fp}}V_{\fp, t}^{\fa, \fd}$.
Since $\fm_{\fp, t}^{\fa, \fd} = \fa_{\fp}*_u \fp^{\widehat{w}}\cO_{K, \fp}$ with $\widehat{w}_j \geq (t+1)u_j$, we also have $\fm_{\fp, t_{\fp}}^{\fa, \fd} \subset (\fa\fc_1)_{\fp}^u$ if $t_{\fp} \geq \ord_{\fp}(\fa\fc_1)$.
Therefore, if $B'$ is sufficiently large
\begin{align*}
    \fm_{\fc_1, \fp}^{\fa, \fd}(B') = \fm_{\fp, t_{\fp}}^{\fa, \fd} \cap (\fa\fc_1)_{\fp}^u = \fm_{\fp, t_{\fp}}^{\fa, \fd} = \fm_{\fp}^{\fa, \fd}(B').
\end{align*}
Hence, $V_{\fp, t}^{\fa, \fd} \cap (\fa\fc_1)^u$ is a union of translate of $\fm_{\fc_1, \fp}^{\fa, \fd}(B') = \fm_{\fp}^{\fa, \fd}(B')$. 
So there is a finite set $X_{\fc_1, \fp, t}^{\fa, \fd}(B')$ such that
\begin{align*}
    V_{\fp, t}^{\fa, \fd} \cap (\fa\fc_1)^u = \bigsqcup_{v \in X_{\fc_1, \fp, t}^{\fa, \fd}(B')  } (v + \fm_{\fp}^{\fa, \fd}(B')).
\end{align*}
We define
\begin{align*}
    e_{\fc_1, \fp, t}^{\fa, \fd}(B') \vcentcolon = \# X_{\fp, t}^{\fa, \fd}(B') - \# X_{\fc_1, \fp, t}^{\fa, \fd}(B').
\end{align*}
Because $V_{\fp, t}^{\fa, \fd} \cap (\fa \fc_1)^u = V_{\fp, t}^{\fa, \fd}$ if $t \geq \ord_{\fp}(\fa\fc_1)$, we have $e_{\fc_1, \fp, t}^{\fa, \fd}(B') = 0$ if $t \geq \ord_{\fp}(\fa\fc_1)$.

Since $\# X_{\fp, t}^{\fa, \fd}(B')$ (resp. $\# X_{\fp, t}^{\fa, \fd}$) is the number of covers of $V_{\fp, t}^{\fa, \fd}$, 
which consist of translates of $\fm_{\fp, t_{\fp}}^{\fa, \fd}$ (resp. $\fm_{\fp, t}^{\fa, \fd}$), 
we have 
\begin{align*}
    \# X_{\fp, t}^{\fa, \fd}(B') = \# X_{\fp, t}^{\fa, \fd} \frac{m_{\fp}(\fm_{\fp, t}^{\fa, \fd})}{m_{\fp}(\fm_{\fp, t_{\fp}}^{\fa,\fd})}.
\end{align*}
By the same argument, we also have
\begin{align*}
    \# X_{\fc_1, \fp, t}^{\fa, \fd}(B') = \# X_{\fc_1, \fp, t}^{\fa, \fd} \frac{m_{\fp}(\fm_{\fp, t}^{\fa, \fd})}{m_{\fp}(\fm_{\fp, t_{\fp}}^{\fa,\fd})}.
\end{align*}
Therefore,
\begin{align*}
e_{\fc_1, \fp, t}^{\fa, \fd}(B')
    =(\# X_{\fp, t}^{\fa, \fd} - \# X_{\fc_1, \fp, t}^{\fa, \fd} ) \frac{ m_{\fp}(\fm_{\fp, t}^{\fa, \fd}) }{ m_{\fp}(\fm_{\fp,t_{\fp} }^{\fa, \fd})  }
    = (\# X_{\fp, t}^{\fa, \fd} - \# X_{\fc_1, \fp, t}^{\fa, \fd} ) q_{\fp}^{(t_{\fp}-t)e(f)|w|}.
\end{align*}
Since $\# X_{\fp}^{\fa, \fd}(B')$ is equal to the sum of $\# X_{\fp, t}^{\fa, \fd}(B')$ over $0 \leq t \leq t_{\fp}$,
\begin{align*}
    \# \fS_{\fc_1, \fp}^{\fa, \fd}(B')
    &= \sum_{0 \leq t < \ord_{\fp}(\fa \fc_1)} \# X_{\fc_1, \fp, t}^{\fa, \fd}(B')
    +  \sum_{\ord_{\fp}(\fa \fc_1) \leq t \leq t_{\fp}}  \# X_{\fc_1, \fp, t}^{\fa, \fd}(B') \\
    &= \sum_{0 \leq t < \ord_{\fp}(\fa \fc_1)} \lbrb{\# X_{\fp, t}^{\fa, \fd}(B') - e_{\fc_1, \fp, t}^{\fa, \fd}(B') }
    +  \sum_{\ord_{\fp}(\fa \fc_1) \leq t \leq t_{\fp}}  \# X_{\fp, t}^{\fa, \fd}(B') \\
    &=
    \# X_{\fp, 0}^{\fa, \fd}\frac{q_{\fp}^{(t_{\fp}+1)e(f)|w|} - 1}{q_{\fp}^{e(f)|w|} - 1} -  \sum_{0 \leq t < \ord_{\fp}(\fa \fc_1)} e_{\fc_1, \fp, t}^{\fa, \fd}(B') .
\end{align*}
For simplicity, we introduce
\begin{align*}
    e_{\fp}(B') \vcentcolon = \sum_{0 \leq t < \ord_{\fp}(\fa \fc_1)} e_{\fc_1, \fp, t}^{\fa, \fd}(B').
\end{align*}
Then
\begin{align} \label{eqn: side fS}
    \# \fS_{\fc_1}^{\fa, \fd}(B') = \prod_{\fp \mid \fd} \# \fS_{\fc_1, \fp}^{\fa, \fd}(B')
    = \prod_{\fp \mid \fd}
    \lbrb{ e_{\fp}(B') +\# X_{\fp, 0}^{\fa, \fd}\frac{q_{\fp}^{(t_{\fp}+1)e(f)|w|} - 1}{q_{\fp}^{e(f)|w|} - 1} }  .
\end{align}
In the remainder of the proof, we use
\begin{align*}
    c_{\fp} \vcentcolon = \ord_{\fp}(\fd) |w| + \ord_{\fp}(\fa) |u|.
\end{align*}
Then by (\ref{eqn: side lattice volume}) and (\ref{eqn: side fS}), we have
\begin{align*}
    \frac{\#\fS_{\fc_1}^{\fa, \fd}(B') }{m_{\infty}(K^{n+1}_{\infty}/\Lambda) }
    &= \prod_{\fp \nmid \fd} q_{\fp}^{-\ord_{\fp}(\fb)|u|}
    \prod_{\fp \mid \fd} \frac{e_{\fp}(B') + \# X_{\fp, 0}^{\fa, \fd} 
    \lbrb{q_{\fp}^{t_{\fp}e(f)|w|} + q_{\fp}^{e(f)|w|(t_{\fp}-1)} + \cdots + 1} }{q_{\fp}^{\ord_{\fp}(\fa)|u| + (t_{\fp}e(f) + \ord_{\fp}(\fd)) |w| }} \\
    &= \prod_{\fp \nmid \fd} q_{\fp}^{-\ord_{\fp}(\fb)|u|}
    \prod_{\fp \mid \fd} \lbrb{ \frac{e_{\fp}(B')  }{q_{\fp}^{t_{\fp}e(f)|w| +c_{\fp} }} + \# X_{\fp, 0}^{\fa, \fd} \lbrb{q_{\fp}^{-c_{\fp}} + q_{\fp}^{ - e(f)|w| -c_{\fp}} + \cdots + q_{\fp}^{-e(f)|w|t_{\fp} - c_{\fp} }} }.
\end{align*}
By the definition of $e_{\fc_1,\fp,t}^{\fa, \fd}(B')$, we have
\begin{align*}
    \frac{e_{\fp}(B')  }{q_{\fp}^{t_{\fp}e(f)|w| +c_{\fp} }} 
    = \sum_{0 \leq t \leq \ord_{\fp}(\fa\fc_1)}
    (\# X_{\fp, t}^{\fa, \fd} - \# X_{\fc_1, \fp, t}^{\fa, \fd} ) 
    q_{\fp}^{ -te(f)|w| - c_{\fp} }.
\end{align*}
So it does not depend on $B'$.
Since $q_{\fp}^{t_{\fp}} \asymp B'$,
\begin{align*}
q_{\fp}^{-c_{\fp}} + q_{\fp}^{ - |w| -c_{\fp}} + \cdots + q_{\fp}^{-|w|t_{\fp} - c_{\fp} }
= q_{\fp}^{-c_{\fp}}\left(1 + O(B'^{-|w|}) \right).
\end{align*}
Therefore,
\begin{align*}
    C(\fa, \fb,\fd, f) \vcentcolon = 
    \prod_{\fp \nmid \fd} q_{\fp}^{-\ord_{\fp}(\fb)|u|}
    \prod_{\fp \mid \fd} \lbrb{ \sum_{0 \leq t \leq \ord_{\fp}(\fa\fc_1)}
    (\# X_{\fp, t}^{\fa, \fd} - \# X_{\fc_1, \fp, t}^{\fa, \fd} ) 
    q_{\fp}^{ -te(f)|w| - c_{\fp} }
    + \# X_{\fp, 0}^{\fa, \fd}q_{\fp}^{-c_{\fp}} }
\end{align*}
satisfies the equation in the statement.
\end{proof}

\begin{proposition} \label{prop: V b Omega F}
Let $\fb \subset \fa$ be integral ideals, let $\fd \in \fD_f$, and let $f: \cP(u) \to \cP(w)$ be a morphism that has finite defect and satisfies (\ref{eqn: widehat w condition}).
Let $\Omega = \lcrc{\Omega_{\fp}^a}_{\fp \in S}$ be a set of irreducible affine local conditions defined by 
\begin{align*}
    \Omega_{\fp}^a = \prod_{j=0}^n \lcrc{x_{\fp, j} \in K_{\fp}  : |x_{\fp, j} - a_{\fp, j}| \leq \omega_{\fp, j}} = \prod_{j=0}^n (a_{\fp, j} + \fp^{r_{\fp, j}}\cO_{K, \fp})
\end{align*}
for some $a_{\fp, j} \in \cO_{K, \fp}$, $r_{\fp, j} \geq 0$, and $\omega_{\fp, j} = N_{K/\bQ}(\fp^{r_{\fp, j}})$.
Suppose each prime $\fp$ in $S$ is not in $S_f$. 
Then for $C(\fa, \fb, \fd, f)$ is an absolute constant in Lemma \ref{lem: fS lattice vol quotient}, $\Lambda = \fm_{\fc_1}^{\fa, \fd}(B) \cap (\fa \fc_0)^{u}$ and $d = [K:\bQ]$,
\begin{align*}
    &\#\lbrb{V^{\fa, \fd} \cap \fb^{u} \cap \Omega_S \cap \cF(B)} \\
    &=C(\fa, \fb, \fd, f)  m_{\infty}(\cF(1))
    \prod_{\fp \in S}
    \frac{m_{\fp}(\Omega_\fp^a \cap \Lambda_{\fp} ) }{m_{\fp}(\Lambda_{\fp})}B^{\frac{|u|}{e(f)}} +
    O\lbrb{ \epsilon(\Omega)
    \frac{B^{\frac{d|u| - u_{\min}}{de(f)}}}{N_{K/\bQ}(\fc_0)^{{\frac{d|u|-u_{\min}}{d}}}} }.
\end{align*}
\end{proposition}
\begin{remark} \label{rmk: volume not dep on B}
Since $\fp \not\in S_f$, $\fm_{\fp}^{\fa, \fd}(B) = \fa_{\fp}^u$ by definition (\ref{eqn: def fm_fp^fa^fd V_fp^fa^fd}).
Therefore
\begin{align*}
    \frac{m_{\fp}(\Omega_{\fp}^a \cap \Lambda_{\fp})}{m_{\fp}(\Lambda_{\fp})}
\end{align*}
does not depend on $B$ even though $\Lambda$ does.
\end{remark}
\begin{proof}
When $\fd = \cO_K$, then $V^{\fa, \fd} = \fa^u \setminus \lcrc{0}$ and the result follows from the arguments below and Proposition \ref{prop:phi3 327}.
We henceforth assume that $\omega(\fd) \geq 1$.

By Lemma \ref{lem: V ad truncate}, we have
\begin{align*}
    \#\left( (V^{\fa, \fd} \setminus V^{\fa, \fd}(B')) \cap \fb^u \cap \Omega_S \cap \cF(B) \right)
    \leq \#\left( (V^{\fa, \fd} \setminus V^{\fa, \fd}(B')) \cap \cF(B) \right)
    \ll \frac{B^{\frac{|u|}{e(f)}}}{B'^{|u|\omega(\fd) }}.
\end{align*}
So by choosing $B' > B^{\frac{u_{\min}}{d |u| e(f)}}$, we can use $V^{\fa, \fd}(B')$ instead of $V^{\fa, \fd}$.
By Lemma \ref{lem: V^ad b^w translation},
\begin{align*}
    \# \lbrb{  V^{\fa, \fd}(B') \cap \fb^{u} \cap  \Omega_S  \cap \cF( B) }
    &= \# \lbrb{ \bigsqcup_{v \in X_{\fb}^{\fa, \fd}(B')}  \lbrb{v+ \fm_{\fc_1}^{\fa, \fd}(B') \cap (\fa \fc_0)^{u} } \cap   \Omega_S \cap \cF( B) }.
\end{align*}
Recall that the diagonal embedding $i_{\infty}: K^{n+1} \to K_{\infty}^{n+1}$ is omitted here.
Here we set $\Lambda = \fm_{\fc_1}^{\fa, \fd}(B') \cap (\fa \fc_0)^{u}$. 
By Lemma \ref{lem:Phi41456} (ii), we have
\begin{align*}
    &\lbrb{v+ \fm_{\fc_1}^{\fa, \fd}(B') \cap (\fa \fc_0)^{u} } \cap   \Omega_S \cap \cF( B) \\
    &= \lcrc{x \in (v + \Lambda) \cap \left(B^{\frac{1}{e(f)d}}*_{\widetilde{u}} \cF(1) \right) : i_{\fp}(x) \in \Omega_{\fp}^a \textrm{ for all } \fp \in S }.
\end{align*}

Let $\fq$ be a prime ideal of $K$.
By definition, $\lbrb{\fm_{\fc_1}^{\fa, \fd}(B') \cap (\fa \fc_0)^{u}}_{\fq}$ is an $(n+1)$-tuple of $\fq$-powers, whether or not $\fq$ divides $\fd$.
Hence, we can apply Proposition \ref{prop:refinePhi3} to
\begin{align*}
\Lambda = 
    \fm_{\fc_1}^{\fa, \fd}(B') \cap (\fa \fc_0)^{u} 
    =\bigcap_{\fq} i_{\fq}^{-1} \lbrb{ \fm_{\fc_1, \fq}^{\fa, \fd}(B') \cap (\fa \fc_0)_{\fq}^{u} }
\end{align*}
and $\Omega_{\infty} = \cF(1)$.
Since $\# X_{\fb}^{\fa, \fd}(B') = \# \fS_{\fc_1}^{\fa, \fd}(B')$ by Lemma \ref{lem: V^ad b^w translation}, we have
\begin{align} \label{eqn: Vab b Omega F}
    \# \lbrb{ V^{\fa, \fd}(B') \cap \fb^{u} \cap  \Omega_S \cap \cF( B) }
    =   \frac{\#\fS_{\fc_1}^{\fa, \fd}(B') m_{\infty}(\cF(1))}{m_{\infty}(K^{n+1}_{\infty}/\Lambda) } \prod_{\fp \in S}
    \frac{m_{\fp}(\Omega_\fp^a \cap \Lambda_{\fp} ) }{m_{\fp}(\Lambda_{\fp})}B^{\frac{|u|}{e(f)}} +O\lbrb{\epsilon(\Omega) B^{\frac{d|u| - u_{\min}}{d e(f)}} }.
\end{align}
Recall that
\begin{align*}
    \epsilon(\Omega) = \max_{0 \leq j \leq n } \lcrc{\prod_{\fp \in S} \omega_{\fp, j}^{-1}(\Omega)} 
    \prod_{\fp \in S} \frac{m_{\fp}(\Omega_\fp^a \cap \Lambda_{\fp} ) }{m_{\fp}(\Lambda_{\fp})}.
\end{align*}
By Lemma \ref{lem: fS lattice vol quotient}, the right-hand side of (\ref{eqn: Vab b Omega F}) is
\begin{align*}
m_{\infty}(\cF(1)) \lbrb{C(\fa, \fb, \fd,  f) + O(B'^{-|w|})} \prod_{\fp \in S}
    \frac{m_{\fp}(\Omega_\fp^a \cap \Lambda_{\fp} ) }{m_{\fp}(\Lambda_{\fp})}B^{\frac{|u|}{e(f)}} +O\lbrb{\epsilon(\Omega) B^{\frac{d|u| - u_{\min}}{d e(f)}} }.
\end{align*}
If we choose $B' > B^{\frac{u_{\min}}{de(f)|w|}}$, then $O(B'^{-|w|})$-term is absorbed to the error term.

Let $\lcrc{\fa_i}_{i=1}^h$ be a set of integral ideals such that $\lcrc{[\fa_i]}_{i=1}^h$ is a set of representatives of the class group of $K$.
For $\fc_0 \in \widetilde{\fD_f}$ and $\fc_1 \in \widehat{\fD_f}$ satisfying $\fb = \fa \fc_0 \fc_1$, we choose $\lambda \in K^\times$ so that $\fc_0 = \lambda \fa_i$ for some $i$. 
Then, $\lambda\cO_{K}$ is not divided by any prime ideals in $S_f$.
By (\ref{eqn: weighted action fI f delta}),
we have $\delta_f(\lambda*_{u}x) = \delta_f(x)$ and $\fI_{u}(\lambda*_{u}x) = \lambda \fI_{u}(x)$.
Also, 
\begin{align*}
    H_{w, \infty} (f(\lambda*_{u}x)) &= H_{w, \infty} (\lambda^{e(f)}*_{w}f(x)) = \prod_{v \in M_{K, \infty}} \max_{0 \leq j \leq n } \lcrc{|\lambda|_v^{e(f)} |f(x)_j|_v^{\frac{1}{w_j}}} = N_{K/\bQ}(\lambda)^{e(f)} H_{w, \infty}(f(x)).
\end{align*}
Therefore,
\begin{align*}
    & V^{\fa, \fd}(B') \cap \fb^u \cap \Omega_S \cap \cD(B) \\
    & = \lcrc{x \in  V^{\fa, \fd }(B') : \fI_u(x)  \subset \lambda \fa \fa_i \fc_1, \delta_f(x) \subset \fd, H_{w, \infty}(f(x)) \leq B, i_{\fp}(x) \in \Omega_{\fp}^a \textrm{ for } \fp \in S    } \\
    &=\lcrc{x \in V^{\fa, \fd}(B') : 
    \begin{array}{cc}
    \fI_{u}(\lambda^{-1}*_{u}x) \subset \fa \fa_i \fc_1,     &  \\
    \delta_f(\lambda^{-1}*_{u}x) \subset \fd,    & 
    \end{array}
     H_{w, \infty}(f(\lambda^{-1}*_{u}x)) \leq \frac{B}{N_{K/\bQ}(\lambda)^{e(f)}},
     i_{\fp}(x) \in \Omega_{\fp}^a \textrm{ for } \fp \in S 
    }  \\
    &=\lcrc{\lambda*_u x \in V^{\fa, \fd}(B') : 
    \begin{array}{cc}
    \fI_{u}(x) \subset \fa \fa_i \fc_1,     &  \\
    \delta_f(x) \subset \fd,    & 
    \end{array}
     H_{w, \infty}(f(x)) \leq \frac{B}{N_{K/\bQ}(\lambda)^{e(f)}},
     i_{\fp}(\lambda*_ux) \in \Omega_{\fp}^a \textrm{ for } \fp \in S 
    }  \\
    &= (\lambda^{-1}*_u V^{\fa, \fd}(B')) \cap (\fa \fa_i \fc_1)^u \cap \lbrb{\lambda^{-1}*_u \Omega_{S}} \cap \cD \lbrb{\frac{B}{N_{K/\bQ}(\lambda)^{e(f)}}}.
\end{align*}
By Lemma \ref{lem: M' step1} (i),  we have
\begin{align*}
    \# \lbrb{V^{\fa, \fd}(B') \cap \fb^u \cap \Omega_S \cap \cF(B)}
    = \# \lbrb{V^{\lambda^{-1}\fa, \fd}(B') \cap (\fa\fa_i\fc_1)^u \cap 
    \lbrb{\lambda^{-1}*_u\Omega_S  }\cap \cF\lbrb{\frac{B}{N_{K/\bQ}(\lambda)^{e(f)}}}  }.
\end{align*}

We now compute the right-hand side.
Since $\fa^{-1}\fb = (\lambda^{-1}\fa)^{-1}\lambda^{-1}\fb$, we have
\begin{align*}
    \fc_i(\fa^{-1}\fb) = \fc_i((\lambda^{-1}\fa)^{-1}\lambda^{-1}\fb).
\end{align*}
In particular, we also have
\begin{align*}
    V^{\lambda^{-1}\fa, \fd}(B') \cap (\fa \fa_i \fc_1(\fa^{-1}\fb))^u
    =V^{\lambda^{-1}\fa, \fd}(B') \cap (\fa \fa_i \fc_1((\lambda^{-1}\fa)^{-1}\lambda^{-1}\fb))^u.
\end{align*}
We will use the simplified notation $\fc_i$ for both $\fc_i(\fa^{-1}\fb)$ and $\fc_i((\lambda^{-1}\fa)^{-1}\lambda^{-1}\fb))$.

Since $V^{\lambda^{-1}\fa, \fd}_{\fc_1}(B')$ is stable under addition by $\fm_{\fc_1}^{\lambda^{-1}\fa, \fd}(B')$,
$V^{\lambda^{-1}\fa, \fd}_{\fc_1}(B') \cap (\lambda^{-1}\fa\fc_0)^u$ is stable under addition by $\fm_{\fc_1}^{\lambda^{-1}\fa, \fd}(B') \cap (\lambda^{-1}\fa\fc_0)^u$.
By Lemma \ref{lem: V^ad b^w translation} for $\lambda^{-1}\fb \subset \lambda^{-1}\fa$,
\begin{align*}
    &\# \lbrb{V^{\lambda^{-1} \fa, \fd}(B') \cap (\fa \fa_i \fc_1)^{u} \cap \lbrb{\lambda^{-1}*_{u}\Omega_S} \cap \cF\lbrb{\frac{B}{N_{K/\bQ}(\lambda)^{e(f)}}} } \\
    &= \# \lbrb{V^{\lambda^{-1} \fa, \fd}_{\fc_1}(B') \cap (\lambda^{-1}\fa \fc_0)^{u} \cap \lbrb{\lambda^{-1}*_{u}\Omega_S} \cap \cF\lbrb{\frac{B}{N_{K/\bQ}(\lambda)^{e(f)}}} }  \\
    &= \# \lbrb{ \bigsqcup_{v \in X_{\fb}^{\lambda^{-1}\fa, \fd}(B')} \lbrb{v + \fm_{\fc_1}^{\lambda^{-1}\fa, \fd}(B') \cap (\lambda^{-1}\fa\fc_0)^{u} } \cap  \lbrb{\lambda^{-1}*_{u}\Omega_S} \cap \cF\lbrb{\frac{B}{N_{K/\bQ}(\lambda)^{e(f)}}} }.
\end{align*}
Applying Proposition \ref{prop:refinePhi3} for $\Lambda = \fm_{\fc_1}^{\lambda^{-1}\fa, \fd}(B') \cap (\lambda^{-1}\fa\fc_0)^{u}$ and $\Omega_{\infty} = \cF(1)$, and Lemma \ref{lem: V^ad b^w translation}, it is
\begin{align} \label{eqn: Vab b Omega F lambda}
    &\frac{ \# \fS^{\lambda^{-1}\fa, \fd}_{\fc_1} (B') m_{\infty}(\cF(1) ) }{m_{\infty}(K^{n+1}_{\infty}/\Lambda) } 
    \prod_{\fp \in S}
    \frac{m_{\fp}(( \lambda^{-1}*_{u}\Omega_S)_\fp \cap \Lambda_{\fp} ) }{m_{\fp}(\Lambda_{\fp})} 
    \frac{B^{\frac{|u|}{e(f)}}}{N_{K/\bQ}(\lambda)^{|u|}}
    +O\lbrb{\epsilon(\lambda^{-1}*_{u}\Omega) \frac{B^{\frac{d|u| - u_{\min}}{d e(f)}}}{N_{K/\bQ}(\lambda)^{\frac{d|u| - u_{\min}}{d}}} }
\end{align}
where 
\begin{align*}
   \epsilon(\lambda^{-1}*_{u}\Omega) = \max_{0 \leq j \leq n} \lcrc{\prod_{\fp \in S} \omega_{\fp, j}^{-1}(\lambda^{-1}*_{u}\Omega) }\prod_{\fp \in S} 
 \frac{m_{\fp}(( \lambda^{-1}*_{u}\Omega_S)_\fp \cap \Lambda_{\fp} )}{m_{\fp}(\Lambda_{\fp})}.
\end{align*}
We will compare (\ref{eqn: Vab b Omega F}) and (\ref{eqn: Vab b Omega F lambda}).
Lemma \ref{lem: temp fS} shows that $\#\fS^{\lambda^{-1}\fa, \fd}_{\fc_1}(B') = \#\fS^{\fa, \fd}_{\fc_1}(B')$. 
By Lemma \ref{lem: determinant}, we have
\begin{align*}
    \frac{ m_{\infty}(K_{\infty}^{n+1}/(\fm_{\fc_1}^{\fa, \fd}(B') \cap (\fa \fc_0)^{u}) ) }{ m_{\infty}(K_{\infty}^{n+1}/(\fm_{\fc_1}^{\lambda^{-1}\fa, \fd}(B') \cap (\lambda^{-1}\fa\fc_0)^{u})) }
    =\frac{[\cO_K^{n+1}: \fm_{\fc_1}^{\fa, \fd}(B')]}{[\cO_K^{n+1}: \fm_{\fc_1}^{\lambda^{-1}\fa, \fd}(B')]}.
\end{align*}
For a prime ideal $\fq \mid \fd$ and a sufficiently large $B'$, we have shown that 
\begin{align*}
    \fm_{\fc_1, \fq}^{\fa, \fd}(B') = \fm_{\fq}^{\fa, \fd}(B') = \fm_{\fq, t_{\fq}}^{\fa, \fd}
\end{align*}
in the proof of Lemma \ref{lem: fS lattice vol quotient}.
Since $\lambda \cO_K$ is relatively prime with an ideal in $S_f$, we have $\fa_{\fq}^{u} = (\lambda^{-1}\fa)_{\fq}^u$.
Therefore,
\begin{align*}
    \fm_{\fc_1, \fq}^{\lambda^{-1}\fa, \fd}(B') = \fm_{\fq}^{\lambda^{-1}\fa, \fd}(B') 
    = (\lambda^{-1} \fa)_{\fq}^u* (\fq\cO_{K, \fq})^{\widehat{w}}
    = \fa_{\fq}^u* (\fq\cO_{K, \fq})^{\widehat{w}} = \fm_{\fq}^{\fa, \fd}(B') = \fm_{\fc_1, \fq}^{\fa, \fd}(B').
\end{align*}
When $\fq \nmid \fd$, we have
\begin{align*}
    \fm_{\fc_1, \fq}^{\lambda^{-1}\fa, \fd}(B') = (\lambda^{-1} \fa)_{\fq}^u \cap (\lambda^{-1} \fa \fc_1)_{\fq}^u = (\lambda^{-1} \cO_{K, \fq})^u * \fm_{\fc_1, \fq}^{\fa, \fd}(B').
\end{align*}
Since $\lambda\cO_K$ is not divided by the prime divisors of $\fd$, we have
\begin{align} \label{eqn: det compare}
    m_{\infty}(K_{\infty}^{n+1}/(\fm_{\fc_1}^{\fa, \fd}(B') \cap (\fa \fc_0)^{u}) )
    = m_{\infty}(K_{\infty}^{n+1}/(\fm_{\fc_1}^{\lambda^{-1}\fa, \fd}(B') \cap (\lambda^{-1}\fa\fc_0)^{u})) N_{K/\bQ}(\lambda)^{|u|}.
\end{align}

For a prime ideal $\fp \in S$, we denote $i_{\fp}(\lambda) \in K_{\fp}$ by $\lambda_{\fp}$ in the remainder of the proof.
Then $(\lambda^{-1}\fa\fc_0)_{\fp}^u = \lambda_{\fp}^{-1}*_u(\fa\fc_0)_{\fp}^u$.
Since $\fp \not\in S_f$, we have $\fm_{\fp}^{\fa, \fd}(B') = \fa_{\fp}^u$ and 
\begin{align*}
     ( \fm_{\fc_1}^{\lambda^{-1}\fa, \fd}(B') \cap (\lambda^{-1}\fa\fc_0)^{u} )_{\fp}
    &=  (\lambda^{-1}_{\fp}  *_u \fm_{\fp}^{\fa, \fd}(B')) \cap  (\lambda_{\fp}^{-1}*_u(\fa\fc_0\fc_1)_{\fp}^{u})   
    = \lambda_{\fp}^{-1}*_u ( \fm_{\fp}^{\fa, \fd}(B') \cap (\fa\fc_0)_{\fp}^{u}).
\end{align*}
Since $(\lambda^{-1}*_u \Omega_S)_{\fp} = \lambda_{\fp}^{-1}*_u\Omega_{\fp}^a$,
there are $a_{\fp, j} \in \cO_{K, \fp}$ and $r_{\fp, j} \geq 0$ such that
\begin{align*}
    \Omega_{\fp}^a &= \prod_{j=0}^n \lcrc{x_{\fp, j} \in K_{\fp} : |x_{\fp, j} - a_{\fp, j}| \leq \omega_{\fp, j}} = \prod_{j=0}^n (a_{\fp, j} + \fp^{r_{\fp, j}}\cO_{K, \fp}),  \\
    \lambda_{\fp}^{-1}*_{u} \Omega_{\fp}^a &= \prod_{j=0}^n \lcrc{\lambda_{\fp}^{-u_j} x_{\fp, j} \in K_{\fp} : |x_{\fp, j} - a_{\fp, j}|_{\fp} \leq \omega_{\fp, j} }
    = \prod_{j=0}^n(\lambda_{\fp}^{-u_j} a_{\fp, j} + \lambda_{\fp}^{-u_j}\fp^{r_{\fp, j}}\cO_{K, \fp} ).
\end{align*}
Hence $\lambda_{\fp}^{-1}*_u\Omega_{\fp}^a$ is also a box-type set.
Under the $u$-weighted action by $\lambda^{-1}_{\fp}$, the Haar measure of any box-type subset, including $\Omega_{\fp}^a$ and $\fm_{\fp}^{\fa, \fd}(B') \cap (\fa\fc_0)^u_{\fp}$, is multiplied by $q_{\fp}^{\ord_{\fp}(\lambda_{\fp}) |u|}$.
Therefore,
\begin{align}
    \prod_{\fp \in S}
    \frac{m_{\fp}(( \lambda^{-1}*_{u}\Omega_S)_\fp \cap \Lambda_{\fp} ) }{m_{\fp}(\Lambda_{\fp})} 
    \textrm{ of (\ref{eqn: Vab b Omega F lambda})}
    &= \prod_{\fp \in S}
    \frac{m_{\fp}((\lambda_{\fp}^{-1}*_{u}\Omega_\fp^a) \cap ( \fm_{\fc_1}^{\lambda^{-1}\fa, \fd}(B') \cap (\lambda^{-1}\fa\fc_0)^{u}  )_{\fp} ) }{m_{\fp}(( \fm_{\fc_1}^{\lambda^{-1}\fa, \fd}(B') \cap (\lambda^{-1}\fa\fc_0)^{u}  )_{\fp}   )} \nonumber \\ \nonumber
    &= \prod_{\fp \in S}
    \frac{m_{\fp}((\lambda_{\fp}^{-1}*_{u}\Omega_\fp^a) \cap 
    ( \lambda_{\fp}^{-1}*_u ( \fm_{\fp}^{\fa, \fd}(B') \cap (\fa\fc_0)_{\fp}^{u})   ) }{m_{\fp}( \lambda_{\fp}^{-1}*_u ( \fm_{\fp}^{\fa, \fd}(B') \cap (\fa\fc_0)_{\fp}^{u})   ) } \\ \nonumber
    &= \prod_{\fp \in S}
    \frac{m_{\fp}(\Omega_\fp^a \cap 
    ( \fm_{\fp}^{\fa, \fd}(B') \cap (\fa\fc_0)_{\fp}^{u})    }{m_{\fp}(  \fm_{\fp}^{\fa, \fd}(B') \cap (\fa\fc_0)_{\fp}^{u})   } \\ \label{eqn: compare Haar measure proportion}
    &=\prod_{\fp \in S} \frac{m_{\fp}(\Omega_{\fp}^a \cap \Lambda_{\fp})}{m_{\fp}(\Lambda_{\fp})}
    \textrm{ of (\ref{eqn: Vab b Omega F}) }.
\end{align}
Together with (\ref{eqn: det compare}), we showed that the main terms of (\ref{eqn: Vab b Omega F}) and (\ref{eqn: Vab b Omega F lambda}) are equal.

By (\ref{eqn: compare Haar measure proportion}) and
\begin{align*}
    \omega_{\fp, j}(\lambda^{-1}*_{u}\Omega) = q_{\fp}^{u_j\ord_{\fp}(\lambda_{\fp}) - r_{\fp, j}} = \omega_{\fp, j}(\Omega)  q_{\fp}^{u_j\ord_{\fp}(\lambda_{\fp})},
\end{align*}
the error term in (\ref{eqn: Vab b Omega F lambda}) is
\begin{align*}
    &O\lbrb{\max_{0 \leq j \leq n} \lcrc{\prod_{\fp \in S} \omega_{\fp, j}^{-1}(\lambda^{-1}*_{u}\Omega) }
    \prod_{\fp \in S} \frac{m_{\fp}((\lambda^{-1}*_{u}\Omega)_\fp \cap ( \fm_{\fc_1}^{\lambda^{-1}\fa, \fd}(B') \cap (\lambda^{-1}\fa\fc_0)^{u}  )_{\fp} )}{m_{\fp}(( \fm_{\fc_1}^{\lambda^{-1}\fa, \fd}(B') \cap (\lambda^{-1}\fa\fc_0)^{u}  )_{\fp})}  
    \frac{B^{\frac{d|u| - u_{\min}}{d e(f)}}}{N_{K/\bQ}(\lambda)^{ {\frac{d|u|-u_{\min}}{d}} }} } \\
    &= O\lbrb{\max_{0 \leq j \leq n} \lcrc{\prod_{\fp \in S} \omega_{\fp, j}^{-1}(\Omega)q_{\fp}^{-u_j \ord_{\fp}(\lambda_{\fp})} }
    \prod_{\fp \in S}
    \frac{m_{\fp}(\Omega_\fp^a \cap ( \fm_{\fc_1}^{\fa, \fd}(B') \cap (\fa\fc_0)^{u}  )_{\fp} ) }{m_{\fp}(( \fm_{\fc_1}^{\fa, \fd}(B') \cap (\fa\fc_0)^{u} )_{\fp})}
    \frac{B^{\frac{d|u| - u_{\min}}{de(f)}}}{N_{K/\bQ}(\fc_0)^{{\frac{d|u|-u_{\min}}{d}}}} } \\
    &= O\lbrb{ \epsilon(\Omega)
    \frac{B^{\frac{d|u| - u_{\min}}{de(f)}}}{N_{K/\bQ}(\fc_0)^{{\frac{d|u|-u_{\min}}{d}}}} },
\end{align*}
which appears in the statement.
Here, $q_{\fp}^{-u_j \ord_{\fp}(\lambda_{\fp})} \leq 1$ is used and the implied constant depends on $\fa_i$, where $\fc_0 = \lambda \fa_i$.
To finish the proof, it suffices to choose $B' = B$.
\end{proof}


\begin{remark} \label{rmk: BM cN}
Let $\fb = \fa \fc_0 \fc_1$ and $\fc_0 = \lambda \mathfrak{r}$.
In the proof of \cite[Lemma 3.12]{BM}, the following equalities appear:
\begin{align*}
    \mathcal{N}_{\phi}(\fb, \fa, \fd, T) = \mathcal{N}_{\phi}(\fc_0\fc_1\fa, \fa, \fd, T)
    = \mathcal{N}_{\phi}(\lambda \mathfrak{r}\fc_1 \fa, \fa, \fd, T)
    = \mathcal{N}_{\phi}(\mathfrak{r}\fc_1 \fa, \fa, \fd, N((\lambda))^{-e(f)} T).
\end{align*}
At this point, it is conceivable that the second argument in the final expression might need to be $\lambda^{-1}\fa$, leading to $\mathcal{N}_{\phi}(\mathfrak{r}\fc_1 \fa, \lambda^{-1}\fa, \fd, N((\lambda))^{-e(f)} T).$ 
This adjustment seems consistent with the proof of Proposition \ref{prop: V b Omega F}.
\end{remark}

\subsection{Proof}

A projective local condition $\Omega_{\fp} \subset \cP(u)(K_{\fp})$ is said to be non-trivial if $\Omega_{\fp} \neq \cP(u)(K_{\fp})$, and $\Omega = \lcrc{\Omega_{\fp}}_{\fp \in S}$ is called non-trivial if every $\Omega_{\fp}$ is non-trivial.


\begin{theorem} \label{thm: phi411 general}
Let $f : \cP(u) \to \cP(w)$ be a morphism that has finite defect and satisfies (\ref{eqn: widehat w condition}), and let
$\fp$ be a prime of $K$ not in $S_f$. Let $a_{\fp, j} \in \cO_{K, \fp}$ satisfying $(a_{\fp,0},\cdots,a_{\fp,n}) \neq (0,\cdots,0)$, $r_{\fp, j} \geq 0$, and $\Omega = \lcrc{\Omega_{\fp}}$ a non-trivial irreducible local condition at one prime $\fp$ such that
\begin{align*}
    \Omega_{\fp, 0}^{\aff} = \prod_{j=0}^n \lbrb{ a_{\fp, j} + \fp^{r_{\fp, j}} \cO_{K, \fp} }.
\end{align*}
Then there is an explicit constant $\kappa$ depending on $K$ and $f$ satisfying
\begin{align*}
&\#\lcrc{x \in \cP(u)(K) : H_{w, K}(f(x)) \leq B, x \in \Omega_S} \\
&= \kappa m_{\fp}\left(\Omega_{\fp}^{\aff} \cap \cO_{K, \fp}^{n+1}\right) B^{\frac{|u|}{e(f)}} 
+
O\lbrb{  \lbrb{q+q^{u_{\max} + 1 + \frac{u_{\min}}{d} - |u|} }  \frac{\displaystyle \max_{0\leq j\leq n}\lcrc{    q^{r_{\fp, j}} }}{\prod_{j=0}^n q^{r_{\fp, j}}} 
\frac{q^{|u| - u_{\max}}}{q^{|u| - u_{\max}} -1} B^{\frac{d|u| - u_{\min}}{de(f)}} \log B}.
\end{align*}
\end{theorem}
\begin{proof}
By (\ref{eqn: Omega decomp Omega_k disj}), a point $x$ in $\cO_K^{n+1}$ satisfies $x \in \Omega_S^{\aff}$ if and only if
\begin{align*}
    i_{\fp}(x) \in \bigsqcup_{k \geq 0} \bigsqcup_{\zeta \in \mu_u(K_{\fp})} (\zeta*_u\Omega_{\fp, k}^{\aff}).
\end{align*}
In the proof, we identify elements of $K^{n+1}\setminus \lcrc{0}, V^{\fa, \fd}/\cO_K^\times$, and $\cP(u)(K)$ since the local condition $\bigsqcup_{\zeta \in \mu_u(K_{\fp})} (\zeta*_u\Omega_{\fp, k}^{\aff})$ is stable under $u$-weighted $\cO_{K, \fp}^\times$-action.
Let $\lcrc{[\fa_i]}_{i=1}^h$ be a set of representatives of the class group of $K$.
We can choose the representatives $\fa_i$ to be integral and $\fp \nmid \fa_i$.
There is a natural partition 
\begin{align*}
    &\lcrc{x \in \cP(u)(K) : H_{w}(f(x)) \leq B , x \in \Omega_S} \\
    &=\bigsqcup_{i=1}^h
    \lcrc{x \in \cP(u)(K) : H_{w}(f(x)) \leq B, [\fI_{u}(x)] = [\fa_i], x \in \Omega_S }.
\end{align*}
The canonical relation between the projective space and its affine cone gives a bijection between $i$-th part of the right-hand side and 
\begin{align*}
    \lcrc{x \in (K^{n+1} \setminus \lcrc{0})/\cO_{K}^\times : 
    H_{w}(f(x)) \leq B,
    \fI_{u}(x) = \fa_i, x \in \Omega_S^{\aff} }.
\end{align*}
We denote the cardinality of this set by $M(\fa_i, \Omega,  B)$, so that
\begin{align*}
    \# \lcrc{x \in \cP(u)(K) : H_{w}(f(x)) \leq B, x \in \Omega_S} = \sum_{i=1}^h M(\fa_i, \Omega,  B).
\end{align*}
Since
\begin{align*}
     H_{w}(f(x)) = \frac{H_{w, \infty}(f(x))}{N_{K/\bQ}(\fI_w(f(x)))}
     = \frac{H_{w, \infty}(f(x))}{N_{K/\bQ}(\delta_f(x)\fI_u(x)^{e(f)})},
\end{align*}
we have
\begin{align*}
    M(\fa_i, \Omega, B) 
    &= \sum_{\fd \in \fD_{f}} \#\lcrc{x \in (K^{n+1} \setminus \lcrc{0})/\cO_{K}^\times : H_{w}(f(x)) \leq B,
    \fI_{u}(x) = \fa_i, \delta_f(x) = \fd,  x \in \Omega_S^{\aff} } \\
    &= \sum_{\fd \in \fD_{f}} \#\lcrc{x \in (K^{n+1} \setminus \lcrc{0})/\cO_{K}^\times : \frac{H_{w, \infty}(f(x))}{N_{K/\bQ}(\fd \fa_i^{e(f)})} \leq B,
    \fI_{u}(x) = \fa_i, \delta_f(x) = \fd,  x \in \Omega_S^{\aff} }.
\end{align*}
On the other hand, 
\begin{align*}
    M(\fa, \fa, \fd, \Omega, B) = \# \lcrc{x \in (K^{n+1} \setminus \lcrc{0}) /\cO_{K}^\times : H_{w, \infty}(f(x)) \leq B, \fI_{u}(x) = \fa, \delta_f(x) = \fd, x \in \Omega_S^{\aff} }
\end{align*}
by definition of $M$ and Lemma \ref{lem: concrete fmfa} (i). 
By (\ref{eqn: M'=sum M}), (\ref{eqn: M''= sum M}) and the M\"obius inversion formula, 
\begin{align*}
    M(\fb, \fa, \fd, \Omega, B) 
    = \sum_{\substack{\fd' \subset \cO_K \\ \fd \fd' \in \fD_f}} \mu_K(\fd') M'(\fb, \fa, \fd'\fd, \Omega, B)
    = \sum_{\substack{\fd' \subset \cO_K \\ \fd \fd' \in \fD_f}} \mu_K(\fd') \sum_{\fc \subset \cO_K} \mu_K(\fc) M''(\fb\fc, \fa, \fd'\fd, \Omega, B)
\end{align*}
where 
\begin{align*}
    \mu_K(\fa) \vcentcolon = \left\{ 
    \begin{array}{ll}
    (-1)^r     & \textrm{ if } \fa = \prod_{i=1}^r \fp_i,  \\
    0     & \textrm{ if $\fa$ is not square-free}.
    \end{array} \right.
\end{align*}
Therefore, 
\begin{align}
    \sum_{i=1}^h M(\fa_i, \Omega, B)  &=  \sum_{i=1}^h \sum_{\fd \in \fD_f} M\left(\fa_i, \fa_i, \fd, \Omega, BN_{K/\bQ}(\fa_i^{e(f)}\fd) \right) \nonumber \\
    &= \sum_{i=1}^h \sum_{\fd \in \fD_f} \sum_{\substack{\fd' \subset \cO_K \\ \fd\fd' \in \fD_{f}}} \mu_K(\fd') 
    \sum_{\fc \subset \cO_K} \mu_K(\fc)
    M''\left( \fa_i \fc, \fa_i, \fd'\fd, \Omega, BN_{K/\bQ}(\fa_i^{e(f)}\fd\fd') \right).
    \label{eqn: mu M' form}
\end{align}
From now on, we use 
\begin{align*}
    \fc_0 = \fc_0(\fc) = \fc_0(\fa_i^{-1} \cdot \fa_i \fc), \qquad \text{ and } \qquad
    \fc_1 = \fc_1(\fc) = \fc_1(\fa_i^{-1} \cdot \fa_i \fc)
\end{align*}
in the proof.
By Lemma \ref{lem: M' step1}, we have 
\begin{align}
    & M''\left(\fa_i\fc, \fa_i, \fd'\fd, \Omega, B N_{K/\bQ}(\fa_i^{e(f)}\fd\fd')\right)\nonumber \\
    & = \frac{1}{\# \mu_{u}(K)} \sum_{\zeta \in \mu_u(K_{\fp})} \sum_{t=0}^{\infty} \#\lbrb{V^{\fa_i, \fd'\fd} \cap (\fa_i\fc)^{u} \cap i_{\fp}^{-1}(\zeta*_u\Omega_{\fp, t}^{\aff}) \cap \cF(B N_{K/\bQ}(\fa_i^{e(f)}\fd\fd'))}. \nonumber
\end{align}
We can apply Proposition \ref{prop: V b Omega F} to each summand in the expression above, with $\Omega_{\fp}^a = \Omega_{\fp, t}^{\aff}$.
For simplicity, we first consider
\begin{align} \label{eqn: interim M''}
    \sum_{\zeta \in \mu_u(K_{\fp})} \sum_{t=0}^{\infty} \#\lbrb{V^{\fa_i, \fd} \cap (\fa_i\fc)^{u} \cap i_{\fp}^{-1}(\zeta*_u\Omega_{\fp, t}^{\aff}) \cap \cF(B )}.
\end{align}
In this case, let $\Lambda =  \fm^{\fa_i, \fd}_{\fc_1}(B) \cap (\fa_i\fc_0)^{u}$.
Since $\Omega_{\fp, t}^{\aff}$ is a box-type set, we have
\begin{align*}
    \epsilon(\zeta*_u\Omega_{\fp, t}^{\aff}) 
    = \max_{0 \leq j \leq n }\lcrc{\omega_{\fp, j}^{-1}(\Omega) q_{\fp}^{tu_j}} \frac{m_{\fp}((\zeta*_u\Omega_{\fp, t}^{\aff}) \cap \Lambda_{\fp})}{m_{\fp}(\Lambda_{\fp})}
    = \max_{0 \leq j \leq n }\lcrc{ q_{\fp}^{r_{\fp, j} + tu_j}} \frac{m_{\fp}((\zeta*_u\Omega_{\fp, t}^{\aff}) \cap \Lambda_{\fp})}{m_{\fp}(\Lambda_{\fp})},
\end{align*}
and 
\begin{align*}
    &\#\lbrb{V^{\fa_i, \fd} \cap (\fa_i\fc)^{u} \cap i_{\fp}^{-1}(\zeta*_u\Omega_{\fp, t}^{\aff}) \cap \cF(B )} 
    =C(\fa_i, \fa_i\fc, \fd, f)m_{\infty}(\cF(1))
    \frac{m_{\fp}((\zeta*_u\Omega_{\fp,t}^{\aff}) \cap \Lambda_{\fp} ) }{m_{\fp}(\Lambda_{\fp})}
    B^{\frac{|u|}{e(f)}} \\ 
    &\qquad +O\lbrb{ \max_{0\leq j \leq n}\lcrc{ q_{\fp}^{r_{\fp, j} +tu_j} } \frac{m_{\fp}((\zeta*_u\Omega_{\fp,t}^{\aff}) \cap \Lambda_{\fp} ) }{m_{\fp}(\Lambda_{\fp})}  \frac{B^{\frac{d|u| - u_{\min}}{de(f)}}}{N_{K/\bQ}(\fc_0)^{\frac{d|u|-u_{\min}}{d}}} }.
\end{align*}
Since $\fp \not\in S_f$, we have $\fm_{\fp}^{\fa_i, \fd}(B) = (\fa_i)_{\fp}^u$ by definition (\ref{eqn: def fm_fp^fa^fd V_fp^fa^fd}), which implies that $\Lambda_{\fp} = (\fa_i\fc_0)_{\fp}^{u}$.
Hence, $\Lambda_{\fp}$ is a product of $\fp^{u_j\ord_{\fp}(\fa_i\fc_0)}\cO_{K, \fp}$. 
On the other hand, we have
\begin{align*}
    \zeta*_u\Omega_{\fp, t}^{\aff}
    =\prod_{j=0}^n \lbrb{\zeta^{u_j}\pi_{\fp}^{tu_j}a_{\fp, j} + \fp^{r_{\fp, j}+tu_j}\cO_{K, \fp}}.
\end{align*}
By Lemma \ref{lem: pm cap xpn},
\begin{align*}
(\zeta*_u\Omega_{\fp, t}^{\aff}) \cap (\fa_i\fc_0)_{\fp}^{u}
= \left\{
\begin{array}{ll}
    \emptyset & \textrm{if } t < \ord_{\fp}(\fa\fc_0),  \\
    (\zeta*_u\Omega_{\fp, t}^{\aff}) & \textrm{if } t \geq \ord_{\fp}(\fa\fc_0),
\end{array}
\right.
\end{align*}
because there is $j$ satisfying $a_{\fp, j} \neq 0$,
Therefore if $t\geq \ord_{\fp}(\fa_i\fc_0)$,
\begin{align*}
    \frac{m_{\fp}((\zeta*_u\Omega_{\fp, t}^{\aff}) \cap \Lambda_{\fp} )}{m_{\fp}(\Lambda_{\fp})} 
    = \frac{m_{\fp}(\zeta*_u\Omega_{\fp, t}^{\aff})}{m_{\fp}(\Lambda_{\fp})} 
    = \left. \lbrb{\prod_{j=0}^n q_{\fp}^{-(r_{\fp, j} + tu_j)} } \middle/ 
    \lbrb{\prod_{j=0}^n q_{\fp}^{-u_j \ord_{\fp}(\fa_i\fc_0)}} \right.
    = m_{\fp}(\Omega_{\fp, t-\ord_{\fp}(\fa_i\fc_0)}^{\aff}).
\end{align*}
We note that $m_{\fp}(\Omega_{\fp, t}^{\aff}) = m_{\fp}(\zeta*_u\Omega_{\fp, t}^{\aff})$ for any $\zeta \in \mu(K_{\fp})$ and the left-hand side of the equality does not depend on $B$ by the same reason mentioned in Remark \ref{rmk: volume not dep on B}.

By definition, $\Omega_{\fp, t}^{\aff}$'s are disjoint.
For arbitrary integral ideal $\fc$ that determines $\fc_i = \fc_i(\fc)$ and $\Lambda =  \fm^{\fa_i, \fd}_{\fc_1}(B) \cap (\fa_i\fc_0)^{u}$,
\begin{align*}
    \sum_{\zeta \in \mu_u(K_{\fp})}\sum_{t=0}^{\infty} \frac{m_{\fp}((\zeta*_u\Omega_{\fp, t}^{\aff}) \cap \Lambda_{\fp} )}{m_{\fp}(\Lambda_{\fp})} 
    &=\sum_{\zeta \in \mu_u(K_{\fp})}\sum_{t = \ord_{\fp}(\fa_i\fc_0)}^{\infty} m_{\fp}(\Omega_{\fp, t-\ord_{\fp}(\fa_i\fc_0)}^{\aff}) \\
    &= \sum_{\zeta \in \mu_u(K_{\fp})} \sum_{t=0}^{\infty} m_{\fp}(\zeta*_u\Omega_{\fp, t}^{\aff}) \\
    &= m_{\fp}\left(\Omega_{\fp}^{\aff} \cap \cO_{K, \fp}^{n+1}\right),
\end{align*}
by (\ref{eqn: Omega decomp Omega_k disj}). 
Hence, the main term of (\ref{eqn: interim M''}) is
\begin{align*}
    C(\fa_i, \fa_i\fc, \fd,  f) m_{\infty}(\cF(1))
    m_{\fp}\left(\Omega_{\fp}^{\aff} \cap \cO_{K, \fp}^{n+1}\right) B^{\frac{|u|}{e(f)}}.
\end{align*}
There is a naive bound $\#\mu_u(K_{\fp}) \leq q_{\fp}$ that can be used to bound the summation over $\zeta \in \mu_u(K_{\fp})$.
So the error term of (\ref{eqn: interim M''}) is
\begin{align*}
    & \leq  q_{\fp} \sum_{t=\ord_{\fp}(\fa_i\fc_0)}^{\infty}
    \max_{0\leq j\leq n}\lcrc{    q_{\fp}^{r_{\fp, j} + tu_j} } 
    \lbrb{\prod_{j=0}^n q_{\fp}^{-r_{\fp, j}-(t - \ord_{\fp}(\fa_i\fc_0))u_j}  } 
    \frac{B^{\frac{d|u| - u_{\min}}{de(f)}}}{N_{K/\bQ}(\fc_0)^{\frac{d|u|-u_{\min}}{d}}}  \\
    & \leq q_{\fp} \sum_{t=\ord_{\fp}(\fa_i\fc_0)}^{\infty}
    q_{\fp}^{tu_{\max}} \max_{0\leq j\leq n}\lcrc{    q_{\fp}^{r_{\fp, j}} }  
    \lbrb{ q_{\fp}^{-(t - \ord_{\fp}(\fa_i\fc_0))|u|} \prod_{j=0}^n q_{\fp}^{-r_{\fp, j}} } 
    \frac{B^{\frac{d|u| - u_{\min}}{de(f)}}}{N_{K/\bQ}(\fc_0)^{\frac{d|u|-u_{\min}}{d}}}  \\
    & = q_{\fp}
    \frac{\displaystyle \max_{0\leq j\leq n}\lcrc{    q_{\fp}^{r_{\fp, j}} }}{\prod_{j=0}^n q_{\fp}^{r_{\fp, j}}}
    \frac{B^{\frac{d|u| - u_{\min}}{de(f)}}}{N_{K/\bQ}(\fc_0)^{\frac{d|u|-u_{\min}}{d}}}
    \sum_{t=0}^{\infty}
    q_{\fp}^{-t|u| + (t + \ord_{\fp}(\fa_i\fc_0)) u_{\max}} \\
    & =
    \frac{\displaystyle \max_{0\leq j\leq n}\lcrc{    q_{\fp}^{r_{\fp, j}} }}{\prod_{j=0}^n q_{\fp}^{r_{\fp, j}}}
    \frac{B^{\frac{d|u| - u_{\min}}{de(f)}}}{N_{K/\bQ}(\fc_0)^{\frac{d|u|-u_{\min}}{d}}}
    q_{\fp}^{\ord_{\fp}(\fa_i\fc_0)u_{\max} + 1}
    \frac{q_{\fp}^{|u| - u_{\max}}}{q_{\fp}^{|u| - u_{\max}} -1}.
\end{align*}
Altogether, we have
\begin{align*}
    (\ref{eqn: interim M''}) &= C(\fa_i, \fa_i\fc, \fd, f) m_{\infty}(\cF(1))
     m_{\fp}\left(\Omega_{\fp}^{\aff} \cap \cO_{K, \fp}^{n+1}\right) B^{\frac{|u|}{e(f)}} \\
    & \qquad + O\lbrb{ \frac{\displaystyle \max_{0\leq j\leq n}\lcrc{    q^{r_{\fp, j}} }}{\prod_{j=0}^n q^{r_{\fp, j}}}
    \frac{B^{\frac{d|u| - u_{\min}}{d e(f)}}}{N_{K/\bQ}(\fc_0)^{\frac{d|u|-u_{\min}}{d}}}
    q_{\fp}^{\ord_{\fp}(\fa_i\fc_0)u_{\max}+1}
    \frac{q_{\fp}^{|u| - u_{\max}}}{q_{\fp}^{|u| - u_{\max}} -1}}.
\end{align*}
By plugging this into (\ref{eqn: mu M' form}), the main term is
\begin{align*}
    &\frac{m_{\infty}(\cF(1))}{\# \mu_{u}(K)} \sum_{i=1}^h \sum_{\fd \in \fD_f} \sum_{\substack{\fd' \subset \cO_K \\ \fd\fd' \in \fD_{f}}} \sum_{\fc \subset \cO_K} \mu_K(\fc) \mu_K(\fd')
     C(\fa_i, \fa_i\fc, \fd\fd',  f)
     m_{\fp}\left(\Omega_{\fp}^{\aff} \cap \cO_{K, \fp}^{n+1}\right) (BN_{K/\bQ}(\fa_i^{e(f)}\fd \fd'))^{\frac{|u|}{e(f)}}= \\
    &\frac{m_{\infty}(\cF(1)) m_{\fp}\left(\Omega_{\fp}^{\aff} \cap \cO_{K, \fp}^{n+1}\right) B^{\frac{|u|}{e(f)}}}{\# \mu_{u}(K)} 
    \sum_{i=1}^h \sum_{\fd \in \fD_f} \sum_{\substack{\fd' \subset \cO_K \\ \fd\fd' \in \fD_{f}}} \sum_{\fc \subset \cO_K} \mu_K(\fc) \mu_K(\fd') C(\fa_i, \fa_i\fc, \fd\fd',  f)
    N_{K/\bQ}(\fa_i^{e(f)}\fd \fd')^{\frac{|u|}{e(f)}}.
\end{align*}

Let 
\begin{align*}
    \kappa(\fa, \fd) \vcentcolon =  \sum_{\substack{\fd' \subset \cO_K \\ \fd\fd' \in \fD_{f}}} \sum_{\fc \subset \cO_K} \mu_K(\fc) \mu_K(\fd') C(\fa, \fa\fc, \fd \fd', f)
    N_{K/\bQ}(\fa^{e(f)}\fd \fd')^{\frac{|u|}{e(f)}}.
\end{align*}
We claim that $\kappa(\fa, \fd)$ is finite.
Since the sum over $\fd' \subset \cO_K$ such that $\fd \fd' \in \fD_f$ is finite, it suffices to show that
\begin{align*}
    \sum_{\fc \subset \cO_K} \mu_K(\fc) C(\fa, \fa\fc, \fd, f)
\end{align*}
is finite. 
By the construction of $C(\fa, \fb, \fd, f)$ in Lemma \ref{lem: fS lattice vol quotient}, there are constants $C_{\fp}(\fa, \fa\fc, \fd, f)$ satisfying
\begin{align*}
    C(\fa, \fa\fc, \fd, f)
    = \prod_{\fp} C_{\fp}(\fa, \fa\fc, \fd, f)
\end{align*}
and $C_{\fp}(\fa, \fa\fc, \fd, f) = q_{\fp}^{-\ord_{\fp}(\fa\fc)|u|}$ if $\fp \nmid \fd$.
Let $\widetilde{S(\fd)}$ (resp. $\widehat{S(\fd)}$) be the set of integral ideals only divided (resp. not divided) by prime ideals dividing $\fd$.
Then, we have 
\begin{align*}
    \sum_{\fc \subset \cO_K} \mu_K(\fc) C(\fa, \fa\fc, \fd, f)
    = \lbrb{\sum_{\fc_0 \in \widetilde{S(\fd)}} \frac{\mu_K(\fc_0)}{ N(\fa\fc_0)^{|u|} }  }
    \lbrb{\sum_{\fc_1 \in \widehat{S(\fd)}} \mu_K(\fc_1) C_{\fp}(\fa, \fa\fc, \fd, f) }.
\end{align*}
The summation over $\widehat{S(\fd)}$ is finite, since it contains only finitely many square-free ideals.
The summation over $\widetilde{S(\fd)}$ is also finite since
\begin{align} \label{eqn: mu/norm = 1/zeta}
    \sum_{\fc \in \cO_K} \frac{\mu_K(\fc)}{N_{K/\bQ}(\fa\fc)^{|u|}}
    = \frac{1}{N_{K/\bQ}(\fa)^{|u|}} \frac{1}{\zeta_K(|u|)} 
\end{align}
and there are only finitely many prime ideals not in $\widetilde{S(\fd)}$.
Hence, both $\kappa(\fa_i, \fd)$ and
\begin{align*}
    \kappa \vcentcolon = 
    \frac{m_{\infty}(\cF(1)) }{\# \mu_{u}(K)} 
    \sum_{i=1}^h \sum_{\fd \in \cD_f} \kappa(\fa_i, \fd)
\end{align*}
are finite.

Substituting the error term from (\ref{eqn: interim M''}) into (\ref{eqn: mu M' form}), we obtain the following error term:
\begin{align*}
&\sum_{i=1}^h \sum_{\fd \in \fD_f} \sum_{\substack{\fd' \subset \cO_K \\ \fd\fd' \in \fD_{f}}} \sum_{\fc \subset \cO_K} \mu_K(\fc\fd') 
    \frac{\displaystyle \max_{0\leq j\leq n} \lcrc{    q_{\fp}^{r_{\fp, j}} }}{\prod_{j=0}^n q_{\fp}^{r_{\fp, j}}}
    \frac{N_{K/\bQ}(\fa_i^{e(f)}\fd \fd')^{\frac{d|u|-u_{\min}}{de(f)}} B^{\frac{d|u| - u_{\min}}{de(f)}}}{N_{K/\bQ}(\fc_0)^{\frac{d|u|-u_{\min}}{d}}}
    q_{\fp}^{\ord_{\fp}(\fa_i\fc_0)u_{\max}+1}
    \frac{q_{\fp}^{|u| - u_{\max}}}{q_{\fp}^{|u| - u_{\max}} -1} \\
&=\frac{\displaystyle \max_{0\leq j\leq n}\lcrc{    q_{\fp}^{r_{\fp, j}} }}{\prod_{j=0}^n q_{\fp}^{r_{\fp, j}}} 
    \frac{q_{\fp}^{|u| - u_{\max}}}{q_{\fp}^{|u| - u_{\max}} -1} B^{\frac{d|u| - u_{\min}}{de(f)}}
    \sum_{i=1}^h \sum_{\fd \in \fD_f} \sum_{\substack{\fd' \subset \cO_K \\ \fd\fd' \in \fD_{f}}} \sum_{\fc \subset \cO_K} \mu_K(\fc\fd') 
    q_{\fp}^{\ord_{\fp}(\fa_i\fc_0)u_{\max}+1} 
    \frac{N_{K/\bQ}(\fa_i^{e(f)}\fd\fd')^{\frac{d|u|-u_{\min}}{de(f)}}}{N_{K/\bQ}(\fc_0)^{\frac{d|u|-u_{\min}}{d}}}.
\end{align*}
The inner summation can be arranged by
\begin{align*}
    \sum_{i=1}^h N_{K/\bQ}(\fa_i)^{\frac{d|u|-u_{\min}}{d}} 
    \sum_{\fd \in \fD_f} \sum_{\substack{\fd' \subset \cO_K \\ \fd\fd' \in \fD_{f}}} 
    \mu_K(\fd')N_{K/\bQ}(\fd\fd')^{\frac{d|u|-u_{\min}}{de(f)}}
    \sum_{\fc \subset \cO_K} 
    \frac{ \mu_K(\fc) q_{\fp}^{\ord_{\fp}(\fa_i\fc_0)u_{\max}+1}  }{N_{K/\bQ}(\fc_0)^{\frac{d|u|-u_{\min}}{d}}}.
\end{align*}
The first three sums are over finite sets, so they are negligible.
Since we have assumed that $\fp \nmid \fa_i$, the last sum can be replaced by
\begin{align*}
    \sum_{\fc \subset \cO_K} 
    \frac{ \mu_K(\fc) q_{\fp}^{\ord_{\fp}(\fc_0)u_{\max}+1}   }{N_{K/\bQ}(\fc_0)^{\frac{d|u|-u_{\min}}{d}}}
    = \lbrb{\sum_{\fc_1 \in \widehat{\fD_f}} \mu_K(\fc_1) }
    \lbrb{\sum_{\fc_0 \in \widetilde{\fD_f}} \frac{ \mu_K(\fc_0) q_{\fp}^{\ord_{\fp}(\fc_0)u_{\max}+1} }{N_{K/\bQ}(\fc_0)^{\frac{d|u|-u_{\min}}{d}}} }.
\end{align*}
The first sum over $\fc_1$ is finite since $\widehat{\fD_f}$ is finite and does not depend on $q_{\fp}$, and the second one is 
\begin{align*}
    \sum_{\fc_0 \in \widetilde{\fD_f}} \frac{ \mu_K(\fc_0) q_{\fp}^{\ord_{\fp}(\fc_0)u_{\max}+1} }{N_{K/\bQ}(\fc_0)^{\frac{d|u|-u_{\min}}{d}}}
    = \sum_{\substack{\fc_0 \in \widetilde{\fD_f} \\ \fp \nmid \fc_0 }} \frac{ \mu_K(\fc_0)q_{\fp} }{N_{K/\bQ}(\fc_0)^{\frac{d|u|-u_{\min}}{d}}}
    - \frac{q_{\fp}^{u_{\max}+1}}{q_{\fp}^{|u| - \frac{u_{\min}}{d}}} \sum_{\substack{\fc_0 \in \widetilde{\fD_f} \\ \fp \nmid \fc_0 }} \frac{ \mu_K(\fc_0) }{N_{K/\bQ}( \fc_0)^{\frac{d|u|-u_{\min}}{d}}}.
\end{align*}
By (\ref{eqn: mu/norm = 1/zeta}), the same argument gives 
\begin{align*}
    \sum_{\substack{\fc_0 \in \widetilde{\fD_f} \\ \fp \nmid \fc_0 }} \frac{ \mu_K(\fc_0) }{N_{K/\bQ}( \fc_0)^{\frac{d|u|-u_{\min}}{d}}} = O(1).
\end{align*}
Hence, the second sum over $\fc_0 \in \widetilde{\fD_f}$ is $O(q_{\fp}+q_{\fp}^{u_{\max} + 1 + \frac{u_{\min}}{d} - |u|})$. Consequently, the error term is
\begin{align*}
    O\lbrb{\frac{\displaystyle \max_{0\leq j\leq n}\lcrc{    q_{\fp}^{r_{\fp, j}} }}{\prod_{j=0}^n q_{\fp}^{r_{\fp, j}}} 
    \frac{q_{\fp}^{|u| - u_{\max}}}{q_{\fp}^{|u| - u_{\max}} -1} B^{\frac{d|u| - u_{\min}}{de(f)}}
    \lbrb{q_{\fp}+q_{\fp}^{u_{\max} +1 + \frac{u_{\min}}{d} - |u|}  }}
\end{align*}
and the implied constant depends on $f, K, u, w$.
To deal with the case $K = \bQ$, we also have an additional error term $\log B$ on the error term.
\end{proof}

To deduce Theorem \ref{prop:phi411} from Theorem \ref{thm: phi411 general}, we need to compute the number of preimages of the forgetful functor $\phi_{\Gamma}: \cX_{\Gamma}(K) \to \cX(K)$.
This is essentially achieved by \cite[Lemma 3.1.8]{CKV}, which we now introduce.
For a subgroup $G \leq \GL_2(\bZ/N\bZ)$ and the natural projection $\pi_N : \SL_2(\bZ) \to \SL_2(\bZ/N\bZ)$, we define
\begin{align*}
    \Gamma_G \vcentcolon= \pi_N^{-1}(G \cap \SL_2(\bZ/N\bZ)).
\end{align*}
For example, when
\begin{align*}
    G = \lcrc{\begin{pmatrix}
        1 & b \\ 0 & d
    \end{pmatrix} \in \GL_2(\bZ/N\bZ) : b, d \in \bZ/N\bZ },
\end{align*}
we have $\Gamma_G = \Gamma_1(N)$. 
Similarly, $\Gamma_{G} = \Gamma(N)$ when $G = \lcrc{I} \leq \GL_2(\bZ/N\bZ)$.
For the congruence group $\Gamma_1(M, N) = \Gamma(M) \cap \Gamma_1(MN)$, we may take $G$ to be the intersection of the preimages of the subgroups corresponding to $\Gamma(M)$ and $\Gamma_1(MN)$ in $\mathrm{GL}_2(\mathbb{Z}/MN\mathbb{Z})$.
We define
\begin{align*}
    r(G) \vcentcolon = \left\{
    \begin{array}{llll}
    {[N_{\GL_2(\bZ/N\bZ)}(G) : G ]} & \textrm{ if } -I \in \Gamma,\\
    {[N_{\GL_2(\bZ/N\bZ)}(G) : G ]}/2 & \textrm{ otherwise.}
    \end{array}
    \right.
\end{align*}
We note that $\# \phi^{-1}_{\Gamma}(E)$ is denoted by $r_G(E)$ in \cite{CKV}.

From now on, we write $\Gamma$ for $\Gamma_G$ and $\Gamma'$ for $\Gamma_{G'}$.
\begin{lemma} \label{lem: CKV preimages}
Let $E$ be an elliptic curve over $K$ satisfying $j(E) \neq 0, 1728$.  Then,\\
(i) If $\# \phi^{-1}_{\Gamma}(E) \geq 1$, then $\# \phi_{\Gamma}^{-1}(E) \geq r(G)$. \\
(ii) If $\# \phi^{-1}_{\Gamma}(E) > r(G)$, then there exists $G' \leq G$ such that $\# \phi^{-1}_{\Gamma'}(E) \geq 1$.
\end{lemma}
\begin{proof}
This is proved in \cite[Lemma 3.1.8]{CKV} for elliptic curves over $\bQ$, and the same argument applies to elliptic curves over any number field $K$.
\end{proof}

\begin{lemma} \label{lem: Serre}
Let $X$ be an irreducible algebraic curve over $K$ with a morphism $\pi : X \to \bP^1$, and assume that $X$ has genus at least 1.
For any $\epsilon > 0$,
\begin{align*}
    \# \lcrc{t \in \bP^1(K) : t \in \pi(X(K)), H(t) \leq B} = O(B^{\epsilon}).
\end{align*}
\end{lemma}
\begin{proof}
It is a part of \cite[Theorem, p.133]{Ser97}. 
In fact, it is $O(\log B^{r(X/K)})$ when $X/K$ is an elliptic curve whose algebraic rank is $r(X/K)$, and $O(1)$ if the genus of $X$ is at least 2, by Faltings’s theorem.
\end{proof}

\begin{lemma} \label{lem: Height compare}
Let $j : \cP(4, 6)(K) \to \bP^1(K)$ be a morphism defined by
\begin{align*}
    j([A, B]) \vcentcolon = [2^8\cdot 3^3\cdot A^3,4A^3+27B^2].
\end{align*}
Then there is a constant $C$ depending only on $K$ such that $H_{(1, 1)}(j([A, B])) \leq CH_{(4, 6)}([A, B])^{12}$.
\end{lemma}
\begin{proof}
Let $S$ be a finite set of finite places such that $\cO_{K, S}$ is a principal ideal domain.
Then for integral $A$ and $B$, 
\begin{align*}
    H_{(1, 1), K}([A, B]) &= \prod_{v \in S} |(A, B)|_{(1, 1), v} \times \prod_{v \in M_{K, \infty}} |(A, B)|_{(1, 1), v} \leq \prod_{v \in M_{K, \infty}} |(A, B)|_{(1, 1), v}\\
    &= \prod_{v \in M_{K, \infty}} \max \lcrc{|A|_v, |B|_v}.
\end{align*}
In the proof of Lemma \ref{lem: Height inf}, we showed that the same set $S$ satisfies
\begin{align*} 
    H_{(4, 6), K}([A, B]) 
    & \geq \prod_{v \in S} q_v^{-n_v} \prod_{v \in M_{K, \infty}} \max \lcrc{|A|_v^{\frac{1}{4}}, |B|_v^{\frac{1}{6}}}
\end{align*}
for some $n_v$, depending only on $K$.
Therefore,
\begin{align*}
    H_{(1, 1), K}(j([A, B])) &= H_{(1, 1), K}([ 1728 \cdot 4 A^3, 4A^3 + 27B^2] )\\
    &\leq 1728 H_{(1, 1), K}( [4 A^3, 4A^3 + 27B^2] ) 
    \leq 1728 \prod_{v \in M_{K, \infty}} \max \lcrc{|4A^3|_v, |4A^3 + 27B^2|_v} \\
    &\leq 1728 \prod_{v \in M_{K, \infty}} \lbrb{|4A^3|_v + |27B^2|_v}
    \leq 1728 \cdot 4 \cdot 27 \cdot 2 \prod_{v \in M_{K, \infty}} \max \lcrc{|A^3|_v, |B^2|_v}  \\
    &= 1728 \cdot 4 \cdot 27 \cdot 2 \prod_{v \in M_{K, \infty}}\lbrb{  \max \lcrc{|A|_v^{\frac{1}{4}}, |B|_v^{\frac{1}{6}} } }^{12} \\
    &\leq 1728 \cdot 4 \cdot 27 \cdot 2 \prod_{v \in S} q_v^{12 n_v } H_{(4, 6), K}([A, B])^{12}.
\end{align*}
So we obtain the result.
\end{proof}

In (\ref{eqn:def eGam}), we have defined
\begin{align*}
    e(\phi_{\Gamma}) = e(\Gamma) \vcentcolon = \frac{u_0u_1}{24}[\SL_2(\bZ) : \Gamma]
\end{align*}
where $\cX_{\Gamma} \cong \cP(u)$, following \cite[Lemma 4.1]{BN}.
The main idea of the following proposition comes from \cite[Step 4, proof of Theorem 3.3.1]{CKV}.

\begin{proposition} \label{prop: preimage degree}
Let $\Gamma$ be a congruence subgroup of level $N$ such that $\mathcal{X}_\Gamma\cong\mathbb{P}^1$ over $\Spec\mathbb{Z}[1/6N]$. Let $\phi_\Gamma:\mathcal{P}(1,1)\rightarrow\mathcal{P}(4,6)$ be the morphism corresponding to forgetting the level structure, and let $\Omega = \lcrc{\Omega_{\fp}}$ be a local condition of $\cP(1, 1)$ at a prime $\fp$ which is not in $S_{\phi_{\Gamma}}$. 
Then,
\begin{align*} 
\begin{aligned}
    &\#\lcrc{y \in \phi_{\Gamma}(\cP(1,1)(K)) : H_{(4, 6)}(y) \leq B, i_{\fp}(y) \in \phi_{\Gamma}(\Omega_{\fp})} \\
    &= r(G) \#\lcrc{x \in \cP(1,1)(K) : H_{(4, 6)}(\phi_{\Gamma}(x)) \leq B, i_{\fp}(x) \in \Omega_{\fp}} + O\lbrb{m_{\fp}\left(\Omega_{\fp}^{\aff} \cap \cO_{K, \fp}^{2}\right)
    \max_{\Sigma < \Gamma} \lbrb{ B^{\frac{2}{e({\Sigma})}}  } },
\end{aligned}
\end{align*}
where the maximum is taken over maximal congruence subgroups $\Sigma$ of $\Gamma$.
\end{proposition}
\begin{proof}
By Lemma \ref{lem: CKV preimages}, if $\#\phi_{\Gamma}^{-1}(E) \neq r(G)$ and $\#\phi_{\Gamma}^{-1}(E) \geq 1$, then the elliptic curve satisfies $j(E) = 0, 1728$ or lies in the image of $\phi_{\Gamma'}$ for $\Gamma' \lneq \Gamma$.
We will show that the contribution from such elliptic curves is $\displaystyle O\lbrb{m_{\fp}\left(\Omega_{\fp}^{\aff} \cap \cO_{K, \fp}^{2}\right)
\max_{\Sigma < \Gamma} \lbrb{ B^{\frac{2}{e({\Sigma})}}  }}$.
Then the result follows.

Given a congruence subgroup $\Sigma\subseteq\mathrm{SL}_2(\mathbb{Z})$ such that $\mathcal{X}_\Sigma$ is representable, consider the following commutative diagram taken from section \ref{sec:Prestack}:
\begin{align*}
    \xymatrix{\cX_{\Sigma}(K) \ar[r]^-{\phi_{\Sigma}} \ar[d]^-\wr & \cX(K) \ar[d]^-{j} \\
    X_{\Sigma}(K) \ar[r]^-{\pi} & \bP^1(K)
    }
\end{align*}
where the left column is a bijection because $\cX_\Sigma$ is representable. Hence, for $t\in\mathbb{P}^1(K)$,
\begin{align*}
    \#\phi_\Sigma^{-1}(\im\phi_\Sigma\cap j^{-1}(t))
    =\# \phi_\Sigma^{-1}(j^{-1}(t))=\#\pi^{-1}(t)\leq\deg\pi
\end{align*}
is finite. 
As a result, elliptic curves with $j$-invariant $0$ or $1728$ are negligible.

Next, define
\begin{align*}
    N_{X,\Sigma}(B)\vcentcolon=\left\{t\in\pi(X_\Sigma(K)):H_{(1,1),K}(t)\leq B\right\}\text{ and } N_{\mathcal{X},\Sigma}(B)\vcentcolon=\left\{y\in\phi_\Sigma(\mathcal{X}_\Sigma(K)):H_{(4,6),K}(y)\leq B\right\}.
\end{align*}
Since we are considering the elliptic curves whose $j$-invariant is not $0$ or $1728$, we consider $E \in \im\phi_{\Gamma'}$ for some $\Gamma' \lneq \Gamma $. 
First, suppose that the genus of $\Gamma'$ is $\geq 1$.
Then, $N_{X, \Gamma'}(B) = O(B^{\epsilon})$ for any $\epsilon >0$ by Lemma \ref{lem: Serre}, and $j(N_{\cX, \Gamma'}(B)) \subset N_{X, \Gamma'}(C B^{12})$ by Lemma \ref{lem: Height compare}.
Therefore, for any $\epsilon > 0$,
\begin{align*}
    \# N_{\cX, \Gamma'}(B) \leq \# N_{X, \Gamma'}(CB^{12}) = O(B^\epsilon).
\end{align*}

Suppose that the genus of $\Gamma'$ is zero.
Let
\begin{align*}
    \Gamma \gneq \Gamma' \gneq \Gamma_1 \gneq \Gamma_2 \gneq  \cdots
\end{align*}
be a chain of congruence subgroups such that there is no congruence subgroup strictly between $\Gamma_i$ and $\Gamma_{i+1}$.
For each chain, there exists $k \in \mathbb{N}$ such that the genus of $\Gamma_k$ is zero, and the genus of any $\Sigma \lneq \Gamma_k$ is $\geq 1$ since there are only finitely many congruence subgroups of genus zero (cf. \cite{CP03}).
By Proposition \ref{representability-(M,N)}, $\cX_{\Gamma_k}$ is representable over $\Spec(\bZ[1/6N])$. 
So we have
\begin{align*}
    &\#\lcrc{y \in \phi_{\Gamma_k}(\cP(1,1)(K)) : H_{w}(y) \leq B, i_{\fp}(y) \in \phi_{\Gamma_k}(\Omega_{\fp})} \\
    &=r(G_k)\#\lcrc{x \in \cP(1,1)(K) : H_{w}(\phi_{\Gamma_k}(x)) \leq B, i_{\fp}(x) \in \Omega_{\fp}} + O(B^\epsilon) 
\end{align*}
where $G_k$ is the group satisfying $\Gamma_{G_k} = \Gamma_k$.
Since $\cX_{\Gamma_k} \cong \bP^1$, $\phi_{\Gamma_k}$ has finite defect by \cite[Corollary 6.5]{BN} and satisfies (\ref{eqn: widehat w condition}) by definition.
Hence, by applying Theorem \ref{thm: phi411 general}, we obtain
\begin{align*}
    \#\lcrc{y \in \phi_{\Gamma_k}(\cP(1,1)(K)) : H_{w}(y) \leq B, i_{\fp}(y) \in \phi_{\Gamma_k}(\Omega_{\fp})}
    =O\lbrb{m_{\fp}\left(\Omega_{\fp}^{\aff} \cap \cO_{K, \fp}^{2}\right)B^{\frac{2}{e({\Gamma_k})}}}.
\end{align*}

Since there is no congruence subgroup between $\Gamma_i$ and $\Gamma_{i+1}$,
any congruence subgroup
$\Sigma \lneq \Gamma_{k-1}$ is either one of $\Gamma_k$ (possibly in a different chain) or has genus $\geq 1$.
By the previous argument and Theorem \ref{thm: phi411 general},
\begin{align*}
    &\#\lcrc{y \in \phi_{\Gamma_{k-1}}(\cP(1,1)(K)) : H_{w}(y) \leq B, i_{\fp}(y) \in \phi_{\Gamma_{k-1}}(\Omega_{\fp})} \\
    &=r(G_{k-1}) \#\lcrc{x \in \cP(1,1)(K) : H_{w}(\phi_{\Gamma_{k-1}}(x)) \leq B, i_{\fp}(x) \in \Omega_{\fp}} \\
    & \qquad + O\lbrb{m_{\fp}\left(\Omega_{\fp}^{\aff} \cap \cO_{K, \fp}^{2}\right) \max_{\Gamma_k < \Gamma_{k-1}}\lbrb{ B^{\frac{2}{e({\Gamma_k})}}} }.
\end{align*}
Here, the maximum is taken over all possible congruence subgroups $\Gamma_k$ in each chain.
Since $e({\Gamma_i}) \nleq e({\Gamma_{i-1}})$ whenever $\Gamma_{i-1} \gneq \Gamma_i$, the result follows by induction.
\end{proof}

The following gives Theorem \ref{prop:phi411}.

\begin{proposition} \label{prop: phi411 representable}
Let $\Gamma$ be a congruence subgroup of level $N$ such that $\mathcal{X}_\Gamma\cong\mathbb{P}^1$ over $\Spec\mathbb{Z}[1/6N]$. Let $\phi_\Gamma:\mathcal{P}(1,1)\rightarrow\mathcal{P}(4,6)$ be the morphism corresponding to forgetting the level structure, and let $\psi_{\fp}$ be the mod $\fp$ reduction map.
Let $z \in \cP(4, 6)(\bF_q)$.
Suppose that the projective local condition $\Omega_{\fp}$ at a single prime $\fp \not\in S_{\phi_{\Gamma}}$ is given by
\begin{align*}
        \Omega_{\fp, z} \vcentcolon= \lcrc{y \in \cP(4, 6)(K_{\fp}) : \psi_{\fp}(y) = z }.
    \end{align*}
Then
\begin{align*}
    &\#\lcrc{y \in \phi_{\Gamma}(\cP(1,1)(K)) : H_{(4, 6), K}(y) \leq B, i_{\fp}(y) \in \Omega_{\fp, z}} \\
    &= \kappa \cdot r(G) \cdot \frac{\# \phi_{\Gamma}^{-1}(z)}{q+1} \cdot B^{\frac{2}{e({\Gamma})}} + O\lbrb{ \lbrb{1 + q^{\frac{1}{d} - 1}} \lbrb{1 + \frac{1}{q}} B^{\frac{d-1}{de({\Gamma})}} \log B }.
\end{align*}
\end{proposition}
\begin{proof}
By Proposition \ref{prop: preimage degree},
\begin{align*} 
\begin{aligned}
    & \frac{1}{r(G)}\#\lcrc{y \in \phi_{\Gamma}(\cP(1,1)(K)) : H_{w}(y) \leq B, i_{\fp}(y) \in \Omega_{\fp, z}} \\
    &= \#\lcrc{x \in \cP(1,1)(K) : H_{w}(\phi_{\Gamma}(x)) \leq B, i_{\fp}(x) \in \phi_{\Gamma}^{-1}( \Omega_{\fp, z})  } + O\lbrb{
    m_{\fp}\left(\Omega_{\fp, z}^{\aff} \cap \cO_{K, \fp}^{2}\right)
    \max_{\Sigma < \Gamma} \lbrb{ B^{\frac{2}{e({\Sigma})}}  } }.
\end{aligned} 
\end{align*}
For $\widetilde{z} \in \cP(1, 1)(\bF_q)$, we define
\begin{align*}
    \Omega_{\fp, \widetilde{z}} \vcentcolon = \lcrc{ x \in \cP(1,1)(K_{\fp}) : \psi_{\fp}(x) = \widetilde{z}}.
\end{align*}
By the diagram (\ref{eqn:comdiag}), we have $\phi_{\Gamma} \circ \psi_{\fp} = \psi_{\fp} \circ \phi_{\Gamma}$. So
\begin{align*}
    \phi_{\Gamma}^{-1}(\Omega_{\fp, z}) = \bigsqcup_{\widetilde{z} \in \phi_{\Gamma }^{-1}(z) }
    \lcrc{ x \in \cP(1, 1)(K_{\fp}) : \psi_{\fp}(x) = \widetilde{z} }
    = \bigsqcup_{\widetilde{z} \in \phi_{\Gamma}^{-1}(z) } \Omega_{\fp, \widetilde{z}}.
\end{align*}
Let $(z_{\fp, 0}, z_{\fp, 1}) \in \cO_{K, \fp}^2$ satisfying $\psi_{\fp}([z_{\fp, 0}, z_{\fp, 1}]) = \widetilde{z}$. 
Then by Lemma \ref{lem: local cond first exam}
\begin{align*}
    \Omega_{\fp, \widetilde{z}, 0}^{\aff} \vcentcolon= \prod_{j=0}^1 (z_{\fp, j} + \fp \cO_{K, \fp}), \quad \textrm{and} \quad 
    \Omega_{\fp, \widetilde{z}} = \bigsqcup_{k \geq 0} \bigsqcup_{\zeta \in \bF_q^\times} (\zeta*_{(1,1)}\Omega_{\fp, \widetilde{z}, k}^{\aff}).
\end{align*}
Since $\phi_{\Gamma}$ has finite defect by \cite[Corollary 6.5]{BN} and satisfies (\ref{eqn: widehat w condition}), we can apply Theorem \ref{thm: phi411 general} for each summand of
\begin{align*}
&\#\lcrc{x \in \cP(1,1)(K) : H_{w}(\phi_{\Gamma}(x)) \leq B, i_{\fp}(x) \in \phi_{\Gamma}^{-1}(\Omega_{\fp, z }) } \\
&= \sum_{\widetilde{z} \in \phi_{\Gamma}^{-1}(z)}
\#\lcrc{x \in \cP(1,1)(K) : H_{w}(\phi_{\Gamma}(x)) \leq B, i_{\fp}(x) \in \Omega_{\fp, \widetilde{z} }}.
\end{align*}
Since the constant $\kappa$ does not depend on the local condition, the main term is
\begin{align*}
    \kappa \cdot r(G) \cdot   B^{\frac{2}{e(\Gamma)}} \cdot \sum_{\widetilde{z} \in \phi_{\Gamma}^{-1}(z)} m_{\fp}\left(\Omega_{\fp, \widetilde{z}}^{\aff} \cap \cO_{K, \fp}^{2}\right).
\end{align*}
Since
\begin{align*}
    m_{\fp}(\Omega_{\fp, \widetilde{z}}^{\aff} \cap  \cO_{K, \fp}^{2})
    = (q-1) \sum_{k=0}^{\infty} m_{\fp}(\Omega_{\fp, \widetilde{z}, k}^{\aff} \cap  \cO_{K, \fp}^{2})
    = (q-1)\sum_{k=1}^{\infty} \frac{1}{q^{2k}} = \frac{1}{q+1},
\end{align*}
this is exactly the main term in the statement.
The additional error term from Theorem \ref{thm: phi411 general} is
\begin{align*}
O\lbrb{
     \lbrb{1 + \frac{1}{q-1}}  B^{\frac{d|u| - u_{\min}}{d e({\Gamma})}}
    \lbrb{1 + q^{u_{\max} + \frac{u_{\min}}{d} - |u| } } \log B }
=O\lbrb{ \lbrb{1 + q^{\frac{1}{d} - 1}} \lbrb{1 + \frac{1}{q}} B^{\frac{d-1}{de({\Gamma})}} \log B }.
\end{align*}
By definition,
\begin{align*}
e({\Gamma}) = \frac{1}{24}[\SL_2(\bZ) : \Gamma].
\end{align*}
So if $\Sigma < \Gamma$, then  $2e({\Gamma}) < e({\Sigma})$.
Hence,
\begin{align*}
    \frac{2}{e({\Sigma})} < \frac{1}{e({\Gamma})} \leq \lbrb{2 - \frac{1}{d}}\frac{1}{e({\Gamma})} = \frac{2d-1}{de({\Gamma})}.
\end{align*}
Therefore,
\begin{align*}
    \max_{\Sigma < \Gamma} \lbrb{ B^{\frac{2}{e({\Sigma})}}  } \ll 
    B^{\frac{d-1}{de({\Gamma})}} \log B, \text{ and } 
     m_{\fp}\left(\Omega_{\fp, z}^{\aff} \cap \cO_{K, \fp}^{2}\right) \ll \lbrb{1 + q^{\frac{1}{d} - 1}} \lbrb{1 + \frac{1}{q}}.
\end{align*}
So the error term part is also deduced.
\end{proof}

\end{document}